\documentclass[12pt, letterpaper]{amsart} 
\usepackage{amsmath}         
\usepackage{amsfonts} 
\usepackage{graphicx} 
\usepackage{mathrsfs}
\usepackage{float}
\usepackage{bbm}
\usepackage[normalem]{ulem}
\usepackage[numbers]{natbib}
\usepackage[utf8]{inputenc}    	
\usepackage{amsmath,amstext,amsthm,amssymb,amsxtra}
\usepackage[top=1.5in, bottom=1.5in, left=1.25in, right=1.25in]    {geometry}
\numberwithin{equation}{section}
\usepackage{tikz}
\usetikzlibrary{arrows}
\usepackage{mathtools}
\mathtoolsset{showonlyrefs,showmanualtags} 
\usepackage{amssymb}

\newtheorem{theorem}{Theorem}[section]
\newtheorem{lemma}[theorem]{Lemma}
\newtheorem{proposition}[theorem]{Proposition}
\newtheorem{corollary}[theorem]{Corollary} 
\newtheorem{algorithm}[theorem]{Algorithm}

\theoremstyle{definition}
\newtheorem{definition}[theorem]{Definition} 
\newtheorem{exmp}[theorem]{Example}
\numberwithin{equation}{theorem}

\newcommand{\cone}{\mathrm{cone}}

\title[Equidistant Circular Split Networks]{Equidistant Circular Split Networks}
\author{Bryson Kagy and Seth Sullivant}

\begin{document}

\maketitle

\begin{abstract}
    Phylogenetic networks are generalizations of trees that allow for the modeling of
    non-tree like evolutionary processes.
    Split networks give a useful way to construct networks with intuitive distance structures
    induced from the associated split graph.
    We explore the polyhedral geometry of distance matrices built from
    circular split systems which have the added property of being equidistant.
    We give a characterization of the facet defining inequalities and the extreme rays
    of the cone of distances that arise from an equidistant network associated with any circular 
    split network.  We also explain a connection to the Chan-Robbins-Yuen polytope
    from geometric combinatorics.
\end{abstract}

%%%%%%%%%%%%%%%%%%%%%%%%%%%%%%%%%%%%%%%%%%%%%%%%%%%%%%%%%%%% Phylogenetic trees record the evolutionary relationship between species, however in practice some cycle edges can be introduced through processes like reticulation events and hybridization. Thus phylogenetic networks are studied which allow got some limited non-tree edges. One such natural type of phylogenetic networks is the circular split network  Split networks give a useful way to construct networks with intuitive distance structures    induced from the associated split graph.   We explore the polyhedral geometry of distance matrices built from    circular split systems which have the added property of being equidistant.   We give a characterization of the facet defining inequalities and the extreme rays    of the cone of distances that arises from an equidistant network associated to any circular    split network.  We also explain a connection to the Chan-Robbins-Yuen polytope   from geometric combinatorics.
%%%%%%%%%%%%%%%%%%%%%%%%%%%%%%%%%%%%%%%%%%%%%%%%%%%%%%%%%%%%
%%%%%%%%%%%%%%%%%%%%%%%%%%%%%%%%%%%%%%%%%%%%%%%%%%%%%%%%%%%%
%%%%%%%%%%%%%%%%%%%%%%%%%%%%%%%%%%%%%%%%%%%%%%%%%%%%%%%%%%%%

\section{Introduction}

Phylogenetics concerns uncovering evolutionary relationships between collections of species.
Traditionally, these relationships are represented by trees.
The combinatorics of rooted tree structures, and distances derived from trees are a 
staple of phylogenetic inference and at the heart of much of the mathematics of evolutionary
biology.  This is the perspective in classic books like \cite{Felsenstein2003, phy}.
However, the presence of evolutionary processes that produce non-tree-like structures
among species have been realized to play an important role  in evolution.

Non-tree-like evolutionary processes include horizontal gene transfer, 
hybridization, and introgression.   
It is desirable to have
phylogenetic structures that can encode these types of more complex, non-tree-like relationships.
This has led to the creation and study of phylogenetic networks as a tool
for phylogenetic inference, where the network structure can encode different types of non-tree-like relationships.

There are a few different approaches to making phylogenetic networks (see \cite[Ch 10]{Steel}).
Some of the choices of which network structure to use are based on which modeling paradigm is 
employed, others are based on which inference techniques are being used, and yet others are just
based on whether the mathematics is interesting. See \cite{Network2022} for a discussion of the different classes of networks.
In this paper, we study the mathematics of distances based on phylogenetic networks.
For this approach, one of the  most natural network structures to study is based on 
split networks.  This is because split networks are naturally tied to cut-semimetrics and
the cut cone, so they naturally fit into the framework of those well-studied objects \cite{Deza2010}.

The most studied family of non-tree-like split networks is the family of circular
split networks because they are the type of split network produced by the NeighborNet algorithm
\cite{Bryant2004}, a widely used algorithm, cited by 2318 papers on Google Scholar as of August 2024.
The geometry of metrics associated with circular split systems is well-studied.
The set of metrics compatible with a particular circular ordering are the Kalmanson metrics.  
Kalmanson metrics associated with the standard ordering $1, \ldots, n$
are metrics $\delta$ on $[n]$ that satisfy the inequalities
\[
\delta(i,j) + \delta(k,l) \leq  \delta(i,k) + \delta(j,l)  \quad \mbox{ and } \quad 
\delta(i,l) + \delta(j,k) \leq  \delta(i,k) + \delta(j,l)   
\]
whenever $i < j < k < l$.    This condition is also famous in combinatorial optimization because
the traveling salesman problem can be solved in polynomial time if distance constraints for the problem
come from a Kalmanson metric \cite{Kalmanson1975}.  There are  a number of papers that explore
the connection between Kalmanson metrics and phylogenetics \cite{main, LevyPachter2011}.

In our paper, we explore a variation of metrics associated to circular split systems where we
add the extra condition that the metric is equidistant.  The equidistant condition means that the
network has a special root vertex, and each of the vertices $i \in [n]$ is the same distance
from the root. In the context of biology, an equidistant circular split network would represent a situation where there is a set of species that evolve at similar rates and are hybridizing with each other. The cone of equidistant circular split networks is the space of all possible distance functions that make an equidistant circular split network on those species. Thus if there is a set of closely related species that are suspected of hybridization, one could see, using the facet description, if the vector of pairwise distances are inside or near the cone of equidistant circular split networks. In addition, the study of the cone of equidistant circular split networks is of mathematical interest as a step towards creating a network version of UPGMA \cite{Michener57}. UPGMA is a popular algorithm for creating phylogenetic trees from pairwise distances which always produces equidistant trees. In the same way that NeighborNet is a generalization of the neighbor joining algorithm, understanding equidistant circular split networks will help in creating a similar algorithm generalization.

 Our main results are a characterization of the inequalities and the extreme rays that
define the cone of equidistant circular split networks.  
The resulting inequality system that arises is a kind of restriction of the Kalmanson conditions
to take into account the equidistant condition shown in Theorem \ref{facets}. On the other hand, while the general Kalmanson cone has only $\binom{n}{2}$ extreme rays, the
general cone of equidistant circular metrics has $2^{n-1}-1$ extreme rays shown, as in Theorem \ref{thm:extremerays}.
In addition, we show that every face of the cone  circular equidistant networks
also corresponds to metrics for subnetworks, and characterize their inequalities and extreme rays.
Finally, we show that the cone of equidistant circular split networks is closely related to the
Chan-Robbins-Yuen polytope from geometric combinatorics.

%%%%%%%%%%%%%%%%%%%%%%%%%%%%%%%%%%%%%%%%%%%%%%%%%%%%%%%%%%%%
%%%%%%%%%%%%%%%%%%%%%%%%%%%%%%%%%%%%%%%%%%%%%%%%%%%%%%%%%%%%
%%%%%%%%%%%%%%%%%%%%%%%%%%%%%%%%%%%%%%%%%%%%%%%%%%%%%%%%%%%%
%%%%%%%%%%%%%%%%%%%%%%%%%%%%%%%%%%%%%%%%%%%%%%%%%%%%%%%%%%%%

\section{Split Networks}

We introduce the notion of split system and split networks. 
Split networks generalize phylogenetic trees by allowing 
for some limited cycles in the graph structure. 
This is inspired by hybridization and reticulation events in biology, which introduce cycles into phylogenetic trees. 
The material in this section is standard in the literature, and more background on 
split systems and split networks can be found in \cite{Steel}.

\begin{definition}
 Let $X$ be a set of  labels with $|X|=n$. A \textit{split} $A|B$ is a partition of $X$  into two nonempty sets. 
 A split is a \textit{trivial split} if one part of the partition has cardinality one. 
 A set of splits is called a \textit{split system}.
\end{definition}
Split graphs and split networks are visual tools used to represent a split system.

\begin{definition}
Let $G= (V,E)$ be a connected bipartite finite graph, $K$ a finite set of labels, and $s$ a surjective map $s:E\to K$. 
The pairing $(G,s)$ is a \emph{split graph} if for all $u$, $v\in V$ 
and for each shortest path $p$ between $u$ and $v$, $s$ maps the edges on $p$ 
one-to-one to a  $S(u,v)\subseteq K$ with $S(u,v)$ the same for all such $p$. 
\end{definition}

One feature of a split graph is that for every $k\in K$, removing all the edges with label $k$ breaks the graph into 
two components, as  seen in  Proposition \ref{2comp}.

\begin{proposition} \cite{Huson_Rupp_Scornavacca_2010}\label{2comp}
Let $(G,s)$ be a split graph with $G= (V,E)$ and $s:E\to K$. 
For any $k\in K$, let $E_k   =  \{  e \in E :  s(e) = k \}. $  
Then the graph $(V, E-E_k)$ has exactly two connected components for every $k\in K$.
\end{proposition}

Now the notion of split graph can be combined with a split system 
to define the notion of a split graph representing a split system, which is called a split network.

\begin{definition}
Let $N$ be a split system on  $X$. Let $(G,s)$ be a split graph with $G= (V,E)$ and $s:E\to K$. 
Let $f:X\to V$ be a map such that for all $A|B\in N$ there exists a $k({A|B}) \in K$ such that $f(A)$ and $f(B)$ 
are exactly  in the two connected components of $(V, E-E_{k({A|B})})$.   Furthermore, assume that 
each element in $K$ corresponds to an element in $N$.
Then $(G,s,f)$ is a \emph{split network} that represents $N$.
\end{definition}
\begin{exmp}\label{split graph}
Consider the bipartite graph $G$ drawn in Figure \ref{splitgraphfig}.
    Let  $s:E_G\to \{\alpha,\beta,\gamma,\delta,\epsilon,\zeta,\eta,\theta,\iota\}$ such that
    \begin{align*}
     &s(\{c,d\})=s(\{f,g\})=s(\{h,m\})=s(\{d,i\})=\alpha,\\
    &s(\{c,g\})=s(\{d,f\})=s(\{k,m\})=\beta,\\
    &s(\{g,m\})=s(\{f,h\})=s(\{i,k\})=\gamma,\\
    &s(\{a,c\})=\delta,
     s(\{b,d\})=\epsilon,
    s(\{e,d\})=\zeta,
       s(\{g,j\})=\eta,
        s(\{k,l\})=\theta,
         s(\{m,n\})=\iota.\\
    \end{align*}
    Then the pairing $(G,s)$ is a split graph. Additionally, consider the split system \[
N = \{1456|23,1234|56,1236|45\} \cup \{ i | [6] \setminus i : i \in [6]  \}.
\]
    Then if $f$ is the map:
     \begin{align*}
     &f(a)=1, f(b)=2,f(e)=3,f(l)=4,f(n)=5, f(j)=6,   
    \end{align*}
$(G,s,f)$ is a split network representing $N$.
\end{exmp}

 \begin{figure}[t]

\centering
\resizebox{150pt}{100pt}{

\tikzset{every picture/.style={line width=0.75pt}} %set default line width to 0.75pt        

\begin{tikzpicture}[x=0.75pt,y=0.75pt,yscale=-1,xscale=1]

%Straight Lines [id:da9097092915712146] 
\draw    (420,185) -- (449.33,154) ;
%Shape: Circle [id:dp3974326977147489] 
\draw  [fill={rgb, 255:red, 0; green, 0; blue, 0 }  ,fill opacity=1 ] (417.5,185) .. controls (417.5,183.62) and (418.62,182.5) .. (420,182.5) .. controls (421.38,182.5) and (422.5,183.62) .. (422.5,185) .. controls (422.5,186.38) and (421.38,187.5) .. (420,187.5) .. controls (418.62,187.5) and (417.5,186.38) .. (417.5,185) -- cycle ;
%Shape: Circle [id:dp8949789883607544] 
\draw  [fill={rgb, 255:red, 0; green, 0; blue, 0 }  ,fill opacity=1 ] (446.83,154) .. controls (446.83,152.62) and (447.95,151.5) .. (449.33,151.5) .. controls (450.71,151.5) and (451.83,152.62) .. (451.83,154) .. controls (451.83,155.38) and (450.71,156.5) .. (449.33,156.5) .. controls (447.95,156.5) and (446.83,155.38) .. (446.83,154) -- cycle ;
%Straight Lines [id:da08335165081238993] 
\draw [color={rgb, 255:red, 208; green, 2; blue, 27 }  ,draw opacity=1 ]   (500,185) -- (555,165) ;
%Straight Lines [id:da4599222245267083] 
\draw [color={rgb, 255:red, 245; green, 166; blue, 35 }  ,draw opacity=1 ]   (580,115) -- (555,165) ;
%Straight Lines [id:da6063767583409294] 
\draw [color={rgb, 255:red, 245; green, 166; blue, 35 }  ,draw opacity=1 ]   (449.33,154) -- (480,95) ;
%Straight Lines [id:da751984648940867] 
\draw [color={rgb, 255:red, 74; green, 144; blue, 226 }  ,draw opacity=1 ]   (449.33,154) -- (455,70) ;
%Straight Lines [id:da8829667984451814] 
\draw [color={rgb, 255:red, 208; green, 2; blue, 27 }  ,draw opacity=1 ]   (480,95) -- (510,105) ;
%Straight Lines [id:da49173596691922006] 
\draw [color={rgb, 255:red, 74; green, 144; blue, 226 }  ,draw opacity=1 ]   (510,105) -- (580,115) ;
%Straight Lines [id:da5033644329133258] 
\draw    (420,45) -- (455,70) ;
%Straight Lines [id:da2453547911620395] 
\draw [color={rgb, 255:red, 74; green, 144; blue, 226 }  ,draw opacity=1 ]   (480,95) -- (515,65) ;
%Straight Lines [id:da6030757640846085] 
\draw [color={rgb, 255:red, 208; green, 2; blue, 27 }  ,draw opacity=1 ]   (580,115) -- (515,65) ;
%Straight Lines [id:da0031443071246821663] 
\draw [color={rgb, 255:red, 245; green, 166; blue, 35 }  ,draw opacity=1 ]   (455,70) -- (486.07,67.41) -- (515,65) ;
%Straight Lines [id:da9622494178402252] 
\draw    (515,65) -- (540,35) ;
%Straight Lines [id:da09580017984337608] 
\draw    (580,115) -- (610,115) ;
%Shape: Circle [id:dp5936794242114236] 
\draw  [fill={rgb, 255:red, 0; green, 0; blue, 0 }  ,fill opacity=1 ] (552.5,165) .. controls (552.5,163.62) and (553.62,162.5) .. (555,162.5) .. controls (556.38,162.5) and (557.5,163.62) .. (557.5,165) .. controls (557.5,166.38) and (556.38,167.5) .. (555,167.5) .. controls (553.62,167.5) and (552.5,166.38) .. (552.5,165) -- cycle ;
%Shape: Circle [id:dp4573675463405762] 
\draw  [fill={rgb, 255:red, 0; green, 0; blue, 0 }  ,fill opacity=1 ] (497.5,185) .. controls (497.5,183.62) and (498.62,182.5) .. (500,182.5) .. controls (501.38,182.5) and (502.5,183.62) .. (502.5,185) .. controls (502.5,186.38) and (501.38,187.5) .. (500,187.5) .. controls (498.62,187.5) and (497.5,186.38) .. (497.5,185) -- cycle ;
%Shape: Circle [id:dp5573346819932286] 
\draw  [fill={rgb, 255:red, 0; green, 0; blue, 0 }  ,fill opacity=1 ] (477.5,95) .. controls (477.5,93.62) and (478.62,92.5) .. (480,92.5) .. controls (481.38,92.5) and (482.5,93.62) .. (482.5,95) .. controls (482.5,96.38) and (481.38,97.5) .. (480,97.5) .. controls (478.62,97.5) and (477.5,96.38) .. (477.5,95) -- cycle ;
%Shape: Circle [id:dp12869765327018734] 
\draw  [fill={rgb, 255:red, 0; green, 0; blue, 0 }  ,fill opacity=1 ] (452.5,70) .. controls (452.5,68.62) and (453.62,67.5) .. (455,67.5) .. controls (456.38,67.5) and (457.5,68.62) .. (457.5,70) .. controls (457.5,71.38) and (456.38,72.5) .. (455,72.5) .. controls (453.62,72.5) and (452.5,71.38) .. (452.5,70) -- cycle ;
%Shape: Circle [id:dp1478266559412562] 
\draw  [fill={rgb, 255:red, 0; green, 0; blue, 0 }  ,fill opacity=1 ] (577.5,115) .. controls (577.5,113.62) and (578.62,112.5) .. (580,112.5) .. controls (581.38,112.5) and (582.5,113.62) .. (582.5,115) .. controls (582.5,116.38) and (581.38,117.5) .. (580,117.5) .. controls (578.62,117.5) and (577.5,116.38) .. (577.5,115) -- cycle ;
%Shape: Circle [id:dp11955533502751714] 
\draw  [fill={rgb, 255:red, 0; green, 0; blue, 0 }  ,fill opacity=1 ] (507.5,105) .. controls (507.5,103.62) and (508.62,102.5) .. (510,102.5) .. controls (511.38,102.5) and (512.5,103.62) .. (512.5,105) .. controls (512.5,106.38) and (511.38,107.5) .. (510,107.5) .. controls (508.62,107.5) and (507.5,106.38) .. (507.5,105) -- cycle ;
%Shape: Circle [id:dp49414207738736415] 
\draw  [fill={rgb, 255:red, 0; green, 0; blue, 0 }  ,fill opacity=1 ] (417.5,45) .. controls (417.5,43.62) and (418.62,42.5) .. (420,42.5) .. controls (421.38,42.5) and (422.5,43.62) .. (422.5,45) .. controls (422.5,46.38) and (421.38,47.5) .. (420,47.5) .. controls (418.62,47.5) and (417.5,46.38) .. (417.5,45) -- cycle ;
%Shape: Circle [id:dp27823350571384964] 
\draw  [fill={rgb, 255:red, 0; green, 0; blue, 0 }  ,fill opacity=1 ] (605,115) .. controls (605,113.62) and (606.12,112.5) .. (607.5,112.5) .. controls (608.88,112.5) and (610,113.62) .. (610,115) .. controls (610,116.38) and (608.88,117.5) .. (607.5,117.5) .. controls (606.12,117.5) and (605,116.38) .. (605,115) -- cycle ;
%Shape: Circle [id:dp3131039617467535] 
\draw  [fill={rgb, 255:red, 0; green, 0; blue, 0 }  ,fill opacity=1 ] (537.5,35) .. controls (537.5,33.62) and (538.62,32.5) .. (540,32.5) .. controls (541.38,32.5) and (542.5,33.62) .. (542.5,35) .. controls (542.5,36.38) and (541.38,37.5) .. (540,37.5) .. controls (538.62,37.5) and (537.5,36.38) .. (537.5,35) -- cycle ;
%Shape: Circle [id:dp506382653725576] 
\draw  [fill={rgb, 255:red, 0; green, 0; blue, 0 }  ,fill opacity=1 ] (512.5,65) .. controls (512.5,63.62) and (513.62,62.5) .. (515,62.5) .. controls (516.38,62.5) and (517.5,63.62) .. (517.5,65) .. controls (517.5,66.38) and (516.38,67.5) .. (515,67.5) .. controls (513.62,67.5) and (512.5,66.38) .. (512.5,65) -- cycle ;
%Straight Lines [id:da20507394995338157] 
\draw [color={rgb, 255:red, 74; green, 144; blue, 226 }  ,draw opacity=1 ]   (449.33,154) -- (500,185) ;
%Straight Lines [id:da8567753441380812] 
\draw    (450,200) -- (449.33,154) ;
%Shape: Circle [id:dp07930888830417615] 
\draw  [fill={rgb, 255:red, 0; green, 0; blue, 0 }  ,fill opacity=1 ] (447.5,200) .. controls (447.5,198.62) and (448.62,197.5) .. (450,197.5) .. controls (451.38,197.5) and (452.5,198.62) .. (452.5,200) .. controls (452.5,201.38) and (451.38,202.5) .. (450,202.5) .. controls (448.62,202.5) and (447.5,201.38) .. (447.5,200) -- cycle ;
%Straight Lines [id:da3568227869502836] 
\draw    (555,165) -- (555,205) ;
%Shape: Circle [id:dp6611557067264136] 
\draw  [fill={rgb, 255:red, 0; green, 0; blue, 0 }  ,fill opacity=1 ] (552.5,205) .. controls (552.5,203.62) and (553.62,202.5) .. (555,202.5) .. controls (556.38,202.5) and (557.5,203.62) .. (557.5,205) .. controls (557.5,206.38) and (556.38,207.5) .. (555,207.5) .. controls (553.62,207.5) and (552.5,206.38) .. (552.5,205) -- cycle ;

% Text Node
\draw (416,192) node [anchor=north west][inner sep=0.75pt]  [font=\tiny] [align=left] {b};
% Text Node
\draw (447.47,140) node [anchor=north west][inner sep=0.75pt]   [align=left] {$ $};
% Text Node
\draw (457,73) node [anchor=north west][inner sep=0.75pt]  [font=\tiny] [align=left] {c};
% Text Node
\draw (437.75,143.75) node [anchor=north west][inner sep=0.75pt]  [font=\tiny] [align=left] {d};
% Text Node
\draw (480.75,102) node [anchor=north west][inner sep=0.75pt]  [font=\tiny] [align=left] {f};
% Text Node
\draw (511,72) node [anchor=north west][inner sep=0.75pt]  [font=\tiny] [align=left] {g};
% Text Node
\draw (508.5,110.25) node [anchor=north west][inner sep=0.75pt]  [font=\tiny] [align=left] {h};
% Text Node
\draw (502.25,189.25) node [anchor=north west][inner sep=0.75pt]  [font=\tiny] [align=left] {i};
% Text Node
\draw (541,42) node [anchor=north west][inner sep=0.75pt]  [font=\tiny] [align=left] {j};
% Text Node
\draw (559,212) node [anchor=north west][inner sep=0.75pt]  [font=\tiny] [align=left] {l};
% Text Node
\draw (419,52) node [anchor=north west][inner sep=0.75pt]  [font=\tiny] [align=left] {a};
% Text Node
\draw (580,119.75) node [anchor=north west][inner sep=0.75pt]  [font=\tiny] [align=left] {m};
% Text Node
\draw (446,207) node [anchor=north west][inner sep=0.75pt]  [font=\tiny] [align=left] {e};
% Text Node
\draw (560,170.25) node [anchor=north west][inner sep=0.75pt]  [font=\tiny] [align=left] {k};
% Text Node
\draw (608,119.25) node [anchor=north west][inner sep=0.75pt]  [font=\tiny] [align=left] {n};

\end{tikzpicture}

}
\caption{An example of a split graph whose corresponding functions are color-coded and described in Example \ref{split graph}. Here $\alpha$ is blue, $\beta$ is orange, $\gamma$ is red, and all other function values are black}\label{splitgraphfig}

\end{figure}
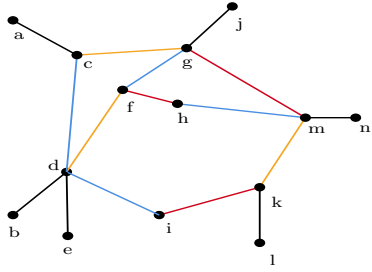

A split system is represented by a split network if the split network has some edge or set of edges that realize every split in the split system. The following definition and theorem exactly characterize when a split system is represented by a tree.

\begin{definition} \label{compatible}
A split system $N$ is \emph{pairwise compatible} if for every pair of splits $A_1| B_1$, $A_2| B_2 \in N$  
at least one of the following sets is empty:
\[
A_1 \cap A_2, \; \; A_1 \cap B_2, \; \; A_2 \cap B_1, \; \; B_1 \cap B_2.
\]
\end{definition}

\begin{definition}
    Let $(G,s, f)$ be a split network on $X$.  Let 
$\Sigma(G)$ be all of the splits of $X$ that are induced by edge classes of $G$. 
\end{definition}

\begin{theorem}{(Splits Equivalence Theorem)}\label{equiv}
Let $N$ be a split system on $X$. 
Then there exists an $X$-tree $T$ with $\Sigma(T)=N$ if and only if the splits in $N$ are pairwise compatible. 
Furthermore, the tree $T$ is uniquely determined.
\end{theorem}

For a detailed proof of Theorem \ref{equiv} see \cite{phy}. 
In general  split systems need not be pairwise compatible. 
Thus, general split systems extend trees to the case of non-compatible splits.
In a general split network,  each split is represented by a set of parallel 
edges that disconnect the graph according to that
partition.

\begin{exmp}\label{treeex}
Consider the split system $N$ on 6 leaves  
\[
N = \{12|3456,1265|34,1234|56\} \cup \{ i | [6] \setminus i : i \in [6]  \}.
\]
This set of  splits is pairwise compatible. 
Thus by the Split Equivalence Theorem, there exists a unique tree with $\Sigma(T)=N$, shown in Figure \ref{uniquetree},  
\end{exmp}

 \begin{figure}[t]

\centering

\tikzset{every picture/.style={line width=0.75pt}} %set default line width to 0.75pt        

\begin{tikzpicture}[x=0.75pt,y=0.75pt,yscale=-1,xscale=1]
%uncomment if require: \path (0,301); %set diagram left start at 0, and has height of 301

%Straight Lines [id:da6655526880571436] 
\draw    (235.18,101.27) -- (202.15,111.4) ;
%Straight Lines [id:da3981563405346358] 
\draw    (202.89,158.14) -- (231.22,181.31) ;
%Straight Lines [id:da33792240594318335] 
\draw    (202.89,158.14) -- (202.89,192.89) ;
%Shape: Ellipse [id:dp7816278913366939] 
\draw  [fill={rgb, 255:red, 18; green, 17; blue, 17 }  ,fill opacity=1 ] (200.53,192.89) .. controls (200.53,191.29) and (201.59,189.99) .. (202.89,189.99) .. controls (204.19,189.99) and (205.25,191.29) .. (205.25,192.89) .. controls (205.25,194.49) and (204.19,195.78) .. (202.89,195.78) .. controls (201.59,195.78) and (200.53,194.49) .. (200.53,192.89) -- cycle ;
%Shape: Ellipse [id:dp3306691396212653] 
\draw  [fill={rgb, 255:red, 18; green, 17; blue, 17 }  ,fill opacity=1 ] (228.86,181.31) .. controls (228.86,179.71) and (229.91,178.41) .. (231.22,178.41) .. controls (232.52,178.41) and (233.58,179.71) .. (233.58,181.31) .. controls (233.58,182.91) and (232.52,184.2) .. (231.22,184.2) .. controls (229.91,184.2) and (228.86,182.91) .. (228.86,181.31) -- cycle ;
%Shape: Ellipse [id:dp9654667878474112] 
\draw  [fill={rgb, 255:red, 18; green, 17; blue, 17 }  ,fill opacity=1 ] (232.95,100.3) .. controls (233.39,98.79) and (234.74,98.01) .. (235.97,98.55) .. controls (237.2,99.08) and (237.84,100.74) .. (237.4,102.25) .. controls (236.96,103.75) and (235.61,104.54) .. (234.38,104) .. controls (233.15,103.46) and (232.51,101.81) .. (232.95,100.3) -- cycle ;
%Straight Lines [id:da9141394242382443] 
\draw    (153.61,133.04) -- (125.63,156.1) ;
%Shape: Ellipse [id:dp628557055149582] 
\draw  [fill={rgb, 255:red, 18; green, 17; blue, 17 }  ,fill opacity=1 ] (127.49,155.81) .. controls (127.49,154.22) and (126.45,152.93) .. (125.16,152.93) .. controls (123.88,152.93) and (122.83,154.22) .. (122.83,155.81) .. controls (122.83,157.4) and (123.88,158.69) .. (125.16,158.69) .. controls (126.45,158.69) and (127.49,157.4) .. (127.49,155.81) -- cycle ;
%Straight Lines [id:da8015406764424691] 
\draw    (153.61,133.04) -- (127.47,112.45) ;
%Shape: Ellipse [id:dp11285059956418508] 
\draw  [fill={rgb, 255:red, 18; green, 17; blue, 17 }  ,fill opacity=1 ] (129.64,112.45) .. controls (129.64,113.87) and (128.67,115.02) .. (127.47,115.02) .. controls (126.26,115.02) and (125.29,113.87) .. (125.29,112.45) .. controls (125.29,111.03) and (126.26,109.87) .. (127.47,109.87) .. controls (128.67,109.87) and (129.64,111.03) .. (129.64,112.45) -- cycle ;
%Straight Lines [id:da7429283186264104] 
\draw    (202.15,111.4) -- (211.67,78.68) ;
%Shape: Ellipse [id:dp42986080648598257] 
\draw  [fill={rgb, 255:red, 18; green, 17; blue, 17 }  ,fill opacity=1 ] (213.73,79.58) .. controls (213.34,80.92) and (212.1,81.6) .. (210.97,81.1) .. controls (209.84,80.61) and (209.23,79.12) .. (209.62,77.78) .. controls (210.01,76.44) and (211.25,75.76) .. (212.38,76.26) .. controls (213.51,76.75) and (214.12,78.24) .. (213.73,79.58) -- cycle ;
%Straight Lines [id:da3269101993862211] 
\draw [color={rgb, 255:red, 208; green, 2; blue, 27 }  ,draw opacity=1 ][fill={rgb, 255:red, 208; green, 2; blue, 27 }  ,fill opacity=1 ]   (202.15,111.4) -- (184.01,134.4) ;
%Straight Lines [id:da5614263648213096] 
\draw [color={rgb, 255:red, 80; green, 227; blue, 194 }  ,draw opacity=1 ]   (184.01,134.4) -- (202.89,158.14) ;
%Straight Lines [id:da966099878474751] 
\draw [color={rgb, 255:red, 248; green, 231; blue, 28 }  ,draw opacity=1 ]   (153.61,133.04) -- (184.01,134.4) ;

% Text Node
\draw (211.57,66.36) node [anchor=north west][inner sep=0.75pt]  [font=\tiny,rotate=-1.32] [align=left] {1};
% Text Node
\draw (240.93,90.28) node [anchor=north west][inner sep=0.75pt]  [font=\tiny,rotate=-1.17] [align=left] {2};
% Text Node
\draw (235.19,185.26) node [anchor=north west][inner sep=0.75pt]  [font=\tiny] [align=left] {3};
% Text Node
\draw (199.54,200.6) node [anchor=north west][inner sep=0.75pt]  [font=\tiny] [align=left] {4};
% Text Node
\draw (116.75,160.83) node [anchor=north west][inner sep=0.75pt]  [font=\tiny] [align=left] {5};
% Text Node
\draw (119.11,103.4) node [anchor=north west][inner sep=0.75pt]  [font=\tiny] [align=left] {6};

\end{tikzpicture}

\caption{The unique tree representing the split system $N= \{12|3456,1265|34,1234|56\} \cup \{ i | [6] \setminus i : i \in [6]  \}$ 
in Example \ref{treeex}, which must exist by Theorem \ref{equiv}. }\label{uniquetree}

\end{figure}

A consequence of Theorem \ref{equiv} is that a split system $N$ will be pairwise 
compatible if and only if there is some tree $T$ such that $(T,s,f)$ represents $N$.

A generic split system will not necessarily have a split network that can be drawn as a planar graph.
However, adding the following circular condition to the split system guarantees that the split network graph is planar.
Aside from the planar nature of circular split systems, they also have the advantage of being easy to represent
and have nice mathematical properties \cite{BanDress,Bryant2004, main, LevyPachter2011}.

\begin{definition}
A split system $N$ with leaf labels $X$ is \textit{circular} with respect to some cyclic ordering $ (x_1,\dots, x_n)$
of $X$ if every split of $S$ is of the form 
\[
x_{i+1}, \dots , x_j | x_{j+1}, \dots , x_i
\] 
for some $i$ and $j$, where the indices are considered cyclically modulo $n$ (e.g. $x_{n+1} = x_1$). 
\end{definition}

The following algorithm constructs a split network that represents a circular split system $N$.
\begin{algorithm}{Circular Network Algorithm }\label{netalg}\citep{netalg}

Let $N$ be a split system with $n$ leaves and the corresponding trivial split for each leaf, i.e.  $i | [n] \setminus i : i \in [n]  \}\in N$.
\begin{enumerate}
\item Construct a star graph with $n$ leaves labeled $1,\dots, n$.
\item Let $i\dots j-1|j \dots n \dots i-1\in N$. Find a path $p$ from $i$ to $j-1$ that uses the least amount of edges. Let the size of the number of internal edges in $p$ be $g$.
\item Let $p=i e_0 u_1 e_1 \dots u_g e_g(j-1)$ where $e_k$ and $u_k$ for $k\in[g]$ are internal edges and vertices along $p$ respectively. Create a copy of $e_0 u_1 e_1 \dots u_g e_g$ which we will call $e_0' u_1' e_1' \dots u_g' e_g'$.
\item Assume that the edges of $u_k$ are $\{e_{k-1}, l_1, \dots, l_h, e_k, r_1, \dots, r_p\}$ where $ l_1 \dots l_h$ are edges that move closer to the leaves $i,\dots ,j-1$ and $r_1 \dots r_p$ are the edges that move closer to the leaves $j, \dots, n, \dots, i-1$.  Let $f_{k}$ be a new edge representing the split $i\dots j-1|j \dots n \dots i-1$. Then change the edges of $u_k$ to $\{e_{k-1}, l_1, \dots, l_h, e_k,f_k\}$ and the edges of $u_k'$ to \begin{equation}
\begin{cases} 
 (f_k, e^\prime_i, r_1, r_2,\ldots, r_p)  & \text{ if }k=1 \\ (e^\prime_{k-1}, f_k, e^\prime_k, r_1, r_2,\ldots, r_p) & \text{ if } 1<k<g \\
 (e^\prime_{k-1}, f_k, r_1, r_2,\ldots, r_p) & \text{ if }k=g\end{cases}.
\end{equation}
\item Repeat Steps $2$ through $4$ for every split in $N$.
\end{enumerate}
\end{algorithm}

\begin{figure}[t]
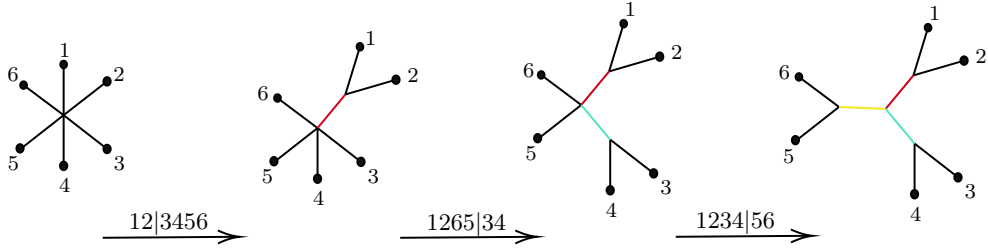


\centering

\tikzset{every picture/.style={line width=0.75pt}} %set default line width to 0.75pt        


\caption{The tree representing the split system $N= \{12|3456,1265|34,1234|56\} \cup \{ i | [6] \setminus i : i \in [6]  \}$ 
in Example \ref{treeex}. The process for constructing this tree is shown in Algorithm \ref{netalg}}\label{treepic}
\end{figure}

\begin{figure}[t]
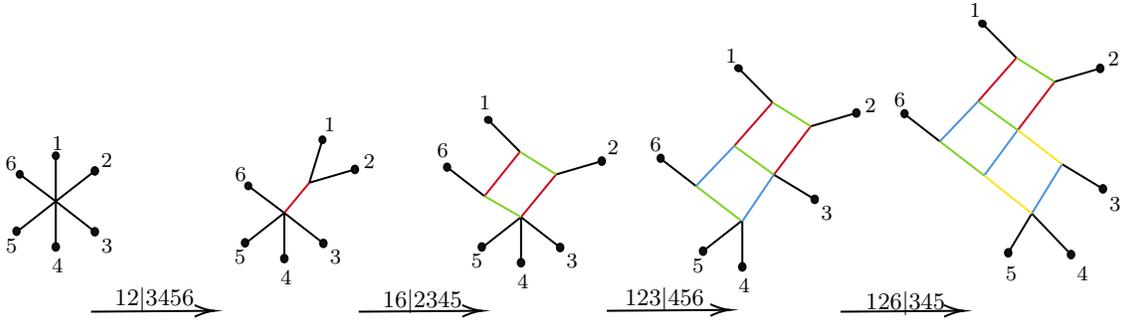


\centering

\tikzset{every picture/.style={line width=0.75pt}} %set default line width to 0.75pt        

%
\caption{A visualization of preforming Algorithm \ref{netalg} on the split system $N= \{12|3456,16|2345,123|456,126|345\}$ from Example \ref{exnet}.}\label{tree}

\end{figure}

Constructing a split network from a split system can also be viewed as starting with a star graph and “pulling” the two sides of each additional split in a different direction, splitting in half any edges as necessary.
\begin{exmp}\label{exnet}
Consider the split system $N$ on 6 leaves with splits 
\[
\{12|3456,16|2345,123|456,126|345\}  \cup \{ i | [6] \setminus i : i \in [6]  \} .
\]
These splits are not pairwise compatible. This can be seen with the splits $A_1|B_1=12|3456$ and $A_2|B_2=16|2345$ and $A_1\cap A_2=1$, $A_1\cap B_2=2$, $B_1\cap A_2=6$, $B_1\cap B_2=34$, none of which are empty.  See Figure \ref{tree} for a visualization of a split network representing $N$ using Algorithm \ref{netalg}.
\end{exmp}

In addition, see Figure \ref{treepic} for a visualization of a tree representing $N$ in Example \ref{treeex} using Algorithm \ref{netalg}.

\begin{figure}[t]
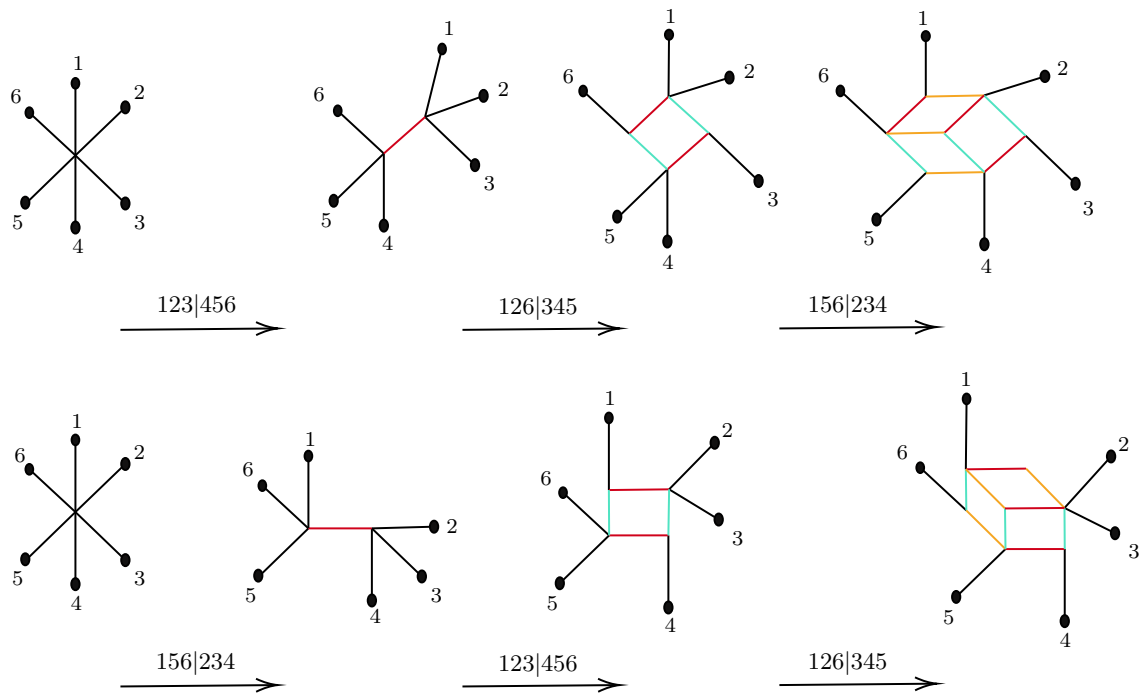


\centering

\tikzset{every picture/.style={line width=0.75pt}} %set default line width to 0.75pt        



\caption{Applying Algorithm \ref{netalg} to the split system $N=\{123|456,126|345,156|234\}$ from Example \ref{orderex}    
by applying the splits in two different orders, resulting in different graphs}\label{order}
\end{figure}

 For any particular drawing of a split network, the order in which the splits are drawn changes the resulting graph which can be seen in the following example.
\begin{exmp}\label{orderex}
Consider the split system $N$ on 6 leaves with splits 
\[ 
\{123|456,126|345,156|234\} \cup \{ i | [6] \setminus i : i \in [6]  \}.
\]
If Algorithm \ref{netalg} is performed by applying the nontrivial splits in the order \\
$(123|456,126|345,156|234) $, 
it will result in a different graph than if they are applied in the order $(156|234,123|456,126|345)$. This is visualized in Figure \ref{order}.
\end{exmp}

Thus, unlike the tree case, there is a choice to be made for the order of the splits in a circular split network to apply Algorithm \ref{netalg}. Rather than fixing one particular ordering, the choice made in this paper was to view circular split networks as simply a set of splits, visualizing them using their dual polygon representation. This choice maintains an independence of the order of the circular split networks.

\begin{definition}
Let $N$ be a circular split system with leaves labeled  $0,1,\dots, n$. 
The \textit{dual polygon representation} of $N$ is constructed in the following way: \\

Take an $n+1$-gon and label the edges sequentially clockwise with $0,1,\dots, n$. 
Label the vertices by the edge that is adjacent to it, clockwise.
Let $i\dots j-1|j \dots 0 n \dots i-1$ where $i<j\in [n]$, be a non-trivial split in $N$. 
Then, $i\dots j-1|j \dots 0 n \dots i-1$ is represented by the  diagonal of the $n+1$-gon that connects 
the vertex  $i$  to the vertex $j$. 
\end{definition}

With the labeling above, the edges of the $n+1$-gon are labeled by the numbers $0,1, \ldots, n$ 
in such a way that the diagonal corresponding to 
the nontrivial split 
$i\dots j-1|j \dots 0 n \dots i-1$
separates the edges into the two sets $\{ i, \dots,  j-1\} $ and $\{ j \dots 0 n \dots i-1 \}$.  
The trivial splits of $N$ correspond to the sides of the $n+1$-gon.

We now introduce the notion of rooted graphs  because the main focus in the rest of this paper will be on rooted split systems,
as that is key for the equidistant property.

 \begin{definition}
A graph $G$ is \emph{rooted} if one of its vertices has been specially designated as a root.
\end{definition}

In the context of biology, the root is the most recent common ancestor of all of the species in the network. In this paper, for any rooted split system, the root will be labeled $0$ and will be a 
leaf in the split network representing the rooted split system. 
The $n$ other leaves will have the labels $1$ through $n$. 
The root will be at the top of any picture, with its one leaf edge suppressed, 
and all of the rest of the leaves will be at the bottom. 
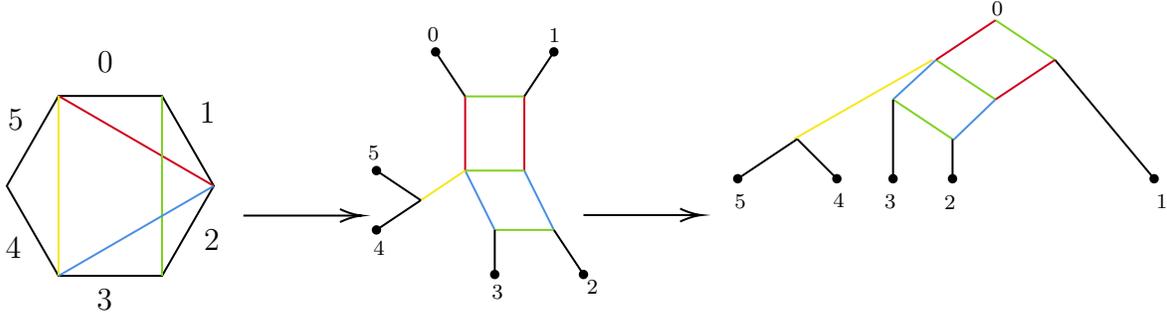
\begin{figure}[t]

\centering

\tikzset{every picture/.style={line width=0.75pt}} %set default line width to 0.75pt        

\begin{tikzpicture}[x=0.75pt,y=0.75pt,yscale=-1,xscale=1]

%Shape: Polygon [id:dp8469351831164718] 
\draw   (137.42,153.63) -- (111.32,199.03) -- (59.12,199.03) -- (33.02,153.63) -- (59.12,108.23) -- (111.32,108.23) -- cycle ;
%Straight Lines [id:da471271912130389] 
\draw [color={rgb, 255:red, 208; green, 2; blue, 27 }  ,draw opacity=1 ]   (59.12,108.23) -- (137.42,153.63) ;
%Straight Lines [id:da20064518446603263] 
\draw [color={rgb, 255:red, 126; green, 211; blue, 33 }  ,draw opacity=1 ]   (111.32,108.23) -- (111.32,199.03) ;
%Straight Lines [id:da7494473698717747] 
\draw [color={rgb, 255:red, 248; green, 231; blue, 28 }  ,draw opacity=1 ]   (59.12,108.23) -- (59.12,199.03) ;
%Straight Lines [id:da16387175050769542] 
\draw [color={rgb, 255:red, 74; green, 144; blue, 226 }  ,draw opacity=1 ]   (137.42,153.63) -- (59.12,199.03) ;
%Straight Lines [id:da23670206489163137] 
\draw    (152.34,168.61) -- (210,168.61) ;
\draw [shift={(212,168.61)}, rotate = 180] [color={rgb, 255:red, 0; green, 0; blue, 0 }  ][line width=0.75]    (10.93,-3.29) .. controls (6.95,-1.4) and (3.31,-0.3) .. (0,0) .. controls (3.31,0.3) and (6.95,1.4) .. (10.93,3.29)   ;
%Straight Lines [id:da7231393953360439] 
\draw [color={rgb, 255:red, 208; green, 2; blue, 27 }  ,draw opacity=1 ]   (264.2,145.9) -- (264.2,108.45) ;
%Straight Lines [id:da2074978742245217] 
\draw [color={rgb, 255:red, 208; green, 2; blue, 27 }  ,draw opacity=1 ]   (294.03,145.9) -- (294.03,108.45) ;
%Straight Lines [id:da0764743422267744] 
\draw [color={rgb, 255:red, 74; green, 144; blue, 226 }  ,draw opacity=1 ]   (279.12,175.85) -- (264.2,145.9) ;
%Straight Lines [id:da33963715189104016] 
\draw [color={rgb, 255:red, 74; green, 144; blue, 226 }  ,draw opacity=1 ]   (308.95,175.85) -- (294.03,145.9) ;
%Straight Lines [id:da5432496051145794] 
\draw [color={rgb, 255:red, 126; green, 211; blue, 33 }  ,draw opacity=1 ]   (264.2,108.45) -- (294.03,108.45) ;
%Straight Lines [id:da28734632521719594] 
\draw [color={rgb, 255:red, 126; green, 211; blue, 33 }  ,draw opacity=1 ]   (264.2,145.9) -- (294.03,145.9) ;
%Straight Lines [id:da7562414361815686] 
\draw [color={rgb, 255:red, 126; green, 211; blue, 33 }  ,draw opacity=1 ]   (279.12,175.85) -- (308.95,175.85) ;
%Straight Lines [id:da19060177372995968] 
\draw    (323.86,198.32) -- (308.95,175.85) ;
%Straight Lines [id:da45582904229495314] 
\draw    (308.95,85.98) -- (294.03,108.45) ;
%Straight Lines [id:da35049252571722267] 
\draw    (264.2,108.45) -- (249.29,85.98) ;
%Straight Lines [id:da3900628976465217] 
\draw    (279.12,198.32) -- (279.12,175.85) ;
%Straight Lines [id:da9267216304776855] 
\draw [color={rgb, 255:red, 248; green, 231; blue, 28 }  ,draw opacity=1 ]   (241.83,160.87) -- (264.2,145.9) ;
%Straight Lines [id:da15758503496556164] 
\draw    (241.83,160.87) -- (219.46,145.9) ;
%Straight Lines [id:da19599279009779336] 
\draw    (219.46,175.85) -- (241.83,160.87) ;
%Shape: Ellipse [id:dp4446089197622214] 
\draw  [fill={rgb, 255:red, 0; green, 0; blue, 0 }  ,fill opacity=1 ] (251.15,85.98) .. controls (251.15,84.95) and (250.32,84.11) .. (249.29,84.11) .. controls (248.26,84.11) and (247.42,84.95) .. (247.42,85.98) .. controls (247.42,87.02) and (248.26,87.86) .. (249.29,87.86) .. controls (250.32,87.86) and (251.15,87.02) .. (251.15,85.98) -- cycle ;
%Shape: Ellipse [id:dp6364475959013984] 
\draw  [fill={rgb, 255:red, 0; green, 0; blue, 0 }  ,fill opacity=1 ] (310.81,85.98) .. controls (310.81,84.95) and (309.98,84.11) .. (308.95,84.11) .. controls (307.92,84.11) and (307.08,84.95) .. (307.08,85.98) .. controls (307.08,87.02) and (307.92,87.86) .. (308.95,87.86) .. controls (309.98,87.86) and (310.81,87.02) .. (310.81,85.98) -- cycle ;
%Shape: Ellipse [id:dp7253620112204724] 
\draw  [fill={rgb, 255:red, 0; green, 0; blue, 0 }  ,fill opacity=1 ] (221.32,145.9) .. controls (221.32,144.86) and (220.49,144.02) .. (219.46,144.02) .. controls (218.43,144.02) and (217.59,144.86) .. (217.59,145.9) .. controls (217.59,146.93) and (218.43,147.77) .. (219.46,147.77) .. controls (220.49,147.77) and (221.32,146.93) .. (221.32,145.9) -- cycle ;
%Shape: Ellipse [id:dp5652004596732276] 
\draw  [fill={rgb, 255:red, 0; green, 0; blue, 0 }  ,fill opacity=1 ] (221.32,175.85) .. controls (221.32,174.82) and (220.49,173.98) .. (219.46,173.98) .. controls (218.43,173.98) and (217.59,174.82) .. (217.59,175.85) .. controls (217.59,176.89) and (218.43,177.72) .. (219.46,177.72) .. controls (220.49,177.72) and (221.32,176.89) .. (221.32,175.85) -- cycle ;
%Shape: Ellipse [id:dp44748268406245906] 
\draw  [fill={rgb, 255:red, 0; green, 0; blue, 0 }  ,fill opacity=1 ] (280.98,198.32) .. controls (280.98,197.28) and (280.15,196.45) .. (279.12,196.45) .. controls (278.09,196.45) and (277.25,197.28) .. (277.25,198.32) .. controls (277.25,199.35) and (278.09,200.19) .. (279.12,200.19) .. controls (280.15,200.19) and (280.98,199.35) .. (280.98,198.32) -- cycle ;
%Shape: Ellipse [id:dp9924126499548422] 
\draw  [fill={rgb, 255:red, 0; green, 0; blue, 0 }  ,fill opacity=1 ] (325.73,198.32) .. controls (325.73,197.28) and (324.89,196.45) .. (323.86,196.45) .. controls (322.83,196.45) and (322,197.28) .. (322,198.32) .. controls (322,199.35) and (322.83,200.19) .. (323.86,200.19) .. controls (324.89,200.19) and (325.73,199.35) .. (325.73,198.32) -- cycle ;
%Straight Lines [id:da6528048180071151] 
\draw    (323.86,168.36) -- (381.53,168.36) ;
\draw [shift={(383.53,168.36)}, rotate = 180] [color={rgb, 255:red, 0; green, 0; blue, 0 }  ][line width=0.75]    (10.93,-3.29) .. controls (6.95,-1.4) and (3.31,-0.3) .. (0,0) .. controls (3.31,0.3) and (6.95,1.4) .. (10.93,3.29)   ;
%Straight Lines [id:da7054684499225463] 
\draw [color={rgb, 255:red, 126; green, 211; blue, 33 }  ,draw opacity=1 ]   (531.67,70) -- (561.67,90) ;
%Straight Lines [id:da16385842504366877] 
\draw [color={rgb, 255:red, 126; green, 211; blue, 33 }  ,draw opacity=1 ]   (501.67,90) -- (531.67,110) ;
%Straight Lines [id:da9418685654080883] 
\draw    (561.67,90) -- (611.67,150) ;
%Straight Lines [id:da6231176415519419] 
\draw [color={rgb, 255:red, 208; green, 2; blue, 27 }  ,draw opacity=1 ]   (561.67,90) -- (531.67,110) ;
%Straight Lines [id:da6784731441776009] 
\draw [color={rgb, 255:red, 208; green, 2; blue, 27 }  ,draw opacity=1 ]   (501.67,90) -- (531.67,70) ;
%Straight Lines [id:da8814365599101575] 
\draw [color={rgb, 255:red, 74; green, 144; blue, 226 }  ,draw opacity=1 ]   (531.67,110) -- (510.53,130.33) ;
%Straight Lines [id:da5909750494546402] 
\draw [color={rgb, 255:red, 74; green, 144; blue, 226 }  ,draw opacity=1 ]   (480,110) -- (501.67,90) ;
%Straight Lines [id:da9835373416413931] 
\draw [color={rgb, 255:red, 126; green, 211; blue, 33 }  ,draw opacity=1 ]   (480,110) -- (510,130) ;
%Straight Lines [id:da11136368196365387] 
\draw    (510,130) -- (510,150) ;
%Straight Lines [id:da6914564358565802] 
\draw    (480,110) -- (480,150) ;
%Straight Lines [id:da6692522241366718] 
\draw [color={rgb, 255:red, 248; green, 231; blue, 28 }  ,draw opacity=1 ]   (430,130) -- (500,90) ;
%Straight Lines [id:da7083215378278274] 
\draw    (431.67,130) -- (451.67,150) ;
%Straight Lines [id:da24219964365275048] 
\draw    (401.67,150) -- (431.67,130) ;
%Shape: Ellipse [id:dp1713953160706363] 
\draw  [fill={rgb, 255:red, 0; green, 0; blue, 0 }  ,fill opacity=1 ] (403.53,150) .. controls (403.53,148.97) and (402.7,148.13) .. (401.67,148.13) .. controls (400.64,148.13) and (399.8,148.97) .. (399.8,150) .. controls (399.8,151.03) and (400.64,151.87) .. (401.67,151.87) .. controls (402.7,151.87) and (403.53,151.03) .. (403.53,150) -- cycle ;
%Shape: Ellipse [id:dp33017821397830716] 
\draw  [fill={rgb, 255:red, 0; green, 0; blue, 0 }  ,fill opacity=1 ] (453.53,150) .. controls (453.53,148.97) and (452.7,148.13) .. (451.67,148.13) .. controls (450.64,148.13) and (449.8,148.97) .. (449.8,150) .. controls (449.8,151.03) and (450.64,151.87) .. (451.67,151.87) .. controls (452.7,151.87) and (453.53,151.03) .. (453.53,150) -- cycle ;
%Shape: Ellipse [id:dp8089352066163109] 
\draw  [fill={rgb, 255:red, 0; green, 0; blue, 0 }  ,fill opacity=1 ] (481.86,150) .. controls (481.86,148.97) and (481.03,148.13) .. (480,148.13) .. controls (478.97,148.13) and (478.14,148.97) .. (478.14,150) .. controls (478.14,151.03) and (478.97,151.87) .. (480,151.87) .. controls (481.03,151.87) and (481.86,151.03) .. (481.86,150) -- cycle ;
%Shape: Ellipse [id:dp8179169020046098] 
\draw  [fill={rgb, 255:red, 0; green, 0; blue, 0 }  ,fill opacity=1 ] (511.86,150) .. controls (511.86,148.97) and (511.03,148.13) .. (510,148.13) .. controls (508.97,148.13) and (508.14,148.97) .. (508.14,150) .. controls (508.14,151.03) and (508.97,151.87) .. (510,151.87) .. controls (511.03,151.87) and (511.86,151.03) .. (511.86,150) -- cycle ;
%Shape: Ellipse [id:dp09306650335525912] 
\draw  [fill={rgb, 255:red, 0; green, 0; blue, 0 }  ,fill opacity=1 ] (613.53,150) .. controls (613.53,148.97) and (612.7,148.13) .. (611.67,148.13) .. controls (610.64,148.13) and (609.8,148.97) .. (609.8,150) .. controls (609.8,151.03) and (610.64,151.87) .. (611.67,151.87) .. controls (612.7,151.87) and (613.53,151.03) .. (613.53,150) -- cycle ;

% Text Node
\draw (77.24,84.25) node [anchor=north west][inner sep=0.75pt]   [align=left] {0};
% Text Node
\draw (76.94,204.01) node [anchor=north west][inner sep=0.75pt]   [align=left] {3};
% Text Node
\draw (130.93,174.36) node [anchor=north west][inner sep=0.75pt]   [align=left] {2};
% Text Node
\draw (128.84,109.95) node [anchor=north west][inner sep=0.75pt]   [align=left] {1};
% Text Node
\draw (32.19,113.55) node [anchor=north west][inner sep=0.75pt]   [align=left] {5};
% Text Node
\draw (31,177.95) node [anchor=north west][inner sep=0.75pt]   [align=left] {4};
% Text Node
\draw (243.79,72) node [anchor=north west][inner sep=0.75pt]  [font=\tiny] [align=left] {0};
% Text Node
\draw (304.94,72.75) node [anchor=north west][inner sep=0.75pt]  [font=\tiny] [align=left] {1};
% Text Node
\draw (324.08,199.81) node [anchor=north west][inner sep=0.75pt]  [font=\tiny] [align=left] {2};
% Text Node
\draw (276.11,202.31) node [anchor=north west][inner sep=0.75pt]  [font=\tiny] [align=left] {3};
% Text Node
\draw (216.45,180.09) node [anchor=north west][inner sep=0.75pt]  [font=\tiny] [align=left] {4};
% Text Node
\draw (213.71,131.91) node [anchor=north west][inner sep=0.75pt]  [font=\tiny] [align=left] {5};
% Text Node
\draw (528.33,59) node [anchor=north west][inner sep=0.75pt]  [font=\tiny] [align=left] {0};
% Text Node
\draw (611.28,156.42) node [anchor=north west][inner sep=0.75pt]  [font=\tiny] [align=left] {1};
% Text Node
\draw (504.33,157) node [anchor=north west][inner sep=0.75pt]  [font=\tiny] [align=left] {2};
% Text Node
\draw (474.28,156.42) node [anchor=north west][inner sep=0.75pt]  [font=\tiny] [align=left] {3};
% Text Node
\draw (448.28,156.08) node [anchor=north west][inner sep=0.75pt]  [font=\tiny] [align=left] {4};
% Text Node
\draw (398.61,156.42) node [anchor=north west][inner sep=0.75pt]  [font=\tiny] [align=left] {5};

\end{tikzpicture}

\caption{The Dual Polygon Representation of the Split System $N=\{01|2345,12|0345,0145|23,0123|45 \}$ from Example \ref{dualpolygonex}. Then Algorithm \ref{netalg} is applied to create a split network representing $N$. Lastly, the split network is drawn rooted at $0$.}\label{dualpolyfig}
\end{figure}

\begin{exmp} \label{dualpolygonex}
Consider the split system $N$ on 5 leaves and one root $0$ with the following splits:
\[ 
N=\{01|2345,12|0345,0145|23,0123|45 \} \cup \{ i | \{0,1,2, 3, 5\} \setminus i : i \in [5]  \}.
\]
The dual polygon representation for this split system is a hexagon with the edges 
labeled $0$ through $5$, the vertices labeled by the edge to their right, 
and the  diagonals connecting the following vertices: $\{0 - 2, 1- 3, 2- 4, 4- 0 \}$. 
See Figure \ref{dualpolyfig} for a visualization of this dual polygon representation. 
Additionally, in Figure \ref{dualpolyfig} there is an application of Algorithm 
\ref{netalg} to $N$ and a rooting of the resulting split network.
\end{exmp}

A dual polygon representation corresponds to a tree if and only if none of the diagonals intersect each other
in their interiors. 
A representational split network can be constructed from the dual polygon
representation by an application of the circular network algorithm to the underlying split system. 
If the dual polygon representation was rooted, to make the  split network rooted, 
put the leaf labeled $0$ at the top and draw the network descending from that vertex. 

In addition to the root being the most recent common ancestor of every vertex in a rooted network,  rooted split networks have a notion of a poset on a the vertices. The join of two species is the most recent common ancestor of those species, making this poset of biological interest.

\begin{definition}
 Let  $G$ be the graph for a split network $(G,s,f)$ that represents $N$, a rooted split system. 
 The vertices of $G$ form a \emph{poset}. The root is the maximal element. Let $P_x$ be the set of paths of minimal length to the root from the vertex $x$. A vertex $y$ is above $x$ if any of the paths in $P_x$ pass through $y$. The join of two leaves, $i$ and $j$ will be referred to as the most recent common ancestor of $i$ and $j$ or $mcra(i,j)$.
\end{definition}

\begin{exmp}\label{posetex}
    Consider the split system $N$ on 4 leaves and one root $0$ with splits 
    \[ 
N=\{01|234,12|034,014|23,04|123,012|34 \} \cup \{ i | \{0,1, 2, 3, 4\}  \setminus i : i \in [4]  \}.
\]
The Hasse diagram for the poset that the split network representing $N$ forms, can be seen in Figure \ref{Hasse}.
\end{exmp}
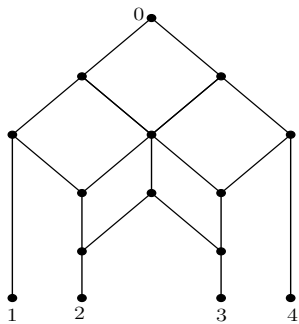
\begin{figure}[t]

\centering

\resizebox{120pt}{120pt}{

\tikzset{every picture/.style={line width=0.75pt}} %set default line width to 0.75pt        

\begin{tikzpicture}[x=0.75pt,y=0.75pt,yscale=-1,xscale=1]

%Straight Lines [id:da5524642329943019] 
\draw    (300,60) -- (350,10) ;
%Straight Lines [id:da42677028767391745] 
\draw    (350,10) -- (400,60) ;
%Straight Lines [id:da7966685816588059] 
\draw    (350,110) -- (300,60) ;
%Straight Lines [id:da8393687350592678] 
\draw    (350,110) -- (400,60) ;
%Straight Lines [id:da06645422365229936] 
\draw    (250,110) -- (300,60) ;
%Straight Lines [id:da9333917121310187] 
\draw    (300,60) -- (350,110) ;
%Straight Lines [id:da5990860940295131] 
\draw    (300,160) -- (250,110) ;
%Straight Lines [id:da0643563754108929] 
\draw    (300,160) -- (350,110) ;
%Straight Lines [id:da7656075626471264] 
\draw    (350,110) -- (400,60) ;
%Straight Lines [id:da5977892117754919] 
\draw    (400,60) -- (450,110) ;
%Straight Lines [id:da4879422365330075] 
\draw    (400,160) -- (350,110) ;
%Straight Lines [id:da2468052929319644] 
\draw    (400,160) -- (450,110) ;
%Straight Lines [id:da9133956994225583] 
\draw    (350,160) -- (350,110) ;
%Straight Lines [id:da2725924011602048] 
\draw    (300,210) -- (300,160) ;
%Straight Lines [id:da645965945136086] 
\draw    (400,210) -- (400,160) ;
%Straight Lines [id:da7454603865060601] 
\draw    (400,210) -- (350,160) ;
%Straight Lines [id:da20333855739996043] 
\draw    (300,210) -- (350,160) ;
%Straight Lines [id:da9791484816643896] 
\draw    (250,250) -- (250,110) ;
%Straight Lines [id:da8896202856670574] 
\draw    (300,250) -- (300,210) ;
%Straight Lines [id:da8880496977836061] 
\draw    (400,250) -- (400,210) ;
%Straight Lines [id:da1947784072472274] 
\draw    (450,250) -- (450,110) ;
%Shape: Circle [id:dp24494049252417516] 
\draw  [fill={rgb, 255:red, 0; green, 0; blue, 0 }  ,fill opacity=1 ] (247,110) .. controls (247,108.34) and (248.34,107) .. (250,107) .. controls (251.66,107) and (253,108.34) .. (253,110) .. controls (253,111.66) and (251.66,113) .. (250,113) .. controls (248.34,113) and (247,111.66) .. (247,110) -- cycle ;
%Shape: Circle [id:dp3236688255739748] 
\draw  [fill={rgb, 255:red, 0; green, 0; blue, 0 }  ,fill opacity=1 ] (297,60) .. controls (297,58.34) and (298.34,57) .. (300,57) .. controls (301.66,57) and (303,58.34) .. (303,60) .. controls (303,61.66) and (301.66,63) .. (300,63) .. controls (298.34,63) and (297,61.66) .. (297,60) -- cycle ;
%Shape: Circle [id:dp7924110493525061] 
\draw  [fill={rgb, 255:red, 0; green, 0; blue, 0 }  ,fill opacity=1 ] (347,10) .. controls (347,8.34) and (348.34,7) .. (350,7) .. controls (351.66,7) and (353,8.34) .. (353,10) .. controls (353,11.66) and (351.66,13) .. (350,13) .. controls (348.34,13) and (347,11.66) .. (347,10) -- cycle ;
%Shape: Circle [id:dp1932823124422145] 
\draw  [fill={rgb, 255:red, 0; green, 0; blue, 0 }  ,fill opacity=1 ] (347,110) .. controls (347,108.34) and (348.34,107) .. (350,107) .. controls (351.66,107) and (353,108.34) .. (353,110) .. controls (353,111.66) and (351.66,113) .. (350,113) .. controls (348.34,113) and (347,111.66) .. (347,110) -- cycle ;
%Shape: Circle [id:dp026237071648756327] 
\draw  [fill={rgb, 255:red, 0; green, 0; blue, 0 }  ,fill opacity=1 ] (397,60) .. controls (397,58.34) and (398.34,57) .. (400,57) .. controls (401.66,57) and (403,58.34) .. (403,60) .. controls (403,61.66) and (401.66,63) .. (400,63) .. controls (398.34,63) and (397,61.66) .. (397,60) -- cycle ;
%Shape: Circle [id:dp28970892458341857] 
\draw  [fill={rgb, 255:red, 0; green, 0; blue, 0 }  ,fill opacity=1 ] (447,110) .. controls (447,108.34) and (448.34,107) .. (450,107) .. controls (451.66,107) and (453,108.34) .. (453,110) .. controls (453,111.66) and (451.66,113) .. (450,113) .. controls (448.34,113) and (447,111.66) .. (447,110) -- cycle ;
%Shape: Circle [id:dp727418828359431] 
\draw  [fill={rgb, 255:red, 0; green, 0; blue, 0 }  ,fill opacity=1 ] (297,160) .. controls (297,158.34) and (298.34,157) .. (300,157) .. controls (301.66,157) and (303,158.34) .. (303,160) .. controls (303,161.66) and (301.66,163) .. (300,163) .. controls (298.34,163) and (297,161.66) .. (297,160) -- cycle ;
%Shape: Circle [id:dp5306253173035567] 
\draw  [fill={rgb, 255:red, 0; green, 0; blue, 0 }  ,fill opacity=1 ] (397,160) .. controls (397,158.34) and (398.34,157) .. (400,157) .. controls (401.66,157) and (403,158.34) .. (403,160) .. controls (403,161.66) and (401.66,163) .. (400,163) .. controls (398.34,163) and (397,161.66) .. (397,160) -- cycle ;
%Shape: Circle [id:dp3610419009716219] 
\draw  [fill={rgb, 255:red, 0; green, 0; blue, 0 }  ,fill opacity=1 ] (347,160) .. controls (347,158.34) and (348.34,157) .. (350,157) .. controls (351.66,157) and (353,158.34) .. (353,160) .. controls (353,161.66) and (351.66,163) .. (350,163) .. controls (348.34,163) and (347,161.66) .. (347,160) -- cycle ;
%Shape: Circle [id:dp7201438348507399] 
\draw  [fill={rgb, 255:red, 0; green, 0; blue, 0 }  ,fill opacity=1 ] (247,250) .. controls (247,248.34) and (248.34,247) .. (250,247) .. controls (251.66,247) and (253,248.34) .. (253,250) .. controls (253,251.66) and (251.66,253) .. (250,253) .. controls (248.34,253) and (247,251.66) .. (247,250) -- cycle ;
%Shape: Circle [id:dp7786701592357972] 
\draw  [fill={rgb, 255:red, 0; green, 0; blue, 0 }  ,fill opacity=1 ] (297,250) .. controls (297,248.34) and (298.34,247) .. (300,247) .. controls (301.66,247) and (303,248.34) .. (303,250) .. controls (303,251.66) and (301.66,253) .. (300,253) .. controls (298.34,253) and (297,251.66) .. (297,250) -- cycle ;
%Shape: Circle [id:dp28364669070648296] 
\draw  [fill={rgb, 255:red, 0; green, 0; blue, 0 }  ,fill opacity=1 ] (297,210) .. controls (297,208.34) and (298.34,207) .. (300,207) .. controls (301.66,207) and (303,208.34) .. (303,210) .. controls (303,211.66) and (301.66,213) .. (300,213) .. controls (298.34,213) and (297,211.66) .. (297,210) -- cycle ;
%Shape: Circle [id:dp9285400912906281] 
\draw  [fill={rgb, 255:red, 0; green, 0; blue, 0 }  ,fill opacity=1 ] (397,210) .. controls (397,208.34) and (398.34,207) .. (400,207) .. controls (401.66,207) and (403,208.34) .. (403,210) .. controls (403,211.66) and (401.66,213) .. (400,213) .. controls (398.34,213) and (397,211.66) .. (397,210) -- cycle ;
%Shape: Circle [id:dp6754457311659847] 
\draw  [fill={rgb, 255:red, 0; green, 0; blue, 0 }  ,fill opacity=1 ] (397,250) .. controls (397,248.34) and (398.34,247) .. (400,247) .. controls (401.66,247) and (403,248.34) .. (403,250) .. controls (403,251.66) and (401.66,253) .. (400,253) .. controls (398.34,253) and (397,251.66) .. (397,250) -- cycle ;
%Shape: Circle [id:dp6098443535728073] 
\draw  [fill={rgb, 255:red, 0; green, 0; blue, 0 }  ,fill opacity=1 ] (447,250) .. controls (447,248.34) and (448.34,247) .. (450,247) .. controls (451.66,247) and (453,248.34) .. (453,250) .. controls (453,251.66) and (451.66,253) .. (450,253) .. controls (448.34,253) and (447,251.66) .. (447,250) -- cycle ;

% Text Node
\draw (244.5,257.5) node [anchor=north west][inner sep=0.75pt]   [align=left] {1};
% Text Node
\draw (293,256.5) node [anchor=north west][inner sep=0.75pt]   [align=left] {2};
% Text Node
\draw (395,258) node [anchor=north west][inner sep=0.75pt]   [align=left] {3};
% Text Node
\draw (446,257.5) node [anchor=north west][inner sep=0.75pt]   [align=left] {4};
% Text Node
\draw (335.5,-0.4) node [anchor=north west][inner sep=0.75pt]   [align=left] {0};

\end{tikzpicture}

}

\caption{The Hasse diagram for the poset that is formed from the split system $N=\{01|234,12|034,014|23,04|123,012|34 \}$ from Example \ref{posetex}.}\label{Hasse}

\end{figure}

%%%%%%%%%%%%%%%%%%%%%%%%%%%%%%%%%%%%%%%%%%%%%%%%%%%
%%%%%%%%%%%%%%%%%%%%%%%%%%%%%%%%%%%%%%%%%%%%%%%%%%%
%%%%%%%%%%%%%%%%%%%%%%%%%%%%%%%%%%%%%%%%%%%%%%%%%%%
%%%%%%%%%%%%%%%%%%%%%%%%%%%%%%%%%%%%%%%%%%%%%%%%%%%

\section{Distances from split networks}

Now that we have defined the split systems that we will study in the paper, 
we introduce distance matrices associated to the
split systems, which generalize tree metrics. Our goal in this paper will be to give polyhedral characterizations of the
dissimilarity maps that can arise from rooted circular split systems. To start, we will define dissimilarity maps and show how they relate to tree metrics.

\begin{definition}
A function $\delta: X \times X \to \mathbb{R}_{\geq 0}
$  is called a \emph{dissimilarity map} if it is a real, symmetric, and nonnegative function.
In addition, a dissimilarity map  is a \emph{metric} if it satisfies the triangle inequality on $X$, meaning for all $x,y,z\in X$,
\[
\delta(x,z)\leq \delta(x,y)+\delta(y,z).
\]
\end{definition}

The following two theorems give necessary and sufficient conditions for when a dissimilarity
map comes from a tree or a circular split system. For Theorem \ref{dis} see \cite{phy} for more details.
 
\begin{theorem}{(Four Point Condition)}\label{dis}
Let $\delta$ be a dissimilarity map on a finite set $X$. Then $\delta$ is the distance function for a tree
if and
only if it satisfies the four-point condition: for every four elements $i,j,k,l\in X$, two of
the three terms 
\[ 
\delta(i, j ) + \delta(k, l), \; \; \delta(i
, l) + \delta(j , k), \; \; \delta(i
, k) + \delta(j , l).
\]
are equal and are greater than or equal to the third.
\end{theorem} 

%For leaves $1,2,3,5$: \[\delta(1, 2 ) + \delta(3, 5)=21, \; \; \delta(1
%, 3) + \delta(2 , 5)=33, \; \; \delta(1
%, 5) + \delta(2 , 3)=33.\]
%For leaves $1,2,4,5$: \[\delta(1, 2 ) + \delta(4, 5)=12, \; \; \delta(1
%, 4) + \delta(2 , 5)=30, \; \; \delta(1
%, 5) + \delta(2 , 4)=30.\]
%For leaves $1,3,4,5$: \[\delta(1, 3 ) + \delta(4, 5)=22, \; \; \delta(1
%, 4) + \delta(3 , 5)=28, \; \; \delta(1
%, 5) + \delta(3 , 4)=28.\]
%For leaves $2,3,4,5$: \[\delta(2, 3 ) + \delta(4, 5)=23, \; \; \delta(2
%, 4) + \delta(3 , 5)=29, \; \; \delta(2
%, 5) + \delta(3 , 4)=29.\]
\begin{exmp}\label{distexample}
    Consider the dissimilarity map on the set $[5]$:
    $$\delta=\begin{bmatrix}
0 & 5 & 15 & 12 & 17\\
 & 0 & 16 & 13& 18\\
 &  & 0 & 11 & 16\\
&  &  & 0 & 7\\
 &  &  &  & 0\\
\end{bmatrix}.$$
This is a distance function for a tree which can be seen by calculating $\delta(i, j ) + \delta(k, l), \; \; \delta(i
, l) + \delta(j , k), \; \; \delta(i
, k) + \delta(j , l)$ for $i,j,k,l \in [5].$ \\
For example, for leaves $1,2,3,4$: \[\delta(1, 2 ) + \delta(3, 4)=16, \; \; \delta(1
, 3) + \delta(2 , 4)=28, \; \; \delta(1
, 4) + \delta(2 , 3)=28.\]

Since every set of 4 leaves satisfies the conditions of Theorem \ref{dis}, $\delta$ is a distance function for a tree. 
A tree that represents the dissimilarity can be seen in Figure \ref{treedist}. See \cite{Steel} for further details on this definition of distance.
\end{exmp}

\begin{figure}[t]

\centering

\resizebox{120pt}{80pt}{

\tikzset{every picture/.style={line width=0.75pt}} %set default line width to 0.75pt        

\begin{tikzpicture}[x=0.75pt,y=0.75pt,yscale=-1,xscale=1]

%Straight Lines [id:da06944366066771335] 
\draw    (290,70) -- (260,100) ;
%Straight Lines [id:da1527937796559724] 
\draw    (290,100) -- (280,80) ;
%Straight Lines [id:da9097092915712146] 
\draw    (290,70) -- (320,40) ;
%Straight Lines [id:da12647762866564505] 
\draw    (320,40) -- (350,70) ;
%Straight Lines [id:da5536286849545449] 
\draw    (350,70) -- (370,90) ;
%Straight Lines [id:da9643673708801186] 
\draw    (320,100) -- (340,60) ;
%Straight Lines [id:da6444831198935761] 
\draw    (350,100) -- (360,80) ;
%Straight Lines [id:da8001377128507627] 
\draw    (360,80) -- (380,100) ;
%Shape: Circle [id:dp3974326977147489] 
\draw  [fill={rgb, 255:red, 0; green, 0; blue, 0 }  ,fill opacity=1 ] (257.5,100) .. controls (257.5,98.62) and (258.62,97.5) .. (260,97.5) .. controls (261.38,97.5) and (262.5,98.62) .. (262.5,100) .. controls (262.5,101.38) and (261.38,102.5) .. (260,102.5) .. controls (258.62,102.5) and (257.5,101.38) .. (257.5,100) -- cycle ;
%Shape: Circle [id:dp5872443355687469] 
\draw  [fill={rgb, 255:red, 0; green, 0; blue, 0 }  ,fill opacity=1 ] (287.5,100) .. controls (287.5,98.62) and (288.62,97.5) .. (290,97.5) .. controls (291.38,97.5) and (292.5,98.62) .. (292.5,100) .. controls (292.5,101.38) and (291.38,102.5) .. (290,102.5) .. controls (288.62,102.5) and (287.5,101.38) .. (287.5,100) -- cycle ;
%Shape: Circle [id:dp6385379038608741] 
\draw  [fill={rgb, 255:red, 0; green, 0; blue, 0 }  ,fill opacity=1 ] (317.5,100) .. controls (317.5,98.62) and (318.62,97.5) .. (320,97.5) .. controls (321.38,97.5) and (322.5,98.62) .. (322.5,100) .. controls (322.5,101.38) and (321.38,102.5) .. (320,102.5) .. controls (318.62,102.5) and (317.5,101.38) .. (317.5,100) -- cycle ;
%Shape: Circle [id:dp4378536197834064] 
\draw  [fill={rgb, 255:red, 0; green, 0; blue, 0 }  ,fill opacity=1 ] (347.5,100) .. controls (347.5,98.62) and (348.62,97.5) .. (350,97.5) .. controls (351.38,97.5) and (352.5,98.62) .. (352.5,100) .. controls (352.5,101.38) and (351.38,102.5) .. (350,102.5) .. controls (348.62,102.5) and (347.5,101.38) .. (347.5,100) -- cycle ;
%Shape: Circle [id:dp06537841449156745] 
\draw  [fill={rgb, 255:red, 0; green, 0; blue, 0 }  ,fill opacity=1 ] (377.5,100) .. controls (377.5,98.62) and (378.62,97.5) .. (380,97.5) .. controls (381.38,97.5) and (382.5,98.62) .. (382.5,100) .. controls (382.5,101.38) and (381.38,102.5) .. (380,102.5) .. controls (378.62,102.5) and (377.5,101.38) .. (377.5,100) -- cycle ;

% Text Node
\draw (256.67,105) node [anchor=north west][inner sep=0.75pt]  [font=\tiny] [align=left] {1};
% Text Node
\draw (287.33,104.67) node [anchor=north west][inner sep=0.75pt]  [font=\tiny] [align=left] {2};
% Text Node
\draw (317,104.67) node [anchor=north west][inner sep=0.75pt]  [font=\tiny] [align=left] {3};
% Text Node
\draw (346.33,104.67) node [anchor=north west][inner sep=0.75pt]  [font=\tiny] [align=left] {4};
% Text Node
\draw (377.67,104) node [anchor=north west][inner sep=0.75pt]  [font=\tiny] [align=left] {5};
% Text Node
\draw (263,79.67) node [anchor=north west][inner sep=0.75pt]  [font=\tiny,color={rgb, 255:red, 208; green, 2; blue, 27 }  ,opacity=1 ] [align=left] {2};
% Text Node
\draw (286,81.67) node [anchor=north west][inner sep=0.75pt]  [font=\tiny,color={rgb, 255:red, 208; green, 2; blue, 27 }  ,opacity=1 ] [align=left] {3};
% Text Node
\draw (291.67,51.67) node [anchor=north west][inner sep=0.75pt]  [font=\tiny,color={rgb, 255:red, 208; green, 2; blue, 27 }  ,opacity=1 ] [align=left] {5};
% Text Node
\draw (330,40.33) node [anchor=north west][inner sep=0.75pt]  [font=\tiny,color={rgb, 255:red, 208; green, 2; blue, 27 }  ,opacity=1 ] [align=left] {1};
% Text Node
\draw (351.33,63) node [anchor=north west][inner sep=0.75pt]  [font=\tiny,color={rgb, 255:red, 208; green, 2; blue, 27 }  ,opacity=1 ] [align=left] {3};
% Text Node
\draw (321.67,72.33) node [anchor=north west][inner sep=0.75pt]  [font=\tiny,color={rgb, 255:red, 208; green, 2; blue, 27 }  ,opacity=1 ] [align=left] {7};
% Text Node
\draw (347.67,78.33) node [anchor=north west][inner sep=0.75pt]  [font=\tiny,color={rgb, 255:red, 208; green, 2; blue, 27 }  ,opacity=1 ] [align=left] {1};
% Text Node
\draw (369,80) node [anchor=north west][inner sep=0.75pt]  [font=\tiny,color={rgb, 255:red, 208; green, 2; blue, 27 }  ,opacity=1 ] [align=left] {6};

\end{tikzpicture}

}

\caption{This weighted tree from Example \ref{distexample} }\label{treedist}

\end{figure}
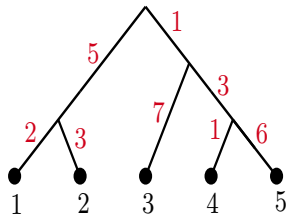
This notion of distance on a tree can be generalized to split networks using the concept of a split separating leaves which is defined below:
\begin{definition}
Let $N$ be a split system, and $i,j \in[n]$. 
A split $A|B\in N$ \emph{separates} $i$ and $j$ in $N$ if $i\in A$ and $j \notin A$ or $i\in B$ and $j \notin B$.
\end{definition}

\begin{definition}
Let $A|B$ be a split and let the \emph{separation indicator function} for  $A|B$ be defined as:
\[
\mathbbm{1}_{A|B}(i,j)=
\begin{cases}
1&\text{ if } A|B  \text{ separates } i \text{ and } j  \\
 0  &  \text{ otherwise.}
\end{cases}
\]
\end{definition}

Note that $\mathbbm{1}_{A|B}$ is an example of a 
semimetric, since it is nonnegative, symmetric, and satisfies the triangle
inequality.  In the context of the theory of finite metric spaces, $\mathbbm{1}_{A|B}$
is known as a cut semimetric \cite{Deza2010}.

\begin{exmp}
    Consider the split $12|345$. 
    This split separates $1$ from $3$, $1$ from $4$, $1$ from $5$, $2$ from $3$, $2$ from $4$, and $2$ from $5$.
    The function $ \mathbbm{1}_{A|B}$   satisfies  $\mathbbm{1}_{A|B}(1,2)  = \mathbbm{1}_{A|B}(3,4)  = \mathbbm{1}_{A|B}(3,5) = \mathbbm{1}_{A|B}(4,5)  = 0$, and $\mathbbm{1}_{A|B}(1,3) = \mathbbm{1}_{A|B}(1,4) = \mathbbm{1}_{A|B}(1,5) = 
    \mathbbm{1}_{A|B}(2,3) = \mathbbm{1}_{A|B}(2,4) = \mathbbm{1}_{A|B}(2,5) = 1$. 
\end{exmp}

\begin{definition} 
Let $N$ be a split system with $n$ leaves. For each split $A|B\in N$, let $a_{A|B}\in \mathbb{R}_{\geq 0}$ be a weight. Let $\textbf{a}$ be the vector of weights for every split. The distance function between any two leaves $i,j \in [n]$  in the split system $N$ with weights $\textbf{a}$, denoted $\delta_{N,\textbf{a}}$ , is defined as follows:
 \[
 \delta_{N,\textbf{a}}(i,j)=\sum_{A|B\in N} a_{A|B} \mathbbm{1}_{A|B}(i,j).
 \]
\end{definition}

If it is clear in context which $N$ and $\textbf{a}$ are being used we will write $\delta_{N,\textbf{a}}(i,j)=\delta(i,j)$.   If a split system is rooted, $\delta(0,i)$ will denote the distance from the root $0$ to $i\in [n]$.
Note that if $N$ is a split network and the edges are labeled with the weight 
from the split they are realizing, then this distance between leaves is exactly the 
distance obtained from the sum of the edge weights along any shortest path between the leaves. 
Since sets of edges that all realize the same split will have the same weight, the distance between any two leaves, $i,j\in [n]$,  can be calculated in the dual polygon representation by drawing a line between the sides of the $n$-gon for those leaves  
and then taking the sum over weights of all diagonals that line intersects, 
as well as $a_{1\dots (i-1) (i+1) \dots n | i} $ and $a_{1\dots (j-1) (j+1) \dots n | j} $.

For circular split systems, there is a nice condition for when a metric represents the system, which can be seen in the following theorem:
\begin{theorem}
\cite{Kalmanson}\label{Kalmanson} Let $\delta$ be a  metric on a finite set $X$. 
Then $\delta$  is the distance function for a circular split system $N$ with weights $\textbf{a}$  if and
only if it satisfies the Kalmanson condition with respect to some circular ordering $\pi$.   
That is,
for every set of leaves $i < j < k < l$ under $\pi$, both inequalities hold:
\[
\delta(i, j ) + \delta(k, l) \leq \delta(i , k) + \delta(j , l)
\quad \quad \quad 
\delta(i, l) + \delta(j , k) \leq \delta(i, k) + \delta(j , l).
\]
\end{theorem}

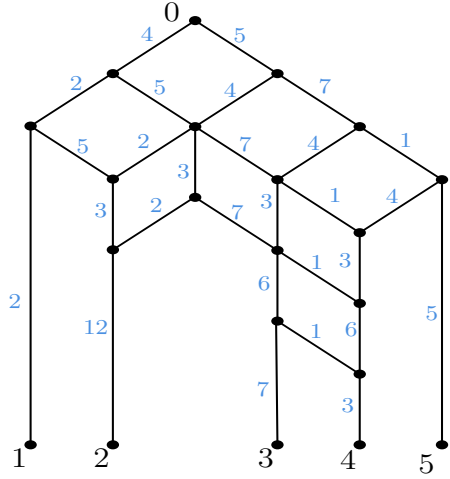
\begin{figure}[t]

\centering

\resizebox{180pt}{180pt}{

\tikzset{every picture/.style={line width=0.75pt}} %set default line width to 0.75pt        

\begin{tikzpicture}[x=0.75pt,y=0.75pt,yscale=-1,xscale=1]

%Straight Lines [id:da5524642329943019] 
\draw    (290,40) -- (330,10) ;
%Straight Lines [id:da42677028767391745] 
\draw    (330,10) -- (370,40) ;
%Straight Lines [id:da7966685816588059] 
\draw    (370,40) -- (330,70) ;
%Straight Lines [id:da9333917121310187] 
\draw    (290,40) -- (330,70) ;
%Shape: Ellipse [id:dp3236688255739748] 
\draw  [fill={rgb, 255:red, 0; green, 0; blue, 0 }  ,fill opacity=1 ] (287.5,40) .. controls (287.5,38.86) and (288.62,37.93) .. (290,37.93) .. controls (291.38,37.93) and (292.5,38.86) .. (292.5,40) .. controls (292.5,41.14) and (291.38,42.07) .. (290,42.07) .. controls (288.62,42.07) and (287.5,41.14) .. (287.5,40) -- cycle ;
%Shape: Ellipse [id:dp7924110493525061] 
\draw  [fill={rgb, 255:red, 0; green, 0; blue, 0 }  ,fill opacity=1 ] (327.5,10) .. controls (327.5,8.86) and (328.62,7.93) .. (330,7.93) .. controls (331.38,7.93) and (332.5,8.86) .. (332.5,10) .. controls (332.5,11.14) and (331.38,12.07) .. (330,12.07) .. controls (328.62,12.07) and (327.5,11.14) .. (327.5,10) -- cycle ;
%Shape: Ellipse [id:dp1932823124422145] 
\draw  [fill={rgb, 255:red, 0; green, 0; blue, 0 }  ,fill opacity=1 ] (327.5,70) .. controls (327.5,68.86) and (328.62,67.93) .. (330,67.93) .. controls (331.38,67.93) and (332.5,68.86) .. (332.5,70) .. controls (332.5,71.14) and (331.38,72.07) .. (330,72.07) .. controls (328.62,72.07) and (327.5,71.14) .. (327.5,70) -- cycle ;
%Shape: Ellipse [id:dp026237071648756327] 
\draw  [fill={rgb, 255:red, 0; green, 0; blue, 0 }  ,fill opacity=1 ] (367.5,40) .. controls (367.5,38.86) and (368.62,37.93) .. (370,37.93) .. controls (371.38,37.93) and (372.5,38.86) .. (372.5,40) .. controls (372.5,41.14) and (371.38,42.07) .. (370,42.07) .. controls (368.62,42.07) and (367.5,41.14) .. (367.5,40) -- cycle ;
%Straight Lines [id:da5665994704520634] 
\draw    (250,250) -- (250,69.67) ;
%Straight Lines [id:da37952234767009796] 
\draw    (250,69.67) -- (290,39.67) ;
%Straight Lines [id:da41637230679070236] 
\draw    (330,69.67) -- (290,99.67) ;
%Straight Lines [id:da985826558039443] 
\draw    (250,69.67) -- (290,99.67) ;
%Shape: Ellipse [id:dp8417045183923113] 
\draw  [fill={rgb, 255:red, 0; green, 0; blue, 0 }  ,fill opacity=1 ] (247.5,69.67) .. controls (247.5,68.53) and (248.62,67.6) .. (250,67.6) .. controls (251.38,67.6) and (252.5,68.53) .. (252.5,69.67) .. controls (252.5,70.81) and (251.38,71.73) .. (250,71.73) .. controls (248.62,71.73) and (247.5,70.81) .. (247.5,69.67) -- cycle ;
%Shape: Ellipse [id:dp42698634315070905] 
\draw  [fill={rgb, 255:red, 0; green, 0; blue, 0 }  ,fill opacity=1 ] (287.5,99.67) .. controls (287.5,98.53) and (288.62,97.6) .. (290,97.6) .. controls (291.38,97.6) and (292.5,98.53) .. (292.5,99.67) .. controls (292.5,100.81) and (291.38,101.73) .. (290,101.73) .. controls (288.62,101.73) and (287.5,100.81) .. (287.5,99.67) -- cycle ;
%Straight Lines [id:da9211522393109044] 
\draw    (370,40) -- (410,70) ;
%Straight Lines [id:da10001094480277328] 
\draw    (410,70) -- (370,100) ;
%Straight Lines [id:da8767259205844067] 
\draw    (330,70) -- (370,100) ;
%Shape: Ellipse [id:dp7880593704337182] 
\draw  [fill={rgb, 255:red, 0; green, 0; blue, 0 }  ,fill opacity=1 ] (367.5,100) .. controls (367.5,98.86) and (368.62,97.93) .. (370,97.93) .. controls (371.38,97.93) and (372.5,98.86) .. (372.5,100) .. controls (372.5,101.14) and (371.38,102.07) .. (370,102.07) .. controls (368.62,102.07) and (367.5,101.14) .. (367.5,100) -- cycle ;
%Shape: Ellipse [id:dp35957904152119546] 
\draw  [fill={rgb, 255:red, 0; green, 0; blue, 0 }  ,fill opacity=1 ] (407.5,70) .. controls (407.5,68.86) and (408.62,67.93) .. (410,67.93) .. controls (411.38,67.93) and (412.5,68.86) .. (412.5,70) .. controls (412.5,71.14) and (411.38,72.07) .. (410,72.07) .. controls (408.62,72.07) and (407.5,71.14) .. (407.5,70) -- cycle ;
%Straight Lines [id:da5259227841805727] 
\draw    (410,70) -- (450,100) ;
%Straight Lines [id:da2262070956751443] 
\draw    (450,100) -- (410,130) ;
%Straight Lines [id:da9495221731957055] 
\draw    (370,100) -- (410,130) ;
%Shape: Ellipse [id:dp11932182931315816] 
\draw  [fill={rgb, 255:red, 0; green, 0; blue, 0 }  ,fill opacity=1 ] (407.5,130) .. controls (407.5,128.86) and (408.62,127.93) .. (410,127.93) .. controls (411.38,127.93) and (412.5,128.86) .. (412.5,130) .. controls (412.5,131.14) and (411.38,132.07) .. (410,132.07) .. controls (408.62,132.07) and (407.5,131.14) .. (407.5,130) -- cycle ;
%Shape: Ellipse [id:dp4565449407075195] 
\draw  [fill={rgb, 255:red, 0; green, 0; blue, 0 }  ,fill opacity=1 ] (447.5,100) .. controls (447.5,98.86) and (448.62,97.93) .. (450,97.93) .. controls (451.38,97.93) and (452.5,98.86) .. (452.5,100) .. controls (452.5,101.14) and (451.38,102.07) .. (450,102.07) .. controls (448.62,102.07) and (447.5,101.14) .. (447.5,100) -- cycle ;
%Straight Lines [id:da47597631349477765] 
\draw    (330,70) -- (330,110) ;
%Shape: Ellipse [id:dp00935656872025481] 
\draw  [fill={rgb, 255:red, 0; green, 0; blue, 0 }  ,fill opacity=1 ] (327.5,110) .. controls (327.5,108.86) and (328.62,107.93) .. (330,107.93) .. controls (331.38,107.93) and (332.5,108.86) .. (332.5,110) .. controls (332.5,111.14) and (331.38,112.07) .. (330,112.07) .. controls (328.62,112.07) and (327.5,111.14) .. (327.5,110) -- cycle ;
%Straight Lines [id:da2834230115846885] 
\draw    (370,140) -- (330,110) ;
%Shape: Ellipse [id:dp2193305825532872] 
\draw  [fill={rgb, 255:red, 0; green, 0; blue, 0 }  ,fill opacity=1 ] (367.5,140) .. controls (367.5,138.86) and (368.62,137.93) .. (370,137.93) .. controls (371.38,137.93) and (372.5,138.86) .. (372.5,140) .. controls (372.5,141.14) and (371.38,142.07) .. (370,142.07) .. controls (368.62,142.07) and (367.5,141.14) .. (367.5,140) -- cycle ;
%Straight Lines [id:da6582716856272426] 
\draw    (370,100) -- (370,140) ;
%Straight Lines [id:da062391012848801886] 
\draw    (290,99.67) -- (290,139.67) ;
%Straight Lines [id:da746304590214377] 
\draw    (330,109.67) -- (290,139.67) ;
%Straight Lines [id:da8743342847390179] 
\draw    (410,130) -- (410,170) ;
%Shape: Ellipse [id:dp13787881343845698] 
\draw  [fill={rgb, 255:red, 0; green, 0; blue, 0 }  ,fill opacity=1 ] (287.5,139.67) .. controls (287.5,138.53) and (288.62,137.6) .. (290,137.6) .. controls (291.38,137.6) and (292.5,138.53) .. (292.5,139.67) .. controls (292.5,140.81) and (291.38,141.73) .. (290,141.73) .. controls (288.62,141.73) and (287.5,140.81) .. (287.5,139.67) -- cycle ;
%Shape: Ellipse [id:dp47837987673906035] 
\draw  [fill={rgb, 255:red, 0; green, 0; blue, 0 }  ,fill opacity=1 ] (407.5,170) .. controls (407.5,168.86) and (408.62,167.93) .. (410,167.93) .. controls (411.38,167.93) and (412.5,168.86) .. (412.5,170) .. controls (412.5,171.14) and (411.38,172.07) .. (410,172.07) .. controls (408.62,172.07) and (407.5,171.14) .. (407.5,170) -- cycle ;
%Straight Lines [id:da8661251921892181] 
\draw    (410,170) -- (370,140) ;
%Straight Lines [id:da3258974438262572] 
\draw    (370,140) -- (370,180) ;
%Straight Lines [id:da8708986198201851] 
\draw    (410,170) -- (410,210) ;
%Straight Lines [id:da8435088915466151] 
\draw    (410,210) -- (370,180) ;
%Straight Lines [id:da5690180070784738] 
\draw    (450,250) -- (450,100) ;
%Straight Lines [id:da8363560269883472] 
\draw    (290,250) -- (290,139.67) ;
%Straight Lines [id:da19187307203179338] 
\draw    (370,250) -- (369.33,178.67) ;
%Straight Lines [id:da7717825909703906] 
\draw    (410,250) -- (410,210) ;
%Shape: Ellipse [id:dp19458072247405278] 
\draw  [fill={rgb, 255:red, 0; green, 0; blue, 0 }  ,fill opacity=1 ] (367.5,180) .. controls (367.5,178.86) and (368.62,177.93) .. (370,177.93) .. controls (371.38,177.93) and (372.5,178.86) .. (372.5,180) .. controls (372.5,181.14) and (371.38,182.07) .. (370,182.07) .. controls (368.62,182.07) and (367.5,181.14) .. (367.5,180) -- cycle ;
%Shape: Ellipse [id:dp1329062637150733] 
\draw  [fill={rgb, 255:red, 0; green, 0; blue, 0 }  ,fill opacity=1 ] (287.5,250) .. controls (287.5,248.86) and (288.62,247.93) .. (290,247.93) .. controls (291.38,247.93) and (292.5,248.86) .. (292.5,250) .. controls (292.5,251.14) and (291.38,252.07) .. (290,252.07) .. controls (288.62,252.07) and (287.5,251.14) .. (287.5,250) -- cycle ;
%Shape: Ellipse [id:dp6774706749239947] 
\draw  [fill={rgb, 255:red, 0; green, 0; blue, 0 }  ,fill opacity=1 ] (407.5,210) .. controls (407.5,208.86) and (408.62,207.93) .. (410,207.93) .. controls (411.38,207.93) and (412.5,208.86) .. (412.5,210) .. controls (412.5,211.14) and (411.38,212.07) .. (410,212.07) .. controls (408.62,212.07) and (407.5,211.14) .. (407.5,210) -- cycle ;
%Shape: Ellipse [id:dp377037866118451] 
\draw  [fill={rgb, 255:red, 0; green, 0; blue, 0 }  ,fill opacity=1 ] (247.5,250) .. controls (247.5,248.86) and (248.62,247.93) .. (250,247.93) .. controls (251.38,247.93) and (252.5,248.86) .. (252.5,250) .. controls (252.5,251.14) and (251.38,252.07) .. (250,252.07) .. controls (248.62,252.07) and (247.5,251.14) .. (247.5,250) -- cycle ;
%Shape: Ellipse [id:dp21730544539427687] 
\draw  [fill={rgb, 255:red, 0; green, 0; blue, 0 }  ,fill opacity=1 ] (367.5,250) .. controls (367.5,248.86) and (368.62,247.93) .. (370,247.93) .. controls (371.38,247.93) and (372.5,248.86) .. (372.5,250) .. controls (372.5,251.14) and (371.38,252.07) .. (370,252.07) .. controls (368.62,252.07) and (367.5,251.14) .. (367.5,250) -- cycle ;
%Shape: Ellipse [id:dp5383176306794775] 
\draw  [fill={rgb, 255:red, 0; green, 0; blue, 0 }  ,fill opacity=1 ] (407.5,250) .. controls (407.5,248.86) and (408.62,247.93) .. (410,247.93) .. controls (411.38,247.93) and (412.5,248.86) .. (412.5,250) .. controls (412.5,251.14) and (411.38,252.07) .. (410,252.07) .. controls (408.62,252.07) and (407.5,251.14) .. (407.5,250) -- cycle ;
%Shape: Ellipse [id:dp5444786100779515] 
\draw  [fill={rgb, 255:red, 0; green, 0; blue, 0 }  ,fill opacity=1 ] (447.5,250) .. controls (447.5,248.86) and (448.62,247.93) .. (450,247.93) .. controls (451.38,247.93) and (452.5,248.86) .. (452.5,250) .. controls (452.5,251.14) and (451.38,252.07) .. (450,252.07) .. controls (448.62,252.07) and (447.5,251.14) .. (447.5,250) -- cycle ;

% Text Node
\draw (239,251) node [anchor=north west][inner sep=0.75pt]   [align=left] {1};
% Text Node
\draw (359,251) node [anchor=north west][inner sep=0.75pt]   [align=left] {3};
% Text Node
\draw (399,252) node [anchor=north west][inner sep=0.75pt]   [align=left] {4};
% Text Node
\draw (313.2,-1.5) node [anchor=north west][inner sep=0.75pt]   [align=left] {0};
% Text Node
\draw (279,251) node [anchor=north west][inner sep=0.75pt]   [align=left] {2};
% Text Node
\draw (437.08,254.2) node [anchor=north west][inner sep=0.75pt]   [align=left] {5};
% Text Node
\draw (270.82,75.51) node [anchor=north west][inner sep=0.75pt]  [font=\scriptsize,color={rgb, 255:red, 74; green, 144; blue, 226 }  ,opacity=1 ,rotate=-0.32] [align=left] {5};
% Text Node
\draw (308.24,42.51) node [anchor=north west][inner sep=0.75pt]  [font=\scriptsize,color={rgb, 255:red, 74; green, 144; blue, 226 }  ,opacity=1 ,rotate=-0.32] [align=left] {5};
% Text Node
\draw (347.24,13.01) node [anchor=north west][inner sep=0.75pt]  [font=\scriptsize,color={rgb, 255:red, 74; green, 144; blue, 226 }  ,opacity=1 ,rotate=-0.32] [align=left] {5};
% Text Node
\draw (267.57,40.01) node [anchor=north west][inner sep=0.75pt]  [font=\scriptsize,color={rgb, 255:red, 74; green, 144; blue, 226 }  ,opacity=1 ,rotate=-0.32] [align=left] {2};
% Text Node
\draw (300.32,72.76) node [anchor=north west][inner sep=0.75pt]  [font=\scriptsize,color={rgb, 255:red, 74; green, 144; blue, 226 }  ,opacity=1 ,rotate=-0.32] [align=left] {2};
% Text Node
\draw (306.57,108.76) node [anchor=north west][inner sep=0.75pt]  [font=\scriptsize,color={rgb, 255:red, 74; green, 144; blue, 226 }  ,opacity=1 ,rotate=-0.32] [align=left] {2};
% Text Node
\draw (302.07,12.51) node [anchor=north west][inner sep=0.75pt]  [font=\scriptsize,color={rgb, 255:red, 74; green, 144; blue, 226 }  ,opacity=1 ,rotate=-0.32] [align=left] {4};
% Text Node
\draw (342.57,43.76) node [anchor=north west][inner sep=0.75pt]  [font=\scriptsize,color={rgb, 255:red, 74; green, 144; blue, 226 }  ,opacity=1 ,rotate=-0.32] [align=left] {4};
% Text Node
\draw (383.07,73.76) node [anchor=north west][inner sep=0.75pt]  [font=\scriptsize,color={rgb, 255:red, 74; green, 144; blue, 226 }  ,opacity=1 ,rotate=-0.32] [align=left] {4};
% Text Node
\draw (421.32,104.01) node [anchor=north west][inner sep=0.75pt]  [font=\scriptsize,color={rgb, 255:red, 74; green, 144; blue, 226 }  ,opacity=1 ,rotate=-0.32] [align=left] {4};
% Text Node
\draw (388.57,41.76) node [anchor=north west][inner sep=0.75pt]  [font=\scriptsize,color={rgb, 255:red, 74; green, 144; blue, 226 }  ,opacity=1 ,rotate=-0.32] [align=left] {7};
% Text Node
\draw (349.82,74.76) node [anchor=north west][inner sep=0.75pt]  [font=\scriptsize,color={rgb, 255:red, 74; green, 144; blue, 226 }  ,opacity=1 ,rotate=-0.32] [align=left] {7};
% Text Node
\draw (345.82,112.76) node [anchor=north west][inner sep=0.75pt]  [font=\scriptsize,color={rgb, 255:red, 74; green, 144; blue, 226 }  ,opacity=1 ,rotate=-0.32] [align=left] {7};
% Text Node
\draw (427.82,72.51) node [anchor=north west][inner sep=0.75pt]  [font=\scriptsize,color={rgb, 255:red, 74; green, 144; blue, 226 }  ,opacity=1 ,rotate=-0.32] [align=left] {1};
% Text Node
\draw (393.57,103.51) node [anchor=north west][inner sep=0.75pt]  [font=\scriptsize,color={rgb, 255:red, 74; green, 144; blue, 226 }  ,opacity=1 ,rotate=-0.32] [align=left] {1};
% Text Node
\draw (384.57,141.51) node [anchor=north west][inner sep=0.75pt]  [font=\scriptsize,color={rgb, 255:red, 74; green, 144; blue, 226 }  ,opacity=1 ,rotate=-0.32] [align=left] {1};
% Text Node
\draw (279.57,112.01) node [anchor=north west][inner sep=0.75pt]  [font=\scriptsize,color={rgb, 255:red, 74; green, 144; blue, 226 }  ,opacity=1 ,rotate=-0.32] [align=left] {3};
% Text Node
\draw (319.73,89.67) node [anchor=north west][inner sep=0.75pt]  [font=\scriptsize,color={rgb, 255:red, 74; green, 144; blue, 226 }  ,opacity=1 ,rotate=-0.32] [align=left] {3};
% Text Node
\draw (360.23,106.92) node [anchor=north west][inner sep=0.75pt]  [font=\scriptsize,color={rgb, 255:red, 74; green, 144; blue, 226 }  ,opacity=1 ,rotate=-0.32] [align=left] {3};
% Text Node
\draw (398.4,140.17) node [anchor=north west][inner sep=0.75pt]  [font=\scriptsize,color={rgb, 255:red, 74; green, 144; blue, 226 }  ,opacity=1 ,rotate=-0.32] [align=left] {3};
% Text Node
\draw (384.32,180.51) node [anchor=north west][inner sep=0.75pt]  [font=\scriptsize,color={rgb, 255:red, 74; green, 144; blue, 226 }  ,opacity=1 ,rotate=-0.32] [align=left] {1};
% Text Node
\draw (359.07,153.01) node [anchor=north west][inner sep=0.75pt]  [font=\scriptsize,color={rgb, 255:red, 74; green, 144; blue, 226 }  ,opacity=1 ,rotate=-0.32] [align=left] {6};
% Text Node
\draw (401.32,179.01) node [anchor=north west][inner sep=0.75pt]  [font=\scriptsize,color={rgb, 255:red, 74; green, 144; blue, 226 }  ,opacity=1 ,rotate=-0.32] [align=left] {6};
% Text Node
\draw (237.82,163.61) node [anchor=north west][inner sep=0.75pt]  [font=\scriptsize,color={rgb, 255:red, 74; green, 144; blue, 226 }  ,opacity=1 ,rotate=-0.32] [align=left] {2};
% Text Node
\draw (274.07,178.11) node [anchor=north west][inner sep=0.75pt]  [font=\scriptsize,color={rgb, 255:red, 74; green, 144; blue, 226 }  ,opacity=1 ,rotate=-0.32] [align=left] {12};
% Text Node
\draw (358.32,212.86) node [anchor=north west][inner sep=0.75pt]  [font=\scriptsize,color={rgb, 255:red, 74; green, 144; blue, 226 }  ,opacity=1 ,rotate=-0.32] [align=left] {7};
% Text Node
\draw (399.65,222.19) node [anchor=north west][inner sep=0.75pt]  [font=\scriptsize,color={rgb, 255:red, 74; green, 144; blue, 226 }  ,opacity=1 ,rotate=-0.32] [align=left] {3};
% Text Node
\draw (440.57,170.36) node [anchor=north west][inner sep=0.75pt]  [font=\scriptsize,color={rgb, 255:red, 74; green, 144; blue, 226 }  ,opacity=1 ,rotate=-0.32] [align=left] {5};

\end{tikzpicture}

}

\caption{The circular split system from  Example \ref{netdistex} }\label{netdist}

\end{figure}
\begin{exmp} \label{netdistex}
    Consider the dissimilarity map, $\delta$, on the set $[5]$:
    $$\delta=\begin{bmatrix}
0 & 22 & 32 & 29 & 26 \\
 & 0 & 34 & 31& 34 \\
 &  & 0 & 11 & 26 \\
&  &  & 0 & 21 \\
 &  &  &  & 0 \\
\end{bmatrix}.$$ 
This is a distance function for a circular split system which can be seen by calculating \[
\delta(i, j ) + \delta(k, l) \leq \delta(i , k) + \delta(j , l)
\quad \quad \quad 
\delta(i, l) + \delta(j , k) \leq \delta(i, k) + \delta(j , l).
\] for $i,j,k,l \in [5].$ 
For example, for the leaves $1,2,3,4$, we have two inequalities: \[\delta(1, 2 ) + \delta(3, 4)=33\leq 63=\delta(1,3) + \delta(2,4), \]
\[ \delta(1
, 4) + \delta(2 , 3)=63\leq 63=\delta(1,3) + \delta(2,4).\]

Since every set of 4 leaves satisfies the conditions of Theorem \ref{Kalmanson}, $\delta$ is a distance function for a circular split system. 
The corresponding split system appears in Figure \ref{netdist}.

\end{exmp}

We will now define equidistant circular split systems,
which are circular rooted split systems where the distances between the root and all of the leaves are equal. 

\begin{definition}
Let  $N$ be a rooted  circular split system with $n$ leaves, weights $\textbf{a}$, and root labeled $0$. The distance function $\delta_{N,\textbf{a}}$
is \emph{equidistant} if the distance from the root to every leaf is equal. That is, $\delta_{N,\textbf{a}}$ is equidistant if for all $i,j\in [n]$, $\delta_{N,\textbf{a}}(0,i)= \delta_{N,\textbf{a}}(0,j)$.  
\end{definition}

The pair $(N,\delta_{N,\textbf{a}})$ will be called an equidistant circular split system 
if $\delta_{N,\textbf{a}}$
is equidistant. Note that a consequence of this condition is that the distance from every internal vertex to the leaves below it must also be equal.

\begin{exmp} \label{disim}
Consider the rooted split network from Example \ref{dualpolygonex}. 
If weights are applied to the splits of the split system as shown in 
Figure \ref{equidistantfig}, this will be an equidistant circular 
split system where the distance from the root to each of the leaves is $8$. 
The resulting dissimilarity map is
\[
\delta=\begin{bmatrix}
0 & 12 & 16 & 16 & 16\\
 & 0 & 8 & 14& 14\\
 &  & 0 & 14 & 14\\
&  &  & 0 & 4\\
 &  &  &  & 0\\
\end{bmatrix}.\]

\end{exmp}

\begin{figure}[t]
\centering

\tikzset{every picture/.style={line width=0.75pt}} %set default line width to 0.75pt        

\begin{tikzpicture}[x=0.75pt,y=0.75pt,yscale=-1,xscale=1]

%Straight Lines [id:da6615756502289212] 
\draw [color={rgb, 255:red, 208; green, 2; blue, 27 }  ,draw opacity=1 ]   (362.01,77.77) -- (409.02,109.72) ;
%Straight Lines [id:da5455171535310388] 
\draw [color={rgb, 255:red, 208; green, 2; blue, 27 }  ,draw opacity=1 ]   (315,109.72) -- (362.01,141.67) ;
%Straight Lines [id:da5119593897649293] 
\draw    (409.02,109.72) -- (487.37,205.57) ;
%Straight Lines [id:da44097810624807354] 
\draw [color={rgb, 255:red, 80; green, 227; blue, 194 }  ,draw opacity=1 ]   (409.02,109.72) -- (362.01,141.67) ;
%Straight Lines [id:da8558008861610349] 
\draw [color={rgb, 255:red, 80; green, 227; blue, 194 }  ,draw opacity=1 ]   (315,109.72) -- (362.01,77.77) ;
%Straight Lines [id:da6405291594853215] 
\draw [color={rgb, 255:red, 74; green, 144; blue, 226 }  ,draw opacity=1 ]   (362.01,141.67) -- (330.33,171) ;
%Straight Lines [id:da6645985631319158] 
\draw [color={rgb, 255:red, 74; green, 144; blue, 226 }  ,draw opacity=1 ]   (315,109.72) -- (282.99,138.05) ;
%Straight Lines [id:da5051477591255233] 
\draw [color={rgb, 255:red, 208; green, 2; blue, 27 }  ,draw opacity=1 ]   (282.99,138.05) -- (330,170) ;
%Straight Lines [id:da23965312082783186] 
\draw    (330,170) -- (330,201.95) ;
%Straight Lines [id:da847280191693599] 
\draw    (282.99,138.05) -- (282.99,201.95) ;
%Straight Lines [id:da6659636868905108] 
\draw [color={rgb, 255:red, 245; green, 166; blue, 35 }  ,draw opacity=1 ]   (205.3,173.62) -- (315,109.72) ;
%Straight Lines [id:da9478462339574647] 
\draw    (205.3,173.62) -- (236.64,205.57) ;
%Straight Lines [id:da9983657185482322] 
\draw    (158.29,205.57) -- (205.3,173.62) ;
%Shape: Ellipse [id:dp6556568480977927] 
\draw  [fill={rgb, 255:red, 0; green, 0; blue, 0 }  ,fill opacity=1 ] (160.15,205.57) .. controls (160.15,204.54) and (159.32,203.7) .. (158.29,203.7) .. controls (157.26,203.7) and (156.43,204.54) .. (156.43,205.57) .. controls (156.43,206.6) and (157.26,207.44) .. (158.29,207.44) .. controls (159.32,207.44) and (160.15,206.6) .. (160.15,205.57) -- cycle ;
%Shape: Ellipse [id:dp2606319832086348] 
\draw  [fill={rgb, 255:red, 0; green, 0; blue, 0 }  ,fill opacity=1 ] (238.51,205.57) .. controls (238.51,204.54) and (237.67,203.7) .. (236.64,203.7) .. controls (235.61,203.7) and (234.78,204.54) .. (234.78,205.57) .. controls (234.78,206.6) and (235.61,207.44) .. (236.64,207.44) .. controls (237.67,207.44) and (238.51,206.6) .. (238.51,205.57) -- cycle ;
%Shape: Ellipse [id:dp19316926489122088] 
\draw  [fill={rgb, 255:red, 0; green, 0; blue, 0 }  ,fill opacity=1 ] (284.85,201.95) .. controls (284.85,200.92) and (284.02,200.08) .. (282.99,200.08) .. controls (281.96,200.08) and (281.12,200.92) .. (281.12,201.95) .. controls (281.12,202.98) and (281.96,203.82) .. (282.99,203.82) .. controls (284.02,203.82) and (284.85,202.98) .. (284.85,201.95) -- cycle ;
%Shape: Ellipse [id:dp6888078471301033] 
\draw  [fill={rgb, 255:red, 0; green, 0; blue, 0 }  ,fill opacity=1 ] (331.86,201.95) .. controls (331.86,200.92) and (331.03,200.08) .. (330,200.08) .. controls (328.97,200.08) and (328.14,200.92) .. (328.14,201.95) .. controls (328.14,202.98) and (328.97,203.82) .. (330,203.82) .. controls (331.03,203.82) and (331.86,202.98) .. (331.86,201.95) -- cycle ;
%Shape: Ellipse [id:dp49628603067030896] 
\draw  [fill={rgb, 255:red, 0; green, 0; blue, 0 }  ,fill opacity=1 ] (489.24,205.57) .. controls (489.24,204.54) and (488.4,203.7) .. (487.37,203.7) .. controls (486.34,203.7) and (485.51,204.54) .. (485.51,205.57) .. controls (485.51,206.6) and (486.34,207.44) .. (487.37,207.44) .. controls (488.4,207.44) and (489.24,206.6) .. (489.24,205.57) -- cycle ;

% Text Node
\draw (359.62,61.99) node [anchor=north west][inner sep=0.75pt]  [font=\tiny,color={rgb, 255:red, 208; green, 2; blue, 27 }  ,opacity=1 ] [align=left] {0};
% Text Node
\draw (489.6,212.61) node [anchor=north west][inner sep=0.75pt]  [font=\tiny,color={rgb, 255:red, 208; green, 2; blue, 27 }  ,opacity=1 ] [align=left] {1};
% Text Node
\draw (324,210) node [anchor=north west][inner sep=0.75pt]  [font=\tiny,color={rgb, 255:red, 208; green, 2; blue, 27 }  ,opacity=1 ] [align=left] {2};
% Text Node
\draw (279,212) node [anchor=north west][inner sep=0.75pt]  [font=\tiny,color={rgb, 255:red, 208; green, 2; blue, 27 }  ,opacity=1 ] [align=left] {3};
% Text Node
\draw (234.17,213.08) node [anchor=north west][inner sep=0.75pt]  [font=\tiny,color={rgb, 255:red, 208; green, 2; blue, 27 }  ,opacity=1 ] [align=left] {4};
% Text Node
\draw (156.34,213.61) node [anchor=north west][inner sep=0.75pt]  [font=\tiny,color={rgb, 255:red, 208; green, 2; blue, 27 }  ,opacity=1 ] [align=left] {5};
% Text Node
\draw (381.53,76.8) node [anchor=north west][inner sep=0.75pt]   [align=left] {{\scriptsize 2}};
% Text Node
\draw (341.53,113.47) node [anchor=north west][inner sep=0.75pt]   [align=left] {{\scriptsize 2}};
% Text Node
\draw (309,141) node [anchor=north west][inner sep=0.75pt]   [align=left] {{\scriptsize 2}};
% Text Node
\draw (171.87,174.8) node [anchor=north west][inner sep=0.75pt]   [align=left] {{\scriptsize 2}};
% Text Node
\draw (456.53,147.13) node [anchor=north west][inner sep=0.75pt]   [align=left] {{\scriptsize 6}};
% Text Node
\draw (336.53,76.47) node [anchor=north west][inner sep=0.75pt]  [font=\scriptsize] [align=left] {1};
% Text Node
\draw (300.99,120.88) node [anchor=north west][inner sep=0.75pt]  [font=\scriptsize] [align=left] {3};
% Text Node
\draw (341.87,157.47) node [anchor=north west][inner sep=0.75pt]  [font=\scriptsize] [align=left] {3};
% Text Node
\draw (390.53,123.8) node [anchor=north west][inner sep=0.75pt]  [font=\scriptsize] [align=left] {1};
% Text Node
\draw (239,138) node [anchor=north west][inner sep=0.75pt]  [font=\scriptsize] [align=left] {5};
% Text Node
\draw (269,168) node [anchor=north west][inner sep=0.75pt]  [font=\scriptsize] [align=left] {4};
% Text Node
\draw (319,182) node [anchor=north west][inner sep=0.75pt]  [font=\scriptsize] [align=left] {2};
% Text Node
\draw (226.2,178.47) node [anchor=north west][inner sep=0.75pt]  [font=\scriptsize] [align=left] {2};

\end{tikzpicture}

\caption{The equidistant split network from Example \ref{disim} }\label{equidistantfig}
\end{figure}

Now that equidistant split networks have been defined, we can consider the cone of all possible distance functions that make a particular split network equidistant. This cone, which is the main interest of this paper, is the space of equidistant circular split networks and thus, membership in the cone indicates that a set of species has this particular hybridization structure.
\begin{definition}
Let $N$ be a rooted circular split system. Then the \emph{Equidistant Cone of $N$} is
\begin{equation}
EDC_N =  \{  \delta:\delta=\delta_{N,\textbf{a}} \text{ for some } \textbf{a} \text { and } (N,\delta_{N,\textbf{a}}) \text{ is equidistant}  \} \subseteq \mathbb{R}^{n \choose 2}.
\end{equation}
\end{definition}

The use of the word ``cone'' in the definition is made clear from the following 
proposition.

\begin{proposition}
The space $EDC_N$ is a polyhedral cone.
\end{proposition}

\begin{proof}
 Let $N$ be a rooted circular split system on $[n]$ leaves and root $0$. Consider the space 
\[ C_N= \{  \delta:\delta=\delta_{N,\textbf{a}} \text{ for some weight vector } \textbf{a}  \}. \]
Since 
 \[
 \delta_{N,\textbf{a}}(i,j)=\sum_{A|B\in S} a_{A|B} \mathbbm{1}_{A|B}(i,j),
 \] 
each point in $C_N$ is a positive combination of separation indicator functions, $\mathbbm{1}_{A|B}$, and thus $C_N$ is a polyhedral cone. Now $EDC_N$ can be obtained by taking the cone $C_N$, intersecting it with the linear space where for all $i,j\in [n]$, $\delta_{N,\textbf{a}}(0,i)= \delta_{N,\textbf{a}}(0,j)$, and then projecting away the coordinates $\delta(0,i)$. Intersecting with a linear space and then projecting preserves being a polyhedral cone and so $EDC_N$ is a polyhedral cone as well.
\end{proof}

We wish to understand the structure of $EDC_N$ for every circular split network, but first the following specific circular split system is introduced as every other 
circular split system is a subset of it. 
Thus studying the properties of this split system will provide insight into all other circular split systems as well.

\begin{definition}
Let the \emph{complete circular split system}, $KN_n$, on $n$ leaves and root $0$ be the rooted circular split system with every diagonal in the dual polygon representation and every trivial split.  In other words it has splits:
\[
i\dots j | j+1 \dots n 0  \dots i-1
\]
where $i<j\in [n]$ and 
\[
i |0 1 \dots i-1 i+1 \dots n   
\]
where $i\in [n]$.
\end{definition}

The complete circular split system has every possible circular split for a particular circular ordering. The complete circular split system $KN_n$ has $\binom {n+1} 2 -1$ splits, $\binom {n} 2 -1$ of which come from diagonals of the $n+1$-gon and $n$ of which are trivial splits. Note we do not have a trivial split for the root $0$ which is why $KN_n$ has $\binom {n+1} 2 -1$ splits instead of $\binom {n+1} 2$.

\section{Facets of $EDC_{KN_n}$ }\label{sec:facets}

Our goal in this section and the next is to give a complete polyhedral description of the
cone $EDC_{KN_n}$.  Our main result  is describing the facet defining inequalities of 
$EDC_{KN_n}$ shown in Theorem \ref{facets}. In Section \ref{sec:extremerays}, we determine the extreme rays of
$EDC_{KN_n}$.
 All other circular split systems are contained in $KN_n$ and thus  $EDC_{KN_n}$ can be used to characterize $EDC_{N}$ for any circular rooted split system $N$.

\begin{theorem} \label{facets}
Let $EDC_{KN_n}$  be the equidistant cone of $KN_n$. Let $[i,j]= \{i, i+1, \dots, j\}$.  The facets of $EDC_{KN_n}$ are 
 \begin{subequations}
\begin{align*}
\delta(1,i) &\leq \delta(1,i+1)\label{li}\tag{left inequalities} \text{ where } i \in [2,n-1], \\
\delta(i,n)&\leq \delta(i-1,n)\label{ri} \tag{right inequalities} \text{ where } i \in [2,n-1], \\
\delta(i-1,i+1)&\leq \delta(i-1,i) +\delta(i,i+1) \label{ti}\tag{triangle inequalities} \text{ where } i \in [2,n-1], \\
\delta(i,j)+\delta(i-1,j+1) &\leq \delta(i,j+1) +\delta(i-1,j) \label{ci} \tag{covering inequalities} \text{ where } i<j \in [2,n-1]. 
\end{align*}
 \end{subequations}
\end{theorem}

 Note that \ref{ti} are the same as \ref{ci} where $i=j$.
The proof of this theorem will require the following lemmas and definitions.

\begin{definition}
Let $N$ be a split system and $I,J\subseteq [n]$. Let the \emph{set of separating splits for  $I$ and $J$}, $S_{I|J}$, be the set of splits $A|B \in N$ for every $i\in I$ and $j \in J$, $i$ is separated from $j$ in $A|B$.
\end{definition}

This definition is important because   
\[
 \delta(i,j)=\sum_{A|B\in S_{i|j}} a_{A|B}.
 \]

\begin{exmp}
    Consider the split system $KN_6$. In this split system, $S_{2|0,5}$ consists of the following splits:
\[
\{03456|12,013456|2,0456|123,0156|234,056|1234,01456|23 
\}
\]
\end{exmp}

The following lemma provides equations to recover the weights 
$\textbf{a}$ from a distance vector $\delta$ in $EDC_{KN_n}$.

\begin{lemma} \label{singlevar}
Let $KN_n$  be a complete circular network on $n$ leaves and root $0$. Let $\delta=\delta_{KN_n,\textbf{a}}$, then following equalities hold for $i,j\in [2,n-1]$:
\begin{eqnarray*}
    \delta(1,i+1)-\delta(1,i)  & = & 2a_{12\dots i|i+1 \dots n 0},   \\
    \delta(i-1,n)-\delta(i,n)  & = &  2a_{01 \dots i-1 |ii+1\dots n},  \\
    \delta(i-1,i) +\delta(i,i+1)-\delta(i-1,i+1)&  =&  2a_{1\dots (i-1) (i+1) \dots n 0| i}, \\
    \delta(i,j+1) +\delta(i-1,j)-\delta(i,j)-\delta(i-1,j+1) &= & 2a_{1\dots i-1 j+1 \dots n 0| i i+1 \dots j-1 j } .
\end{eqnarray*}
\end{lemma}

Note that these equations give us formulas for every parameter, 
except the trivial splits for the leaves labeled $1$ and $n$. These parameters are not necessary because of the equidistant condition, they can be written in terms of the other weights.

\begin{proof}
We go through the formulas one by one.
First, for the difference $\delta(1,i+1)-\delta(1,i) $, 
since $\delta(1,i+1)=\sum_{A|B\in S_{1|i+1}} a_{A|B}$  and 
$\delta(1,i)= \sum_{A|B\in S_{1|i}} a_{A|B}$, the difference can be rewritten as: 
\[
\delta(1,i+1)-\delta(1,i)= \sum_{A|B\in S_{i+1|i,0}} a_{A|B} + a_{12\dots i|i+1 \dots n 0} - \sum_{C|D\in S_{i|i+1,0,1}} a_{C|D}.\]  
Using the fact that the most recent common ancestor of $i$ and $i+1$ must also be equidistant from $i$ and $i+1$  the following equality must also be true: \\
\[
\sum_{A|B\in S_{i+1|i,0}} a_{A|B} 
 =\sum_{C|D\in S_{i|i+1,0}} a_{C|D}=a_{12\dots i|i+1 \dots n 0} +  \sum_{C|D\in S_{i|i+1,0,1}} a_{C|D}.
 \]
Thus,
  \[ \delta(1,i+1)-\delta(1,i)= 2a_{12\dots i|i+1 \dots n 0} +  \sum_{C|D\in S_{i|i+1,0,1}} a_{C|D} - \sum_{C|D\in S_{i|i+1,0,1}} a_{C|D}=2a_{12\dots i|i+1 \dots n 0}.
  \]  
  
 Using a similar argument and symmetry, we can also derive:
 \[
 \delta(i-1,n)-\delta(i,n)=2a_{01 \dots i-1 |ii+1\dots n}. 
 \]
 
 For the difference  $\delta(i-1,i) +\delta(i,i+1)-\delta(i-1,i+1), $
since $ \delta(i-1,i)= \sum_{A|B\in S_{i-1|i}} a_{A|B}$, $\delta(i,i+1)=\sum_{A|B\in S_{i|i+1}} a_{A|B}$ and  $\delta(i-1,i+1)= \sum_{A|B\in S_{i-1|i+1}} a_{A|B}$, the difference can be rewritten as:
\[
\sum_{A|B\in S_{i-1|i}} a_{A|B}+\sum_{C|D\in S_{i|i+1}} a_{C|D} -\sum_{E|F\in S_{i-1|i+1}} a_{E|F}. 
\]
 Note both of $S_{i-1|i}$ and $S_{i|i+1}$ include the split $1\dots (i-1) (i+1) \dots n 0| i$, which $S_{i-1|i+1}$ does not.   Thus, 
 \[ 
 \delta(i-1,i) +\delta(i,i+1)-\delta(i-1,i+1)=2a_{1\dots (i-1) (i+1) \dots n 0| i}. 
 \]

Lastly, for
$ \delta(i,j+1) +\delta(i-1,j)-\delta(i,j)-\delta(i-1,j+1) $ 
the terms $\delta(i,j)= \sum_{A|B\in S_{i|j}} a_{A|B}$ and $\delta(i-1,j+1)=\sum_{A|B\in S_{i-1|j+1}} a_{A|B}$.  Note that both of $S_{i|j}$ and $S_{i-1|j+1}$  do not include the split  
$1\dots i-1j+1 \dots n 0| ii+1 \dots j-1 j $. Similarly, the terms $\delta(i,j+1)= \sum_{A|B\in S_{i|j+1}} a_{A|B}$ and $\delta(i-1,j)= \sum_{A|B\in S_{i-1|j}} a_{A|B}$. 
Note that both of $S_{i|j+1}$ and $S_{i-1|j}$ do include the split  $1\dots i-1j+1 \dots n 0| ii+1 \dots j-1 j $. Thus the difference becomes:
\[
 \delta(i,j+1) +\delta(i-1,j)-\delta(i,j)-\delta(i-1,j+1)= 2a_{1\dots i-1 j+1 \dots n 0| i i+1 \dots j-1 j }.   \qedhere
 \]
\end{proof}

The following related cone  is parameterized by weights instead of pairwise distances. This cone is introduced because it will be easier to prove the last needed Lemma \ref{weighteq} in this other space.

\begin{definition}
Let $ EDC_{N}$ be the equidistant cone of $N$. For $\delta \in EDC_{N}$, let $\textbf{a}_\delta \in \mathbb{R}^{\binom {n+1} 2 -1}$ be the vector of weights obtained from $\delta$ using the equations in Lemma \ref{singlevar}. Then, \emph{$EDWC_{N}$,  the weighted equidistant cone of $N$ is} :
\[
\{  \textbf{a}:\textbf{a}=\textbf{a}_{\delta} \text{ for some } \delta \in EDC_{N}  \} \subseteq \mathbb{R}^{\binom {n+1} 2 -1}.
\]
\end{definition}
Since the equations in Lemma \ref{singlevar} are a linear map of $EDC_{N}$, $EDWC_{N}$ is a polyhedral cone as well. The distance between leaves can be generalized to distance between arbitrary vertices  in a split network using the previously stated formula \[
 \delta(v,w)=\sum_{A|B\in S_{v|w}} a_{A|B}
 \]
where  $v$ and $w$ are now arbitrary vertices in the split network and  $S_{v|w}$ is the set of spits that are induced by the edge classes that any shortest path between $u$ and $v$ must have. This is the natural generalization of the previous definition of $S_{i|j}$, the set of splits that separates leaves $i$ and $j$.

\begin{lemma}\label{weighteq}
Let $ EDWC_{KN_n}$ be the weighted equidistant cone of $KN_n$. Let $\textbf{a} \in EDWC_{KN_n}$. 
The number of  non-redundant linear equalities between the weights that 
are the entries of the vector $\textbf{a}$ is $n-1$.  That is, $EDWC_{KN_n}$ is contained
in a linear space of codimension $n-1$.
\end{lemma}

\begin{proof}
First, we claim that every equality between the weights $a_{A|B} \in \textbf{a}$ is generated by some 
internal vertex and its distance to the leaves below it. This is because $\textbf{a}\in EDWC_{KN_n}$ and $EDWC_{KN_n}$ is an alternate parametrization of $EDC_{KN_n}$. Therefore, since $EDC_{KN_n}$ is  generated by the Kalmanson condition inequalities, shown in Theorem \ref{Kalmanson}
and the equidistant condition (for all $i,j\in [n]$, $\delta_{N,\textbf{a}}(0,i)= \delta_{N,\textbf{a}}(0,j)$), $EDWC_{KN_n}$ is governed by these same inequalities and equalities. Translating these equations to be among the weights $a_{A|B} \in \textbf{a}$, all of the equalities are generated by one of the equidistant conditions, since the Kalmanson condition itself has no equalities. Thus, since the equidistant condition can be instead thought of as a series of equalities generated by  some 
internal vertex and its distance to the leaves below it, the claim is true. 

Let $v_{i,j}$ be the most recent common ancestor of $i$ and $j$ with $i,j\in [n]$.  
The distance from $v_{i,j}$ to $i$ is 
\[
\sum_{A|B \in S_{i|j,0}} a_{A|B}.
\]
The distance from $v_{i,j}$ to $j$ is 
\[
\sum_{C|D \in S_{j|i,0}} a_{C|D}.
\]
Thus since $\delta(v_{i,j}, i) =  \delta(v_{i,j}, i)$, we have:
\begin{equation}\label{eq:}
    \sum_{A|B \in S_{i|j,0}} a_{A|B}=\sum_{C|D \in S_{j|i,0}} a_{C|D}.
\end{equation}
Consider the equality generated by each $v_{i,i+1}$ for all $i\in [n-1]$. They look like:
\[
\sum_{A|B \in S_{i|i+1,0}} a_{A|B}=\sum_{C|D \in S_{i+1|i,0}} a_{C|D}.
\]
There are $n-1$ of them and no combination of them makes any of the others redundant. 
Ordering these from vertices $v_{1,2}, v_{2,3}, \ldots$, we see that adding a new equality
to the list in order involves at least one new $a_{A|B}$ term that does not appear in prior
equalities.

Thus it just remains to be shown that any  equality   between the weights 
$a_{A|B} \in \textbf{a}$ can be written as a combination of those generated by $v_{i,i+1}$ for $i\in [n-1]$. Consider again the equality generated by $v_{i,j}$ which is:
\[
\sum_{A|B \in S_{i|j,0}} a_{A|B}=\sum_{C|D \in S_{j|i,0}} a_{C|D}.
\]
Then since, \[
S_{i|j,0}= \bigcup_{i'=i}^{i'=j-1} S_{i'|i'+1} \]
 and \[
 S_{j|i,0}= \bigcup_{i'=j}^{i'=i+1} S_{i'|i'-1},
  \]
the equality \[
\sum_{A|B \in S_{i|j,0}} a_{A|B}=\sum_{C|D \in S_{j|i,0}} a_{C|D},
\] is a combination  of the equalities generated by $v_{i',i'+1}$ for $i' \in [i,j-1]$, namely:
\[
\sum_{A|B \in S_{i|j,0}} a_{A|B}=\sum^{j-1}_{i'=i}\sum_{A|B \in S_{i'|i'+1}} a_{A|B}=\sum^{j}_{j'=i+1}\sum_{C|D \in S_{j'|j'-1}} a_{C|D}= \sum_{C|D \in S_{j|i,0}} a_{C|D}.
\]
\end{proof}
%This can also be thought of as the side of the diagonal set partition's side without $0$ corresponds to the least term on the left side of the inequality. For example the diagonal $1234|50$ corresponds with the term $\delta(1,4)$ which is the least term in the inequality $\delta(1,4)\leq \delta(1,5)$ and the diagonal $1250|34$ corresponds with the term $\delta(3,4)$ which is the least term in the inequality $ \delta(2,5)+\delta(3,4)\leq \delta(3,5)+ \delta(2,4)$.

\begin{proof}[Proof of Theorem \ref{facets}]
Let $c'\in \mathbb{R}^{\binom n 2}$ and let
  \begin{align}
    \boldsymbol\delta &= \begin{bmatrix}
           \delta(1,2) \\
          \delta(1,3) \\
           \vdots \\
           \delta(n-1,n)
         \end{bmatrix}\in  EDC_{KN_n} \subseteq \mathbb{R}^{\binom n 2}
  \end{align}
  be a vector of pairwise distances for a distance function $\delta$. Let  $c'\boldsymbol\delta \geq 0$ for all $\boldsymbol\delta\in EDC_{KN_n}$ be an arbitrary valid inequality on $EDC_{KN_n}$. Let $M$ be a $\binom n 2$ by $\binom {n+1} 2 -1$ matrix whose rows are indexed by pairwise distances $\delta(i,j)$ with $i,j\in [n]$ and the columns are indexed by weights $a_{A|B}$ for a split $A|B \in KN_n$. Let the $(\delta(i,j),a_{A|B})$th entry of $M$ be $\mathbbm{1}_{A|B}(i,j)$. The matrix $M$ is a transformation from coordinates in  pairwise distances $\delta(i,j)$ to coordinates in weights $a_{A|B}$ based on the parameterization of $KN_n$. 
  
  Let $\textbf{a} \in EDWC_{KN_n}$. Let $c'M=c$, then  $c\textbf{a}\geq 0$ is a valid inequality for all $ \textbf{a} \in EDWC_{KN_n}$. $ EDWC_{KN_n}$ can be defined by $A\textbf{a}=0$ and $B\textbf{a}\geq 0$ where $A$ is a $n-1$ by ${\binom {n+1} 2 -1}$ matrix whose rows are the coefficients of the non-redundant equidistant equalities of $KN_n$ from Lemma \ref{weighteq} and $B= I_{\binom {n+1} 2 -1}$ which is from the non-negativity condition on the weights. Thus by 
 the Farkas Lemma, there exists a $\mu \in \mathbb{R}^{n-1}$ and a $\lambda \in {\mathbb{R}_{\geq 0}}^{\binom {n+1} 2 -1}$ such that $$c= \mu A  + \lambda B.$$
Consider $ c-\mu A= \lambda B.$ 
 Since $B= I_{\binom {n+1} 2 -1}$, $\lambda B= \lambda \in {\mathbb{R}_{\geq 0}}^{\binom {n+1} 2 -1}$. Thus  $ (c-\mu A) \textbf{a}=\lambda \textbf{a}$ is a non-negative sum of the weights $a_{A|B}$. From Lemma \ref{singlevar} each of the facets in Theorem \ref{facets}  can be reduced to be $2a_{A|B}\geq 0$ for some $A|B\in KN_n$. Thus $(c-\mu A) \textbf{a}$ can be written as a non-negative sum of the facet inequalities. 
 
  Let $N$ be a $\binom {n+1} 2 -1$ by $\binom n 2$ matrix whose rows are indexed by weights $a_{A|B}$ for a split $A|B$ and whose columns are indexed by pairwise distances $\delta(i,j)$ with $i,j\in [n]$.  Let the $(a_{A|B},\delta(i,j))$th entry of $N$ be the coefficient of $\delta(i,j)$ in the facet inequality that can be reduced to $2a_{A|B} \geq 0$ by Lemma \ref{singlevar}. Note that for $a_{2 \dots n 0| 1}$ and $a_{1 2 \dots n-1 0| n}$ there is no facet inequality but $\delta(1,2) =2a_{2 \dots n 0| 1} \geq 0$ and $\delta(n-1,n) =2a_{1 2 \dots n-1 0| n} \geq 0$ can be used. This matrix is a transformation from coordinates in weights $a_{A|B}$ to coordinates in pairwise distances $\delta(i,j)$. Thus $c N = c'$. Now consider $(\mu A N) \boldsymbol\delta $ which reduces to $$\mu A N \boldsymbol\delta= \mu A \textbf{a} = 0. $$ Now consider again $ (c-\mu A) \textbf{a}$. It can be reduced to \[(c-\mu A) \textbf{a}=(c-\mu A) N\boldsymbol\delta= c N\boldsymbol\delta- \mu A N\boldsymbol\delta= c N\boldsymbol\delta= c'\boldsymbol\delta.\] Thus since it was already shown that $(c-\mu A) \textbf{a}$ can be written as a non-negative sum of the facet inequalities $c'\delta$ can be as well.  Since $c'\boldsymbol\delta \geq 0$ was an arbitrary inequality so the proof is done. 
\end{proof}

\begin{exmp}

     Consider the cone $EDC_{KN_5}$. The facets of this cone are: 
     \begin{align*}
         & \text{Left inequalities: }\delta(1,2) \leq \delta(1,3) \leq \delta(1,4) \leq \delta(1,5); \\
          &\text{Right inequalities: }  \delta(4,5) \leq \delta(3,5) \leq \delta(2,5) \leq \delta(1,5);\\
           &\text{Triangle inequalities: }\delta(1,3)\leq \delta(1,2) +\delta(2,3)  ; \; \; \delta(2,4)\leq \delta(2,3) +\delta(3,4)  ;  \\
           & \delta(3,5)\leq \delta(3,4) +\delta(4,5); 
           \quad   \text{Covering inequalities: } 
           \delta(2,3)+\delta(1,4) \leq \delta(1,3) +\delta(2,4)  ; \\ &\delta(3,4)+\delta(2,5) \leq \delta(2,4) +\delta(3,5)   ; 
           \delta(2,4)+\delta(1,5) \leq \delta(1,4) +\delta(2,5)   ;  \\         
     \end{align*}
\end{exmp}

Now, using the characterization of the facets of $EDC_{KN_n}$, the equidistant cone for any circular split system can be described. Every other split system, $N$, is a subset of $KN_n$ and can be obtained by removing the splits not in $N$ from $KN_n$.

\begin{corollary}\label{face}
Let  $N$ be a rooted circular split system on $n$ leaves and root $0$. Then $EDC_N$ is a face of $EDC_{KN_n}$.
Furthermore, every face of $EDC_{KN_n}$ is of the form $EDC_N$ for some subnetwork $N$.
\end{corollary}

\begin{proof}
A way to construct $N$ from $KN_n$  is to start with $KN_n$
and set all weights for splits not in $N$ to $0$. 
However, by Lemma \ref{singlevar}, this means there is equality on all  facets in $KN_n$  that are not  in $N$. Thus we can describe $N$ by using all of the facet inequalities from $KN_n$ but making the facets that can be rewritten using Lemma \ref{singlevar} as $2a_{A|B}\geq 0$ with $A|B\notin N$ into equalities. Then the cone $EDC_N$ will be a face of $EDC_{KN_n}$ as desired.
\end{proof}

\begin{exmp}
    
 Consider the split system $N=\{04|123,012|34,014|23\} \cup \{ i | [4] \setminus i : i \in [4]  \}$. Then $EDC_{N}$ is a face of $EDC_{KN_4}$ and is generated by the following inequities and equalities:
     \begin{align*}
         & \text{Left inequalities: }\delta(1,2) = \delta(1,3) \leq \delta(1,4)  ;\\
          &\text{Right inequalities: }   \delta(3,4) \leq \delta(2,4) = \delta(1,4);\\
           &\text{Triangle inequalities: }\delta(1,3)\leq \delta(1,2) +\delta(2,3)  ; \; \; \delta(2,4)\leq \delta(2,3) +\delta(3,4)  ;   \\
           &\text{Covering inequalities: } 
           \delta(2,3)+\delta(1,4) \leq \delta(1,3) +\delta(2,4)  ; \; \; \\
     \end{align*}
     Note that $\delta(1,2) = \delta(1,3)$ and $\delta(2,4) = \delta(1,4)$ since $ 034|12,01|234 \notin N $, respectively.
\end{exmp}

%%%%%%%%%%%%%%%%%%%%%%%%%%%%%%%%%%%%%%%%%%%%%%
%%%%%%%%%%%%%%%%%%%%%%%%%%%%%%%%%%%%%%%%%%%%%%
%%%%%%%%%%%%%%%%%%%%%%%%%%%%%%%%%%%%%%%%%%%%%%%%%%%%%%%%%%%%%%%%%%%%%%%%%%%%%%%%%%%%%%%%%%%%
%%%%%%%%%%%%%%%%%%%%%%%%%%%%%%%%%%%%%%%%%%%%%%
%%%%%%%%%%%%%%%%%%%%%%%%%%%%%%%%%%%%%%%%%%%%%%

\section{Extreme Rays of $EDC_{KN_n}$ }  \label{sec:extremerays}

The goal of this section is to describe the extreme rays of $EDC_{KN_n}$. 
The extreme rays of $EDC_{KN_n}$ can be described using a particular kind of set partition 
of $[n]$ which will be defined below. As an application,
we also characterize which extreme rays lie on which facets of $EDC_{KN_n}$,
which, given Corollary \ref{face}, will also characterize the facets and extreme 
rays for $EDC_{N}$ for arbitrary networks.

\begin{definition}
Let $\tau= t_1| \dots |t_k$ be a set partition of $[n]$. 
We let $r_{\tau}$ be the vector whose coordinates, $r_{\tau}(i,j)$, are
\begin{equation}
r_{\tau}(i,j) = \begin{cases}
  0  &  \text{ if there exists some } l \in [k] \text{ such that } i,j \in t_l \\
  1 &  \text{ otherwise.}
\end{cases}
\end{equation}
\end{definition}

\begin{definition}
   A \emph{fixed order set partition of $[n]$}, $\tau= t_1| \dots |t_k$, is a set partition such that for all  $l \in [k]$, $t_l= [i,j] = \{i, i+1, \ldots, j\}$ for some $i,j \in [n]$.
Let 
\[
R_n=  \{  r_{\tau}\neq 0: \tau \text{ is a fixed order set partition of } [n]  \}.
\]  
 \end{definition}
 
Note that $R_n$ has $2^{n-1} - 1$ elements since there are $2^{n-1}$ fixed order
set partitions of $n$, and all but one of $r_\tau$ are nonzero.

\begin{exmp}
    The fixed order set partitions of $[5]$ are the following: 
    \begin{align*} &1|2345,12|345,123|45,1234|5,1|2|345,1|23|45,1|234|5, 12|3|45,12|34|5,123|4|5,\\
    &
    1|2|3|45,1|2|34|5,1|23|4|5,12|3|4|5,1|2|3|4|5
    \end{align*} 
    We omit $\tau = 12345$, the partition with one block,  since $r_\tau$ is just the
    zero vector.
The set of rays, $R_5$, corresponding to these set partitions are, respectively:
     \begin{align*}
   &\begin{bmatrix}
1 & 1 & 1 & 1 \\
 & 0 & 0 & 0\\
 &  & 0 & 0\\
&  &  & 0\\
\end{bmatrix}, 
\begin{bmatrix}
0 & 1 & 1 & 1 \\
 & 1 & 1 & 1\\
 &  & 0 & 0\\
&  &  & 0\\
\end{bmatrix}, 
\begin{bmatrix}
0 & 0 & 1 & 1 \\
 & 0 & 1 & 1\\
 &  & 1 & 1\\
&  &  & 0\\
\end{bmatrix}, 
\begin{bmatrix}
0 & 0 & 0 & 1 \\
 & 0 & 0 & 1\\
 &  & 0 & 1\\
&  &  & 1\\
\end{bmatrix}, 
\begin{bmatrix}
1 & 1 & 1 & 1 \\
 & 1 & 1 & 1\\
 &  & 0 & 0\\
&  &  & 0\\
\end{bmatrix},
\end{align*} 

\begin{align*}
&\begin{bmatrix}
1 & 1 & 1 & 1 \\
 & 0 & 1 & 1\\
 &  & 1 & 1\\
&  &  & 0\\
\end{bmatrix}, \begin{bmatrix}
1 & 1 & 1 & 1 \\
 & 0 & 0 & 1\\
 &  & 0 & 1\\
&  &  & 1\\
\end{bmatrix}, \begin{bmatrix}
0 & 1 & 1 & 1 \\
 & 1 & 1 & 1\\
 &  & 1 & 1\\
&  &  & 0\\
\end{bmatrix}, \begin{bmatrix}
0 & 1 & 1 & 1 \\
 & 1 & 1 & 1\\
 &  & 0 & 1\\
&  &  & 1\\
\end{bmatrix}, \begin{bmatrix}
0 & 0 & 1 & 1 \\
 & 0 & 1 & 1\\
 &  & 1 & 1\\
&  &  & 1\\
\end{bmatrix},
\end{align*}

\begin{align*}
&\begin{bmatrix}
1 & 1 & 1 & 1 \\
 & 1 & 1 & 1\\
 &  & 1 & 1\\
&  &  & 0\\
\end{bmatrix}, \begin{bmatrix}
1 & 1 & 1 & 1 \\
 & 1 & 1 & 1\\
 &  & 0 & 1\\
&  &  & 1\\
\end{bmatrix}, \begin{bmatrix}
1 & 1 & 1 & 1 \\
 & 0 & 1 & 1\\
 &  & 1 & 1\\
&  &  & 1\\
\end{bmatrix}, \begin{bmatrix}
0 & 1 & 1 & 1 \\
 & 1 & 1 & 1\\
 &  & 1 & 1\\
&  &  & 1\\
\end{bmatrix}, \begin{bmatrix}
1 & 1 & 1 & 1 \\
 & 1 & 1 & 1\\
 &  & 1 & 1\\
&  &  & 1\\
\end{bmatrix}.\\
     \end{align*}
\end{exmp}

\begin{lemma}\label{part1}
 If  $\tau$ is a fixed order set partition, then  $r_\tau \in EDC_{KN_n}$.
 \end{lemma}
 
\begin{proof}
Let $\tau= t_1| \dots |t_k$ be a fixed order set partition. 
We must show that  $r_\tau$ satisfies all the inequalities from Theorem \ref{facets}. 

For the left inequalities, $\delta(1,i) \leq \delta(1,i+1)$ with $i\in [2,n-1]$ 
consider if there exists $l\in [k]$  such that $1,i\in t_l$. 
If this does exist then $\delta(1,i)=0$ and so the facet will be $0\leq \delta(1,i+1)$ which is always satisfied. If it does not exist then $i+1$ must also not be in the same block as $1$, since $\tau$ is a fixed order set partition, and thus $\delta(1,i)=1$ and $\delta(1,i+1)=1$ so the facet will become $1\leq 1$.
 
The right inequalities, $\delta(i,n)\leq \delta(i-1,n)$, follow by symmetry. 
 
 For the triangle inequalities, $\delta(i-1,i+1)\leq \delta(i-1,i) +\delta(i,i+1)$, consider whether there exists $l\in [k]$  such that $i-1,i+1\in t_l$. If it exists then $\delta(i-1,i+1)=0$. Additionally, it must be that, $i\in t_l$  as well so $\delta(i-1,i)=0$ and $\delta(i,i+1)=0$, and the facet becomes $0 \leq 0$.
 If it does not exist then $\delta(i-1,i+1)=1$ and $\delta(i-1,i) +\delta(i,i+1)$ is $1$ or $2$, depending on if $i$ is separated from both $i-1$ and $i+1$ or not. Thus the facet becomes $1 \leq 1 $ or $1 \leq2$.
 
 For the covering inequalities $\delta(i,j)+\delta(i-1,j+1) \leq \delta(i,j+1) +\delta(i-1,j)$ with  $i,j\in [2,n-1]$, because the $r_\tau$ only have $0$ or $1$ entries, each distance must be either $1$ or $0$. We can break up these distances into just the relevant entries of the vector $r_\tau$. 
 For this argument, it is useful to note that if some entries of $r_\tau$ equal $1$, then certain
 others must equal $1$ also.  So for example, we can write $r_\tau(i-1, j+1) = \max( r_\tau(i-1, i), r_\tau(i,j), r_\tau(j, j+1) )$.  Using this observation, 
 the inequality $\delta(i,j)+\delta(i-1,j+1) \leq \delta(i,j+1) +\delta(i-1,j)$, when applied to $r_\tau$ can be rewritten as
\begin{align*}
r_\tau(i,j) + \max(r_\tau(i,j),  r_\tau(i-1,i),  r_\tau(j,j+1))\\
\leq \max(r_\tau(i,j)  ,  r_\tau(j,j+1)) + \max(r_\tau(i,j) ,  r_\tau(i-1,i) ).
\end{align*}
If $r_\tau(i,j)=1$ then the inequality evaluates to $2 \leq 2$.
If $r_\tau(i,j)=0$ then the  inequality evaluates to 
$  \max(r_\tau(i-1,i), r_\tau(j,j+1)) \leq  r_\tau(i-1,i) + r_\tau(j,j+1)$ which is always satisfied. 
This proves that the rays in $R_n\subseteq EDC_{KN_n}$. 
\end{proof}

Our next step will be to prove that  $EDC_{KN_n} \subseteq \mathrm{cone}(R_n)$.
First,  we will describe a general strategy to show that a given candidate set of
extreme rays actually generates a cone.

\begin{lemma}\label{raymethod}
    Let $C=\{x:  Ax\geq 0 \}$ be a polyhedral cone where 
    \begin{align}
    A &= \begin{pmatrix}
           a_1 \\
           \vdots \\
           a_k
         \end{pmatrix}
         \end{align}
         and let $V=\{v_1, \dots, v_n \}$ be a set of vectors in $\mathbb{R}^k$. 
         A subset  $S\subseteq [k]$  is called valid for $C$ if there exists some 
         $y\in C$ such that $a_iy=0$ for $i \in S$ and $a_jy>0$ for all $ j \notin S$.
         
         If for each valid $S$ there exists $v_l \in V$ such that $a_iv_l=0 \text{ for all } i \in S$, 
         then $C\subseteq \cone(V)$.
\end{lemma}

\begin{proof}
   Suppose that we have an $y$ in $C$, and let $S$
be the associated set of indices. The proof is by induction on the size of $[k] \setminus S$.   If $S = [k]$ there is nothing to show, since
$y$ must be the zero vector in that case. So suppose $S \neq [k]$ and let $v_l$ be the element of $V$ that is guaranteed to exist.  Then we can compute $y' =  y - \lambda v_l$ where $\lambda\geq 0$ is chosen as large as possible
so that $y' \in C$.  Since $\lambda$ is as large as possible, $y'$ must have an associated $S'$, which strictly
contains $S$.  By the inductive hypothesis, we can write $y' \in \cone(V)$, so there
are $\lambda_i \geq 0$ so that $y' = \sum_{i = 1}^n  \lambda_i  v_i$.   But then
$y = \sum_{i = 1}^n  \lambda_i  v_i + \lambda v_l $, so $y \in \cone(V).$
\end{proof}

To use Lemma \ref{raymethod} on the cone $EDC_{KN_n}$, the next step will be to first characterize which $S$ are valid for $EDC_{KN_n}$ and then find an $r_{\tau}$ for each valid $S$. In order to do this, first we prove some useful non-facet inequalities on $EDC_{KN_n}$.

\begin{lemma} \label{otherineq}
    Let $\delta \in EDC_{KN_n}$ and let $\delta(i,j)$ be the $(i,j)$th coordinate of $\delta$. Then for all $i<j<k \in [n]$, $\delta(i,j)\geq 0$, $\delta(i,j)\leq \delta(i,k)$, and $\delta(j,k)\leq \delta(i,k)$.
\end{lemma}

\begin{proof}
    The inequality $\delta(i,j)\geq 0$ is implied from $\delta$ being a dissimilarity map.\\
    
    The inequality $\delta(1,j)\leq \delta(1,k)$ is implied by the facets $\delta(1,l) \leq \delta(1,l+1)$ for all $l\in [j,k-1]$. The inequality $\delta(j,n)\leq \delta(i,n)$ is implied by the facets $\delta(l,n) \leq \delta(l-1,n)$ for all $l\in [i+1,j]$.
    
The inequality $\delta(i,j) \leq \delta(i,k)$ is implied by the non-facet inequality $\delta(i-1,j) \leq \delta(i-1,k)$ and the covering facets:
\[
\delta(i,l)+\delta(i-1,l+1) \leq \delta(i,l+1) +\delta(i-1,l) 
\]
for all $l\in [j,k-1]$. Incrementing $l$ and canceling terms gives $\delta(i,j) \leq \delta(i,k)$ as desired. The non-facet inequality $\delta(i-1,j) \leq \delta(i-1,k)$ can be implied by facet inequalities using the same argument as was just used for  $\delta(i,j) \leq \delta(i,k)$. This will create a recursion of implications whose base case will be $\delta(1,j) \leq \delta(1,k)$ which was shown to be implied by facets above. Using a very similar argument the inequality $\delta(j,k) \leq \delta(i,k)$ is implied by the non-facet inequality $\delta(j,k+1) \leq \delta(i,k+1)$ and the facets:
\[
\delta(l,k)+\delta(l-1,k+1) \leq \delta(l,k+1) +\delta(l-1,k) 
\]
for all $l\in [i+1,j]$.
\end{proof}

Now to characterize which $S$ are valid for $EDC_{KN_n}$, we must see 
if equality on one facet or set of facets implies equality on any other facets. 
In order to clearly see which $S$ are valid a new diagram will be 
introduced which will also provide additional insights into the 
structure of the space $EDC_{KN_n}$. Before defining the diagram itself we will need to define some components used in it.

\begin{definition}
 Let $\delta \in EDC_{KN_n}$  be represented as an $n \times n$  strictly upper triangular matrix. 
 Let $\tilde{\delta}$ be an $n+1 \times n+1$ matrix obtained from $\delta$ where the bottom left 
 $n$ by $n$ block is $\delta$ and first row and last column are all $1$'s. 
 Because of the natural labeling of the rows and columns of $\delta$, 
 the rows of $\tilde{\delta}$ will be indexed from $0$ to $n$ 
 and the columns of $\tilde{\delta}$ will be indexed from $1$ to $n+1$.
 \end{definition}

 The following function is introduced because its possible values exactly correspond with valid $S$ for $EDC_{KN_n}$.
 
 \begin{definition}\label{f}
    Let $ k \in [0,n-2]$, $ l \in [2,n]$  and  $ k< l$  such that if $k=0$, $ l \neq n$.
  The \emph{facet indicator function for $\delta$}, $ f_\delta(k,l)$,
is a Boolean function such that \[
         f_\delta(k,l)=\begin{cases}
  1  &  \text{ if } \tilde{\delta}(k,l)+\tilde{\delta}(k+1,l+1)=\tilde{\delta}(k+1,l)+\tilde{\delta}(k,l+1) \\
  0 &  \text{ otherwise. }  
\end{cases}
\]
 \end{definition}
 
 Note that the facet indicator function is defined on each $2 \times 2$ sub-matrix of $\tilde{\delta}$ such that at least two of the entries are also in the upper triangular portion of $\delta$. 
 Additionally, notice that the facets of $EDC_{KN_n}$ correspond directly with $f_\delta(k,l)$. The submatrix corresponding to $f_\delta(k,l)$ is the $2 \times 2$ matrix of the entries $\tilde{\delta}(k,l),\tilde{\delta}(k+1,l),\tilde{\delta}(k,l+1),\tilde{\delta}(k+1,l+1)$.
 If two of the entries of the submatrix for $ f_\delta(k,l)$ are $1's$, the 
 corresponding facet is a left or right inequality (depending on if the $1's$ are from the first row or last column respectively).
 If one of the entries of the submatrix is $0$, the corresponding facet is a triangle inequality. If all four of the entries of the submatrix are also in $\delta$ then  the corresponding facet is a  covering inequality. Additionally, if $ f_\delta(k,l)=1$ that corresponding facet will be satisfied with equality and if $ f_\delta(k,l)=0$ it will be a strict inequality. Thus, $f_\delta(k,l)$ gives a list of boolean outputs for which facets a particular point of $EDC_{KN_n}$ lies on. Because of this property characterizing which $f$ can arise as possible $f_\delta$ will characterize which $S$ are valid for  $EDC_{KN_n}$.
 
 \begin{definition}
    Let $ k \in [0,n-2]$, $ l \in [2,n]$  and  $ k< l$  such that if $k=0$, $ l \neq n$. The \emph{facet corresponding to (k,l)} is the facet $\tilde{\delta}(k,l)+\tilde{\delta}(k+1,l+1)\geq\tilde{\delta}(k+1,l)+\tilde{\delta}(k,l+1).$
 \end{definition}

 In order to help characterize for which $f$  there exists a $\delta$ such that $f = f_\delta$, meaning that $f$ is 1 exactly when a facet is satisfied by $\delta$ with equality,
 two more related Boolean functions for $\tilde{\delta}$ will be defined.
 
 \begin{definition}\label{g}
 Let $i \in [0,n-1]$, $j \in [n+1]$.
 The \emph{vertical indicator function for $\delta$}, $g_\delta$, is a Boolean function such that \[
         g_\delta(i,j)=\begin{cases}
  1  &  \text{ if } \tilde{\delta}(i,j)=\tilde{\delta}(i+1,j) \\
  0 &  \text{ otherwise. }  
\end{cases}
\]
 Let $i'\in [0,n]$, $j'\in [n]$.
 The \emph{horizontal indicator function for $\delta$}, $h_\delta$, is a  Boolean function such that \[
         h_\delta(i',j')=\begin{cases}
  1  &  \text{ if } \tilde{\delta}(i',j')=\tilde{\delta}(i',j'+1) \\
  0 &  \text{ otherwise.}  
\end{cases}
\]
\end{definition}

Note that the vertical and horizontal indicator functions are defined on each pair of vertically adjacent entries of $\tilde{\delta}$ and each pair of horizontally adjacent entries of $\tilde{\delta}$, respectively.

\begin{definition}
  Let $ \delta \in EDC_{KN_n}$. The grouping $(f_\delta,g_\delta,h_\delta) $  is the \emph{X-diagram} 
  associated to $\delta$.  More generally, a triple of boolean functions $(f,g,h)$ where $f$ has domain $\{(k,l)| k \in [0,n-2]$, $ l \in [2,n]$  and  $ k< l\}$, $g$ has domain $\{(i,j) |i \in [0,n-1]$, $j \in [n+1]\}$, and $h$ has domain $\{(i',j') |i'\in [0,n]$, $j'\in [n]\}$ is called an \emph{X-diagram} of size $n$.
  An X-diagram $(f,g,h)$ is \emph{valid} if there exists a $\delta \in EDC_{KN_n}$
  such that $\delta \in EDC_{KN_n}$ with $(f,g,h) = (f_\delta,g_\delta,h_\delta) $.
 \end{definition}
 
 To draw an X-diagram $(f,g,h)$, take the following steps:
 \begin{enumerate}
     \item Draw a grid graph  with vertices for each entry of $\tilde{\delta}$ (We draw the vertices as either boxes, or $0$'s and $1$'s if the entry of $\tilde{\delta}$ is always that number).
     \item Remove the vertices of $\tilde{\delta}$ from the subdiagonal of  $\delta$ and lower.
     \item Take a subgraph of this grid graph, keeping edges for every two
 adjacent vertices only if the indicator function for corresponding entries of $\tilde{\delta}$ is $1$. Here their indicator function is either horizontal or vertical indicator function depending on if they are horizontally or vertically adjacent.
 \item Consider each $2 \times 2$ submatrix of $\tilde{\delta}$ where $(k,l)$  in the domain of $f(k,l)$, i.e the entries $\tilde{\delta}(k,l),\tilde{\delta}(k+1,l),\tilde{\delta}(k,l+1),$ and $ \tilde{\delta}(k+1,l+1)$. If $f(k,l)=0$ place a dot in the center of the vertices of the grid graph corresponding to the entries of $2 \times 2$ submatrix. If $f(k,l)=1$ place an X.
 \end{enumerate}
  See Figure \ref{xdiagrampic} for an example.
 
\begin{exmp}\label{xdiamex}
 Consider  $ \delta \in EDC_{KN_6}$, such that \[\delta =\begin{bmatrix}
0& 4 & 4 & 8 & 8 & 8 \\
 &0& 1 & 6 & 6& 7\\
 &  &0& 3 & 5& 6\\
&  &  &0& 2& 3\\
&  &  & & 0&2\\
&  &  & & &0\\
\end{bmatrix}.
\]
The $X$ diagram for this $\delta$ can be seen in Figure \ref{xdiagrampic}. Additionally in this $X$-diagram you can see that
  \[
         g_\delta(i,j)=\begin{cases}
  1  &  \text{ if } i\in [0,4] \text{ and } j=7\\
  0 &  \text{ otherwise;}  
\end{cases}
\]
 \begin{align*}
         &h_\delta(0,2)=h_\delta(0,3)= h_\delta(0,4)=h_\delta(0,5)=h_\delta(1,2)=h_\delta(1,4) =h_\delta(1,5)=h_\delta(2,4)=1 \\
         &\text{and }  
  h_\delta(i,j)=0   \text{ otherwise;}\\          &f_\delta(0,2)=f_\delta(0,4)=f_\delta(0,5)=f_\delta(1,4)=f_\delta(2,5)=f_\delta(3,5)=f_\delta(3,4)=1\\ 
&\text{and }  
  f_\delta(i,j)=0   \text{ otherwise.}  
\end{align*}
    
\end{exmp}
 \begin{figure}[t]
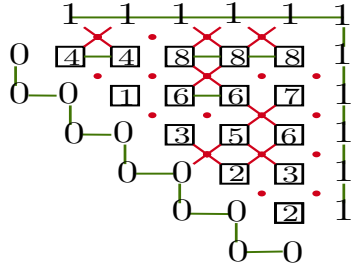


\centering
\resizebox{150pt}{100pt}{

\tikzset{every picture/.style={line width=0.75pt}} %set default line width to 0.75pt        



}
\caption{An example $X$ diagram for a point $\delta$  described in Example \ref{xdiamex}.}\label{xdiagrampic}

\end{figure}

 The following definition will be useful to characterize rays $r_{\tau}$ on an $X$-diagram and in the proof of Lemma \ref{rtmethod}.
 
 \begin{definition} 
 An element of  the form $\tilde{\delta}(i,i+1)$ with 
 $1\leq i\leq n$ is called a \emph{diagonal element}.
 %A \emph{diagonal element} is zero-forced if $e_{ii+1,i+1i+1}=1$ or $e_{ii,ii+1}=1$.
 \end{definition}
 Notice that every $r_{\tau} \in R_n$ can be specified by which the
 diagonal elements are $0$ and $1$. 
 Consider an $\tilde{\delta}$ as matrix for some $X$-diagram where only the diagonal elements of $\tilde{\delta}$ are specified.  The other entries of the sub-matrix $\delta$ will be:
 \[
         \delta(i,j)=\begin{cases}
  0  &  \text{ if } \tilde{\delta}(k,k+1)=0 \text{ for all  } k\in[i,j] \\
  1 &  \text{ otherwise. }  
\end{cases}
\]
 This will give a $\delta$ such that $\delta=r_{\tau}$ for some $ r_{\tau} \in R_n$  unless all of the diagonal entries are $0$.
 
 \begin{exmp}
    Let $ r_{\tau} =\delta \in EDC_{KN_6}$ with diagonal elements ${\delta}(1,2)=0,{\delta}(2,3)=1,{\delta}(3,4)=0,{\delta}(4,5)=0,{\delta}(5,6)=0$. Then the other entries of $\delta$ are \[\delta =\begin{bmatrix}
0&0 & 1 & 1 & 1 & 1 \\
 & 0&1 & 1 & 1& 1\\
 &  & 0&0 & 0& 0\\
&  &  & 0&0& 0\\
&  &  & &0& 0\\
&  &  & && 0\\
\end{bmatrix}.
\]
 \end{exmp}

 %%Note that the facet indicator function $f$ gives an explicit correspondence to sets of 
 %%facets $S$, and in the context of Lemma \ref{rtmethod}, we can work directly with $f$ instead.
 The extra information that comes from $g$ and $h$ in terms of the $X$-diagram will
 help to characterize the allowable $f$'s.  Our next goal is to develop methods to detect
 which triples $(f,g,h)$ are valid $X$-diagrams.  
%%Let $C$ be a cycle of vertices in the grid graph formed in $(\tilde{\delta},f,g,h)$. If all but one of the edges either $g(i,j)$ or $h(i',j')$ in $C$ are $1$, then the remaining edge must also be $1$.  i.e, if there is a set of equality lines that together form a closed loop missing only one line, that line is implied.
 \begin{lemma} \label{equalrules}
 Let $(f,g,h)$ be a valid $X$-diagram.
 \begin{enumerate}
 \item If $f(k,l)=1$ then $g(k,l) = g(k,l+1)$ and  $h(k,l) = h(k+1,l)$.  \label{linemove}
 \item If $g(k,l)= g(k,l+1)=1$  then $f(k,l)=1$.  Similarly, if $h(k,l)= h(k+1,l)=1$ then  $f(k,l)=1$. \label{xforce}
 \item Let  $j> i$. Then if $g(i,j) = 1$, then $g(i, j+1) = 1$.  
Similarly, if $h(i,j) = 1$, then  $h(i-1, j) = 1$. \label{moveup}
 \item  If $g(i,j)=g(i,j+1)=h(i,j)=1$ then $h(i+1,j)=1$. Similarly, if  $g(i,j+1)=h(i,j)=h(i+1,j)=1$ then $g(i,j)=1$\label{equalbox}.
 
 \end{enumerate}
 \end{lemma}
 \begin{proof}
 
 For (\ref{linemove}), if $f(k,l)=1$ then, \[\tilde{\delta}(k,l)+\tilde{\delta}(k+1,l+1)=\tilde{\delta}(k+1,l)+\tilde{\delta}(k,l+1).\] If $g(k,l)=1$  then $\tilde{\delta}(k,l)=\tilde{\delta}(k,l+1)$ which implies that $\tilde{\delta}(k+1,l+1)=\tilde{\delta}(k+1,l)$ which forces that $g(k,l+1)=1.$ Similarly $g(k,l+1)=1$ forces that $g(k,l)=1.$ By the same logic $h(k,l)=1$ forces $h(k+1,l)=1$ and vice versa. This proves (\ref{linemove}).
 \\
 
For (\ref{xforce}), if $g(k,l)=1$ and $g(k,l+1)=1$, then $\tilde{\delta}(k,l)=\tilde{\delta}(k,l+1)$ and $\tilde{\delta}(k+1,l)=\tilde{\delta}(k+1,l+1)$. Thus $\tilde{\delta}(k,l)+\tilde{\delta}(k+1,l+1)=\tilde{\delta}(k+1,l)+\tilde{\delta}(k,l+1)$ so $f(k,l)=1$. Similarly $h(k,l)=1$ and $h(k+1,l)=1$  implies that $\tilde{\delta}(k,l)+\tilde{\delta}(k+1,l+1)=\tilde{\delta}(k+1,l)+\tilde{\delta}(k,l+1)$ so $f(k,l)=1$, as desired.
\\

For (\ref{moveup}), Let  $j> i$, and  $g(i,j) = 1$. Thus $\tilde{\delta}(i,j)=\tilde{\delta}(i+1,j)$. Consider the facet inequality \[\tilde{\delta}(i,j)+\tilde{\delta}(i+1,j+1)\geq\tilde{\delta}(i+1,j)+\tilde{\delta}(i,j+1).\] Since $\tilde{\delta}(i,j)=\tilde{\delta}(i+1,j)$, this inequality reduces to \[\tilde{\delta}(i+1,j+1)\geq\tilde{\delta}(i,j+1).\]  By Lemma \ref{otherineq}, $\tilde{\delta}(i+1,j+1)\leq\tilde{\delta}(i,j+1),$ and thus $\tilde{\delta}(i+1,j+1)=\tilde{\delta}(i,j+1),$ which implies that $g(i, j+1) = 1$. A very similar argument will show that if $h(i,j) = 1$, then  $h(i-1, j) = 1$. 
\\

For (\ref{equalbox}),  if $g(i,j)=g(i,j+1)=h(i,j)=1$ then $\tilde{\delta}(i,j)=\tilde{\delta}(i+1,j)$,  $\tilde{\delta}(i,j+1)=\tilde{\delta}(i+1,j+1)$, and $\tilde{\delta}(i,j)=\tilde{\delta}(i,j+1)$. Thus $\tilde{\delta}(i+1,j)=\tilde{\delta}(i,j)=\tilde{\delta}(i,j+1)=\tilde{\delta}(i+1,j+1)$ so $h(i+1,j)=1$. A similar argument will show that if  $g(i,j+1)=h(i,j)=h(i+1,j)=1$ then $g(i,j)=1$.
 \end{proof}
  Now, Lemma \ref{equalrules} can be used to classify which $S$ are valid for $EDC_{KN_n}$.

\begin{exmp}\label{exwrongS}
   Consider the $X$-diagram on the left in Figure \ref{wrongS}. Using the rules in Lemma \ref{equalrules}, we can see that the underlying $S$ for this $X$-diagram is not valid because $f(3,6)=1$ and $g(3,7)=1$ but $g(3,6)=0$ violating rule $1$. 
   In addition $g(1,3)=1$ and $g(1,4)=1$ but $f(1,3)=0$, violating rule $2$. Similarly,
   $h(3,3)=1$ but $h(2,3)=0$  violating rule $3$. 
   Lastly, $g(1,4)=h(1,4)=h(2,4)=1$ but $g(1,5)=0$ violating rule $4$. The $X$-diagram on the right is the one obtained by following all of the rules in Lemma \ref{equalrules}.
\end{exmp}
\begin{figure}[t]
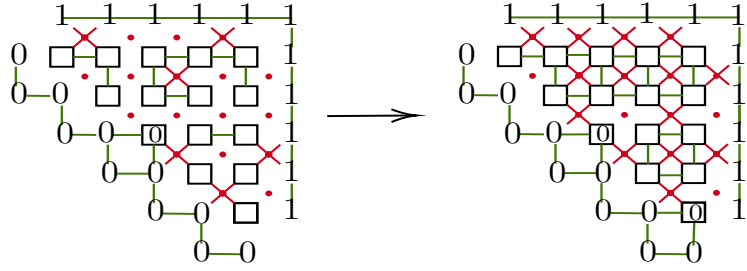


\centering
\resizebox{300pt}{100pt}{

\tikzset{every picture/.style={line width=0.75pt}} %set default line width to 0.75pt        

\tikzset{every picture/.style={line width=0.75pt}} %set default line width to 0.75pt        



}
\caption{The $X$-diagram on the left is not valid as it violates multiple implications as described in Example \ref{exwrongS}. The $X$-diagram on the right is valid and is obtained by starting with the one on the left and following all rules in Lemma \ref{equalrules}.}\label{wrongS}

\end{figure}
The following definition will be used to give the argument to find a $r_{\tau}$ for every valid $S$.
 \begin{definition} Let $\tilde{\delta}$ be the matrix for some  $(\tilde{\delta},f,g,h)$. The \emph{diagonal entries associated to $(k,l)$} are the diagonal elements $\tilde{\delta}(k,k+1)$ and $\tilde{\delta}(l,l+1)$, if they exist. Note that all $(k,l)$ have two diagonal elements associated to them, except for when $k=0$ or $l=n$, the ones corresponding to left and right facets, which only have one associated diagonal entry.
 \end{definition}

 \begin{lemma}\label{rtmethod}
     Let $EDC_{KN_n}=\{x: Bx\geq 0\}$ with \begin{align}
    B &= \begin{bmatrix}
           b_1 \\
           \vdots \\
           b_k
         \end{bmatrix}.
         \end{align}
   For a given valid $S$ for $EDC_{KN_n}$, there exists a $r_{\tau}\in R_n$ such that $b_i r_{\tau}=0$ for all $i\in S$ and $b_i r_{\tau}>0$ for all $i\notin S$.
 \end{lemma}

 \begin{proof}
     First, notice that in a valid X-diagram, $(f,g,h)$, the only way that a $r_{\tau} \in R_n$ does not make $f(k,l)=1$, is when the corresponding $2$ by $2$ submatrix has exactly the entries: \[\begin{pmatrix}
1 & 1\\
0 & 1 
\end{pmatrix}.\]
Thus if for a given valid $S$, if $r_{\tau}$ does not put $\begin{pmatrix}
1 & 1\\
0 & 1 
\end{pmatrix}$ on any $2$ by $2$ sub matrix with upper left corner $(k,l)$ and $f(k,l)=1$ in the $X$-diagram corresponding to $S$, then $b_i r_{\tau}=0$ for all $i\in S$.
To show this is always possible, we will construct $r_{\tau}$ from the $X$-diagram corresponding to $S$ by drawing a line from the left side of $\tilde{\delta}$ to the right side. Everything below the line will be $0$ and everything on or above the line will be $1$. \\

Starting at $(0,0)$, increment the row index to its maximum before incrementing the column index. Let $i_*$ be the index of the first column such that for some $l_*$, $f(l_*,i_*)$ has diagonal elements are that not $0$. This must exist because if all diagonal elements are $0$, the system $S$ will just be equality everywhere, which can be handled by the ray $r_{\tau}=(1,\dots, 1).$ There are two cases, either $l_*=0$ and $f(l_*,i_*)$ has one diagonal element, or $f(l_*,i_*)$ has two diagonal elements. Let  $(k',l')$ be the set of indices such that all $f(k',l')$  have $\tilde{\delta}(i_*,i_*+1)$ as a diagonal element. Any $(k',l')$ with any of its diagonal elements 
 equal to $0$, must have $f(k',l')=1$ by Lemma \ref{equalrules} rules $2$ and $3$. However, if $f(k',l')=1$ for all $(k',l')$ in a row or column then $\tilde{\delta}(i_*,i_*+1)$  would be $0$ using the rule $1$ from Lemma \ref{equalrules}. Thus there must exist a $(l_*,i_*)$ with $f(l_*,i_*)=0$.
%and $f(n,i_*+1)=0$ because all $(k',l')$ with $k'\neq 0$ or $l'\neq n$  have two diagonal elements and $\tilde{\delta}(i_*,i_*+1)$ is the first non $0$ diagonal element or there exists $l_*$ with $f(i_*,l_*)=0$ and both of its diagonal elements are not $0$. In the first case, start at $\tilde{\delta}(0,0)$ and draw a horizontal line until $f(0,i_*)$ and then draw the line vertically until reaching $\tilde{\delta}(i_*,i_*+1)$. In the second case

Draw the line starting from  $\tilde{\delta}(l_*,0)$  until reaching $\tilde{\delta}(l_*,i_*+1)$ and then draw the line vertically until reaching $\tilde{\delta}(i_*,i_*+1)$, which must be non-zero.
 Now consider the row of $(i_*,l')$ who have $\tilde{\delta}(i_*,i_*+1)$ as a diagonal element. One or both of two things are true:
\begin{enumerate}
\item there exists $l_*$ with $f(i_*,l_*)=0$ and both of its diagonal elements are not $0$
\item  $f(i_*,n)=0$.
\end{enumerate}

This is true because it cannot be that $f(i_*,l')=1$ for all $l' \in [2,n]$ or by Lemma \ref{equalrules} rule $1$, $\tilde{\delta}(i_*,i_*+1)$ would $0$, and the only possible $(i_*,l')$ such that $f(i_*,l')\neq 1$ are those with non $0$ diagonal blocks.

If the second case is true, draw the line horizontally to $\tilde{\delta}(i_*,n)$ and the line is finished. If only the first case is true then draw the line to $\tilde{\delta}(i_*,l_*+1)$ and then draw it vertically down to the other non $0$ diagonal block and continue the process. There are only finitely many entries so the line will eventually reach the right side. This line only puts $
$ in its upper corners and the upper corners are only on $(k,l)$ such that $f(k,l)=0$ from the above argument. Thus $b_i r_{\tau}=0$ for all $i\in S$ and $b_i r_{\tau}>0$ for all $i\notin S$ as desired.
 \end{proof}
 \begin{exmp}
     Consider the valid $X$-diagram shown in Figure \ref{xdiagramLine}.  The blue line is created by following the method described in Lemma \ref{rtmethod}. Thus the corresponding extreme ray for this $S$ such that  $b_i r_{\tau}=0$ for all $i\in S$ is \[\delta =%

}
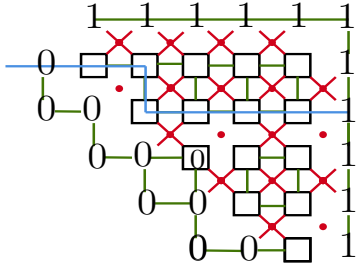
\caption{An example of the method to find an extreme ray for an $S$  such that $b_i r_{\tau}=0$ for all $i\in S$ as described in Lemma  \ref{rtmethod}.}\label{xdiagramLine}
\end{figure}

 \begin{corollary}\label{part2}
     $EDC_{KN_n} \subseteq cone(R_n).$ 
 \end{corollary}
   \begin{proof}
This is immediate from Lemmas \ref{raymethod} and \ref{rtmethod}.
 \end{proof}
 
Now Lemma \ref{part1} and Lemma \ref{part2} prove that $EDC_{KN_n}= \text{cone}(R_n)$ so all that remains to be shown is that all $r_{\tau}$ are the extreme rays.

\begin{lemma}\label{part3}
 For all  $r_\tau \in R_n$ there exists a linear functional that maximizes $r_\tau$ among all elements of $R_n$.
\end{lemma}
 
\begin{proof}
  First, intersect the hyperplane 
  $\sum_{1\leq i<j\leq n} \delta(i,j)=1$
  with the cone $EDC_{KN_n}$ to obtain a polytope $P_{KN_n}$.  
  This hyperplane intersects every ray of $EDC_{KN_n}$ since $EDC_{KN_n}$ is contained in the
  positive orthant.  Furthermore,  $EDC_{KN_n} =  \cone(P_{KN_n})$ so vertices
  of $P_{KN_n}$ correspond to extreme rays of $EDC_{KN_n}$.  Define the point $p_{\tau}$  by
  
  \begin{equation}
p_{\tau}(i,j) = 
\begin{cases}
  0  &  \text{if } i \mbox{ and } j  \mbox{ are in the same blocks in } \tau  \\
  \frac{1}{K}, &  \text{otherwise. }  
\end{cases}
\end{equation}
where $K$ is the number pairs $\{i,j\}$ such that $i$ is separated from $j$ in $\tau$. 
Note that $p_\tau$  is  the point in $P_{KN_n}$ corresponding to $r_\tau$. 
Thus it suffices to find a linear functional that maximizes $p_{\tau}$  in the $P_{KN_n}$ as this will show that $p_{\tau}$ is a vertex of this polytope and thus $r_{\tau}$ will be an extreme ray in $EDC_{KN_n}$.
Furthermore, we only need to check that the linear functional maximizes among all the points $p_\tau$
since $EDC_{KN_n} = \cone(R_n)$, by previous results.

Consider the linear functional:
\begin{equation}
p^*_{\tau}(i,j) = 
\begin{cases}
  1  &  \text{if } p_{\tau} >0  \\
  -1 &  \text{if } p_{\tau}=0 
\end{cases}
\end{equation}
We will see this is maximized at $p_\tau$.  Note that evaluating $p^*_{\tau}$ on  $p_{\tau}$ gives $1$. 
But for any other $p_{\tau'}$ for $\tau'\neq \tau$, the highest it could sum to is $1$ and it can not sum to $1$ because that would mean $p_{\tau'}$ has exactly the same zero and non-zero entries as $p_{\tau}$, contradicting
that $\tau \neq \tau'$.
 \end{proof}

\begin{theorem}\label{thm:extremerays}
The extreme rays of $EDC_{KN_n}$  are $R_n$.
\end{theorem}
\begin{proof}
 This follows from Lemmas \ref{part1}, \ref{part2}, and \ref{part3}.
 \end{proof}

Using this same association we can describe which facets  of $EDC_{KN_n}$ each $r_{\tau}$ lies on.

 \begin{proposition} \label{lieon}
   Let $KN_n$ be the complete rooted circular split network on $n$ leaves. The extreme ray, $r_{\tau}$, is contained in all facets of $KN_n$ except the following:
   \begin{itemize}
    
\item Left inequalities,
$\delta(1,i) \leq \delta(1,i+1)$ where $i$ is separated from $i+1$ and $i$ is not separated from $1$ in $\tau$,
\item Right inequalities,
$\delta(i,n)\leq \delta(i-1,n)$
where  $i$ is separated from $i-1$  and $i$ is not separated from $n$ in $\tau$,
\item Triangle inequalities,
$\delta(i-1,i+1)\leq \delta(i-1,i) +\delta(i,i+1)$ where  $i$ is separated from $i-1$ and $i$ is separated from $i+1$  in $\tau$,
\item Covering  inequalities,
$\delta(i,j)+\delta(i-1,j+1) \leq \delta(i,j+1) +\delta(i-1,j)$ where  $i$ is not separated from $j$ but $i$ is separated from $i-1$ and  $j$ is separated from $j+1$ and in $\tau$. 

 \end{itemize}
 \end{proposition}
   \begin{proof}
  For left inequalities, $\delta(1,i) \leq \delta(1,i+1)$, the only way there cannot be equality is if $\delta(1,i)=0$ and  $\delta(1,i+1)=1$. From the definition of $r_{\tau}$, this happens precisely when $i$ is separated from $i+1$ and $i$ is not separated from $1$ in $\tau$.   The same argument applies by symmetry to the right inequalities.
  
  For the triangle inequalities, $\delta(i-1,i+1)\leq \delta(i-1,i) +\delta(i,i+1)$, if $\delta(i-1,i+1)=0$ that means that $i-1$ and $i+1$ are in the same block in $\tau$. This implies that $\delta(i-1,i)=0$ and $\delta(i,i+1)=0$. Thus the only way to not have equality on this facet is if $\delta(i-1,i+1)=1$ and $\delta(i-1,i) +\delta(i,i+1)=2$. This happens precisely when $i$ is separated from $i-1$ and $i$ is separated from $i+1$  in $\tau$.
  
  For the covering inequalities,
$\delta(i,j)+\delta(i-1,j+1) \leq \delta(i,j+1) +\delta(i-1,j)$, if $\delta(i,j)=1$, that means that $i$ and $j$ are separated in $\tau$. Then $\delta(i-1,j+1)=1$ as well as $\delta(i,j+1) =1$ and $\delta(i-1,j)=1$. Thus the inequality becomes $2\leq 2$. Thus in order for $r_{\tau}$ to not be on the covering inequality $i$ and $j$ must not be separated in $\tau$. So assuming $\delta(i,j)=0$, if $\delta(i-1,j+1)=0$ as well, $i$, $j$, $i-1$, $j+1$ are all in the same block and so $\delta(i,j+1) =0$ and $\delta(i-1,j)=0$ as well, making the inequality $0\leq 0$. Thus the only way to not get equality on this facet is if $\delta(i,j)=0$, $\delta(i-1,j+1)=1$, $\delta(i,j+1)=1$, $\delta(i-1,j)=1$, which is exactly when  $i$ is not separated from $j$ but $i$ is separated from $i-1$ and  $j$ is separated from $j+1$ and   in $\tau$.
\end{proof}

 \begin{corollary} \label{noncompleterays}
 Let $N$ be a circular split network.
The extreme rays of $EDC_{N}$ are the subset of $R_n$ that are contained in $EDC_N$.
 \end{corollary}
 
 \begin{proof}
 Since $EDC_{N}$ is a face of $EDC_{KN_n}$ by Corollary \ref{face}, its extreme rays must be a subset of $R_n$
 contained in the face.
 \end{proof}
 
 \begin{exmp}
Consider $EDC_{KN_5}$ which has extreme rays associated to the following fixed order set partitions:  \begin{align*} &1|2345,12|345,123|45,1234|5,1|2|345,1|23|45,1|234|5, 12|3|45,12|34|5,123|4|5,\\
    &
    1|2|3|45,1|2|34|5,1|23|4|5,12|3|4|5,1|2|3|4|5.
    \end{align*}      
    Using Proposition \ref{lieon} we can characterize which facets of  $EDC_{KN_5}$ each of these rays lies on. For example:

%&1|2345 \text{ lies on all facets except } \delta(2,5) \leq \delta(1,5);\\
 %&12|345 \text{ lies on all facets except } \delta(1,2) \leq \delta(1,3), \delta(3,5) \leq \delta(2,5); \\
 %& 123|45 \text{ lies on all facets except } \delta(1,3) \leq \delta(1,4), \delta(4,5) \leq \delta(3,5);\\
 %&1234|5 \text{ lies on all facets except } \delta(1,4) \leq \delta(1,5);\\
 %&1|2|345 \text{ lies on all facets except } \delta(3,5) \leq \delta(2,5),\delta(1,3) \leq \delta(1,2)+\delta(2,3);\\
%&12|3|45 \text{ lies on all facets except } \delta(1,2) \leq \delta(1,3), \delta(4,5) \leq \delta(3,5),\\ 
% &\delta(2,4) \leq \delta(2,3)+\delta(3,4);\\
 %&12|34|5 \text{ lies on all facets except } \delta(1,2) \leq \delta(1,3), \delta(3,4) +\delta(2,5) \leq \delta(2,4)+\delta(3,5);\\
% &123|4|5 \text{ lies on all facets except } \delta(1,3) \leq \delta(1,4), \delta(3,5) \leq \delta(3,4)+\delta(4,5);\\
 %&1|2|3|45 \text{ lies on all facets except } \delta(4,5) \leq \delta(3,5), \delta(1,3) \leq \delta(1,2)+\delta(2,3),\\
% &\delta(2,4) \leq \delta(2,3)+\delta(3,4);\\
%& 1|23|4|5 \text{ lies on all facets except } \delta(3,5) \leq \delta(3,4)+\delta(4,5),\\
%&\delta(2,3) +\delta(1,4) \leq \delta(1,3)+\delta(2,4);\\
% &12|3|4|5 \text{ lies on all facets except } \delta(1,2) \leq \delta(1,3),\delta(2,4) \leq \delta(2,3)+\delta(3,4),\\
% &\delta(3,5) \leq \delta(3,4)+\delta(4,5);\\
 \begin{align*} &1|2345 \text{ lies on all facets except } \delta(2,5) \leq \delta(1,5),\\
 &1|23|45 \text{ lies on all facets except } \delta(4,5) \leq \delta(3,5),\delta(2,3) +\delta(1,4) \leq \delta(1,3)+\delta(2,4); \\
% &1|234|5 \text{ lies on all facets except } \delta(2,4) +\delta(1,5) \leq \delta(1,4)+\delta(2,5);\\
% &1|2|34|5 \text{ lies on all facets except } \delta(1,3) \leq \delta(1,2)+\delta(2,3),\\
% &\delta(3,4) +\delta(2,5) \leq \delta(2,4)+\delta(3,5);\\
  &1|2|3|4|5 \text{ lies on all facets except } \delta(1,3) \leq \delta(1,2)+\delta(2,3), \delta(2,4) \leq \delta(2,3)+\delta(3,4),\\
  &\delta(3,5) \leq \delta(3,4)+\delta(4,5);\\
    \end{align*}   
 
   \end{exmp}

   \begin{exmp}
       Consider the split system $N=\{01|2345,12|0345,0145|23,0123|45 \}$ together with the trivial splits,
       which are visualized in Figure \ref{dualpolyfig}. 
       This split system is a subset of $KN_5$ where the splits, 
       \[
       \{012|345,05|1234,045|123,015|234,0125|34 \}
       \] 
       are not in $N$.  So $N$ can be obtained from $KN_5$ by setting the weights for those splits equal to $0$. Specifically, this means that $EDC_N$ will lie in the following equalities, corresponding to 
       those splits
       \begin{align*}         
          &\delta(2,5) - \delta(3,5)=2a_{012|345}=0 \\
          & \delta(1,5)-\delta(1,4)=2a_{05|1234}=0; \delta(1,4)-\delta(1,3)=2a_{045|123}=0; \\
           &\delta(1,4) +\delta(2,5) -\delta(2,4)-\delta(1,5)=2a_{015|234}=0; \\
          & \delta(2,4) +\delta(3,5) -\delta(3,4)-\delta(2,5)=2a_{0125|34}=0;\\
     \end{align*}
       
       Using Corollary  \ref{noncompleterays}, $EDC_{N}$ has the extreme rays associated to the following fixed order set partitions:

\begin{align*} &1|2345,1|23|45,12|3|45,
    1|2|3|45,1|23|4|5,12|3|4|5,1|2|3|4|5.
    \end{align*}

       \end{exmp}

%https://arxiv.org/pdf/1607.06978.pdf

%H-J. Bandelt and A. Dress. Split decomposition: a new and useful approach to phylogenetic analysis
%of distance data, Molecular Phylogenetics and Evolution 1 (1992) 242–252.
%\cite{Steel}
%%%%%%%%%%%%%%%%%%%%%%%%%%%%%%%%%%%%%%%%%%%%%%%%%%%%%
%%%%%%%%%%%%%%%%%%%%%%%%%%%%%%%%%%%%%%%%%%%%%%%%%%%%%
%%%%%%%%%%%%%%%%%%%%%%%%%%%%%%%%%%%%%%%%%%%%%%%%%%%%%
%%%%%%%%%%%%%%%%%%%%%%%%%%%%%%%%%%%%%%%%%%%%%%%%%%%%%
%%%%%%%%%%%%%%%%%%%%%%%%%%%%%%%%%%%%%%%%%%%%%%%%%%%%%
As mentioned in the introduction, $EDC_{N}$ can be used to determine if a set of closely related species may have hybridization, using the facet description. While the extreme ray description of this cone is combinatorially nice, we did not see a clear biological interpretation of the extreme rays and this is a potential area of further study.

\section{The Chan-Robbins-Yuen Polytope}

The Chan-Robbins-Yuen Polytope ($CRY_n$) is a face of the Birkhoff polytope,
and appears in other contexts as an example of a flow polytope \cite{Meszaros2015}.
It has generated interest in the combinatorics
community because its normalized volume is a product of Catalan numbers,
specifically:
\[
{\rm Vol}(CRY_n)  =  \prod_{i = 1}^{n-2}  \mathrm{Cat(i)}
\]
where $\mathrm{CAT}(i) = \tfrac{1}{i+1} \binom{2i}{i}$. This polytope was first discussed in \cite{CRY2000}, and subsequently studied by many authors.
In this section, we show a relation between the Chan-Robbins-Yuen polytope and the
cone $EDC_{KN_n}$.

\begin{definition}
    The Chan-Robbins-Yuen Polytope ($CRY_n$) is defined by the following set of equations and inequalities:
    \begin{align*}
CRY_n  &=  \{  x \in \mathbb{R}^{n \times n} :       x_{ij} \geq 0 \mbox{ for all } i,j \in [n], \\
  &  \quad \quad \sum_{i = 1}^n  x_{ij} = 1 \mbox{ for all } j \in [n],  \quad    \sum_{j = 1}^n  x_{ij} = 1 \mbox{ for all } i \in [n],  \mbox{ and }\\
   & \quad \quad  x_{ij} = 0, \mbox{ if } i -j > 1 \}.
    \end{align*}
\end{definition} 

Each element of $CRY_n$ is an $n \times n$, doubly-stochastic matrix, which is zero below the first subdiagonal.
The vertices of $CRY_n$ are the permutation matrices that satisfy the condition that  $x_{ij} = 0, \mbox{ if } i -j > 1$.
Note that there are exactly $2^{n-1}$ such permutation matrices in total.  

To explain the relationship between the Chan-Robbins-Yuen polytope and equidistant network cone, we introduce a related polytope:
\[
PEDC_n  =  EDC_{KN_n}  \cap \{ \delta \in \mathbb{R}^{n(n-1)/2} :   \delta(1,n) \leq 1  \}
\]
Since all of the extreme rays of $EDC_{KN_n}$ have a $1$ in the $\delta(1,n)$ coordinate, this polytope has vertices consisting 
of all the vectors $r_\tau$ for all fixed ordered set partitions of $[n]$, including for the set partition with a single block
$123\cdots n$, which yields the origin.  Note that $PEDC_n$ then also has $2^{n-1}$ vertices.  

\begin{theorem}\label{thm:CRY}
    The polytopes $PEDC_n$ and $CRY_n$ are affinely isomorphic.  The isomorphism preserves the lattice spanned by the vertices in these
    polytopes, hence these polytopes have the same Ehrhart series and normalized volume.
\end{theorem}

To prove Theorem \ref{thm:CRY} we will construct the affine isomorphism explicitly.
Consider the map $\phi: \mathbb{R}^{n \times n}  \rightarrow \mathbb{R}^{n(n-1)/2},$
defined by
\[
\phi_{k,l}(x) =  1 - \sum_{i = 1}^{k} \sum_{j = l }^n  x_{i,j},  \mbox{ for }  1 \leq k < l \leq n.
\]
The inverse map $\psi:  \mathbb{R}^{n(n-1)/2} \rightarrow \mathbb{R}^{n \times n}$ is provided by the following

{\footnotesize
\[
\psi_{k,l}(\delta)  =  \begin{cases}
    1 - \delta(1,n)   &  (k,l) = (1,n)  \\
    \delta(1,j+1) - \delta(1,j) &  (k,l) = (1,j), 2 \leq j \leq n-1  \\
    \delta(i,n) - \delta(i+1, n) &  (k,l) = (i,n), 2 \leq i \leq n-1  \\
    \delta(1,2)   &  (k,l) = (1,1)  \\
    \delta(n-1, n)  &   (k,l) = (n,n)  \\
    \delta(i-1, j) + \delta(i,j+1) - \delta(i, j) - \delta(i-1, j+1)  &  (k,l) = (i,j),  2 \leq i < j < n-1  \\
    \delta(i-1,i) + \delta(i, i+1) - \delta(i-1, i+1)  &  (k,l) = (i,i),  2 \leq i \leq n-1  \\
    1- \delta(i,i+1)  &  (k,l) =  (i+1, i)
\end{cases}
\]
}

\begin{proof}
    We want to show that the inequality system that defines the polytope $PEDC_n$, transforms into the
    inequality system for  $CRY_n$ when applying the map $\phi$.  That is, we consider the inequalities
    satisfied by $\delta \in PEDC_n$, and then make the substitution $\delta(k,l) = \phi_{k,l} (x)$, we should
    get the inequality system that describes $CRY_n$.
    We investigate each type of inequality that defines $PEDC_n$.

    First, we consider the inequality $\delta(1,n) \leq 1.$  Applying $\phi$ we get $\phi_{1,n}(x)  = 1 - x_{1,n} \leq 1$,
    which is equivalent to $x_{1,n} \geq 0$.

    Next we consider the inequalities $\delta(1, l) \leq \delta(1,l+1)$.  Applying $\phi$ we get  
    \[
        1 - \sum_{i = l}^n x_{1, i}  \leq  1 - \sum_{i = l+1}^n x_{1, i}
    \]
    which is equivalent to $x_{1,l} \geq 0.$  Similarly, the inequalities $\delta(k+1, n) \leq \delta(k,n)$ yields the inequality $x_{k,n}\geq 0$.

    For the covering inequalities $\delta(i,j) + \delta(i-1, j+1)  \leq  \delta(i-1, j) + \delta(i, j+1)$, substitution and cancellation yields that this
    is equivalent to $x_{i,j} \geq 0$, using a simple inclusion and exclusion argument.  

    The most difficult to analyze are the triangle inequalities, $\delta(l-1, l+1) \leq  \delta(l-1, l) + \delta(l, l+1)$. 
    After making the substitution from $\phi$, we get the inequality
    \[
        \sum_{i = 1}^l  \sum_{j = l}^n x_{i,j}  \leq  1 + x_{l,l}
    \]
    We claim that this is equivalent to the inequality $x_{l,l} \geq 0$.  To prove this, we use the double stochastic feature
    of the polytope $CRY_n$.  In particular, we have the $\sum_{i = 1}^{j+1}  x_{i,j} = 1$ and  $\sum_{j = i-1}^n x_{i,j} = 1$.
    In particular, the sum of the first $l$ rows of $x$ equals $l$, while the sum of the first $l-1$ columns is $l-1$.  Since
    the first $l-1$ columns are all zero below the subdiagonal, the difference between these two sums, on the one hand, is equal to $1$, and on the 
    other hand, is equal  to $\sum_{i = 1}^l  \sum_{j = l}^n x_{i,j} $.  Thus we get, $1 \leq 1 + x_{l,l}$, or $x_{l,l} \geq 0$.

    At this point, we have shown that the inequalities that define $PEDC_n$ become the inequalities, $x_{i,j} \geq 0$ for
    all $1 \leq j < k \leq n$, and the inequalities $x_{i,i} \geq 0$ for $i = 2, \ldots, n-1$.  
    This shows that $\phi(CRY_n) \subseteq PEDC_n$.

    Next, we must verify that $\psi$ is the inverse map of $\phi$ and that $\psi(PEDC_n) \subseteq CRY_n$.
    The fact that $\psi$ is the inverse map can be checked directly by applying it coordinate by coordinate, using
    the different formula for $\psi_{k,l}$.  Most of this follows from the argument above.  For instance, the proof
    that the covering inequality turns into the inequality $x_{i,j} \geq 0$, shows that $(\psi \circ \phi)_{i,j} (x)  = x_{i,j}$
    when $1< i < j < n-1$.  A similar approach follows for the other coordinates.
\end{proof}

It is worth noting how the maps $\phi$ and $\psi$ transform vertices of the two polytopes.
The map is easiest to see on $\phi:  CRY_n \rightarrow PEDC_n$.
Each permutation matrix that is in $CRY_n$ has the form of a block diagonal matrix, where each diagonal block is
either a $1 \times 1$ block of a $1$, or is a $k \times k$ block that has $1$'s on the subdiagonal and a $1$ in the upper
right corner.  The map $\phi$ sends this permutation matrix to the $0/1$ upper triangular array obtained by putting ones
precisely in the part above and to the right of all the blocks of permutation matrices.

Here is an example of such a permutation matrix in $CRY_9$ and the vector $r_\tau$ that it maps to in $PEDC_9$ under $\phi$:
{\footnotesize\[
\begin{pmatrix}
    1 &   & &  & & & & &      \\
      &  0 & 0 & 1 & & & & &  \\
      &  1 & 0 & 0 & & & & &  \\
      &  0 & 1 & 0 & & & & &  \\
      &  & & &  0 & 1 & & &  \\
      & & & &   1 & 0 & & &   \\
      & & & &  &   &  1 & &  \\
      &  & & & &   &   &  0 & 1  \\
      & & & & &   &   &  1 &  0
\end{pmatrix}   \mapsto  
\begin{pmatrix}
   & 1 & 1 & 1 & 1 & 1 & 1 & 1 & 1 \\
   &   & 0 & 0 & 1 & 1 & 1 & 1 & 1  \\
   &   &   & 0 & 1 & 1 & 1 & 1 & 1  \\
   &   &    &  & 1 & 1 & 1 & 1 & 1  \\
    &   &    &  &  & 0 & 1 & 1 & 1  \\
    &   &    &  &  &  & 1 & 1 & 1  \\    
    &   &    &  &  &  &   & 1 & 1  \\   
    &   &    &  &  &  &   &  &  0  \\  
    &   &    &  &  &  &   &  &    \\  
\end{pmatrix}.
\]
}

\bibliographystyle{plain}
\bibliography{Equidistance_Networks_2.bib}

\begin{thebibliography}{10}

\bibitem{BanDress}
Hans-Jürgen Bandelt and Andreas~W.M. Dress.
\newblock Split decomposition: A new and useful approach to phylogenetic analysis of distance data.
\newblock {\em Molecular Phylogenetics and Evolution}, 1(3):242--252, 1992.

\bibitem{Bryant2004}
David Bryant and Vincent Moulton.
\newblock {Neighbor-Net: An Agglomerative Method for the Construction of Phylogenetic Networks}.
\newblock {\em Molecular Biology and Evolution}, 21(2):255--265, 02 2004.

\bibitem{CRY2000}
Clara~S. Chan, David~P. Robbins, and David~S. Yuen.
\newblock On the volume of a certain polytope.
\newblock {\em Experiment. Math.}, 9(1):91--99, 2000.

\bibitem{Kalmanson}
Victor Chepoi and Bernard Fichet.
\newblock A note on circular decomposable metrics.
\newblock {\em Geom. Dedicata}, 69(3):237--240, 1998.

\bibitem{main}
Satyan~L. Devadoss and Samantha Petti.
\newblock A space of phylogenetic networks.
\newblock {\em SIAM J. Appl. Algebra Geom.}, 1(1):683--705, 2017.

\bibitem{Deza2010}
Michel~Marie Deza and Monique Laurent.
\newblock {\em Geometry of cuts and metrics}, volume~15 of {\em Algorithms and Combinatorics}.
\newblock Springer, Heidelberg, 2010.
\newblock First softcover printing of the 1997 original [MR1460488].

\bibitem{netalg}
Andreas Dress and Daniel Huson.
\newblock Constructing splits graphs.
\newblock {\em IEEE Trans. Comput. Biol. Bioinformatics}, 1:1109--1115, 01 2004.

\bibitem{Felsenstein2003}
Joseph Felsenstein.
\newblock {\em Inferring Phylogenies}.
\newblock Sinauer, 2003.

\bibitem{Huson_Rupp_Scornavacca_2010}
Daniel~H. Huson, Regula Rupp, and Celine Scornavacca.
\newblock {\em Phylogenetic Networks: Concepts, Algorithms and Applications}.
\newblock Cambridge University Press, 2010.

\bibitem{Kalmanson1975}
Kenneth Kalmanson.
\newblock Edgeconvex circuits and the traveling salesman problem.
\newblock {\em Canadian J. Math.}, 27(5):1000--1010, 1975.

\bibitem{Network2022}
Sungsik Kong, Joan Pons, Laura Kubatko, and Kristina Wicke.
\newblock Classes of explicit phylogenetic networks and their biological and mathematical significance.
\newblock {\em Journal of Mathematical Biology}, 84, 05 2022.

\bibitem{LevyPachter2011}
Dan Levy and Lior Pachter.
\newblock The neighbor-net algorithm.
\newblock {\em Adv. in Appl. Math.}, 47(2):240--258, 2011.

\bibitem{Meszaros2015}
Karola M\'{e}sz\'{a}ros and Alejandro~H. Morales.
\newblock Flow polytopes of signed graphs and the {K}ostant partition function.
\newblock {\em Int. Math. Res. Not. IMRN}, 2015(3):830--871, 2013.

\bibitem{Michener57}
Charles~D. Michener and Robert~R. Sokal.
\newblock A quantitative approach to a problem in classification.
\newblock {\em Evolution}, 11(2):130--162, 1957.

\bibitem{phy}
Charles Semple and Mike Steel.
\newblock {\em Phylogenetics}.
\newblock Oxford lecture series in mathematics and its applications. Oxford University Press, 2003.

\bibitem{Steel}
Mike Steel.
\newblock {\em Phylogeny: Discrete and Random Processes in Evolution}.
\newblock CBMS-NSF Regional Conference Series in Applied Mathematics. Society for Industrial and Applied Mathematics, 2016.

\end{thebibliography}

\end{document}